\documentclass[11pt,a4paper,oneside,reqno,final]{amsart}

\usepackage{amsmath,amssymb,amsthm}
\usepackage{fixme}
\usepackage{mathrsfs}
\usepackage[english]{babel}

\usepackage{hyperref}
\usepackage{bookmark}
\hypersetup{                                                         
    pdfauthor={Kevin Aguyar Brix and Toke Meier Carlsen},
    pdftitle={C*-algebras, groupoid and covers of shift spaces},
      colorlinks=true,
    linkcolor={blue},
      citecolor={red!50!black},
      urlcolor={blue!80!black},
      linktocpage=true                                                      
  }

\newcommand{\N}{\mathbb{N}}
\newcommand{\Z}{\mathbb{Z}}
\newcommand{\T}{\mathbb{T}} 
\newcommand{\K}{\mathbb{K}} 

\newcommand{\I}{\mathcal{I}}
\newcommand{\G}{\mathcal{G}} 
\newcommand{\R}{\mathcal{R}} 
\newcommand{\D}{\mathcal{D}} 
 
\newcommand{\F}{\mathcal{F}} 
\newcommand{\OO}{\mathcal{O}}

\newcommand{\X}{\mathsf{X}} 
\newcommand{\Y}{\mathsf{Y}} 
\newcommand{\LL}{\mathsf{L}} 
\newcommand{\A}{\mathfrak{A}} 
\newcommand{\LRA}{\longrightarrow}
\newcommand{\LMT}{\longmapsto}

\DeclareMathOperator{\Iso}{\textup{Iso}}
\DeclareMathOperator{\lp}{lp} 
\DeclareMathOperator{\id}{id} 
\DeclareMathOperator{\supp}{supp} 
 
\DeclareMathOperator{\Stab}{Stab}

\let\tilde\widetilde{}
\let\leq\leqslant{}     
\let\geq\geqslant{}
\let\subset\subseteq{}      
{}
\newcommand{\widesim}[2][1.5]{\mathrel{\overset{#2}{\scalebox{#1}[1]{$\sim$}}}}

\usepackage{tikz-cd}
\usetikzlibrary{matrix,arrows,decorations.markings}
\usepgflibrary{arrows.meta}
\usepackage{float}

\newtheorem{lemma}{Lemma}[section] 
\newtheorem{corollary}[lemma]{Corollary}
\newtheorem{theorem}[lemma]{Theorem}
\newtheorem{proposition}[lemma]{Proposition}

\newtheorem*{theorem*}{Theorem}
\theoremstyle{definition}
\newtheorem{definition}[lemma]{Definition}
\newtheorem{example}[lemma]{Example}
\newtheorem{remark}[lemma]{Remark}

\numberwithin{equation}{section} 
\begin{document}

\begin{abstract}
    To every one-sided shift space $\X$ we associate a cover $\tilde{\X}$, a groupoid $\G_\X$ and a $\mathrm{C^*}$-algebra $\OO_\X$.
    We characterize one-sided conjugacy, eventual conjugacy and (stabilizer-preserving) continuous orbit equivalence between $\X$ and $\Y$
    in terms of isomorphism of $\G_\X$ and $\G_\Y$, 
    and diagonal-preserving $^*$-isomorphism of $\OO_\X$ and $\OO_\Y$.
    We also characterize two-sided conjugacy and flow equivalence of the associated two-sided shift spaces $\Lambda_\X$ and $\Lambda_\Y$
    in terms of isomorphism of the stabilized groupoids $\G_\X\times \R$ and $\G_\Y\times \R$,
    and diagonal-preserving $^*$-isomorphism of the stabilized $\mathrm{C^*}$-algebras $\OO_\X\otimes \K$ and $\OO_\Y\otimes \K$.
    Our strategy is to lift relations on the shift spaces to similar relations on the covers.

    Restricting to the class of sofic shifts whose groupoids are effective,
    we show that it is possible to recover the continuous orbit equivalence class of $\X$ from the pair $(\OO_\X, C(\X))$,
    and the flow equivalence class of $\Lambda_\X$ from the pair $(\OO_\X\otimes \K, C(\X)\otimes c_0)$.
    In particular, continuous orbit equivalence implies flow equivalence for this class of shift spaces.
\end{abstract}

\title{$\mathrm{C^*}$-algebras, groupoids and covers of shift spaces}

\author[K.A.~Brix]{Kevin Aguyar Brix}
\address[K.A.~Brix]{School of Mathematics and Applied Statistics, 
University of Wollongong, Wollongong NSW 2522, Australia}
\email{kabrix.math@fastmail.com}

\author[T.M.~Carlsen]{Toke Meier Carlsen}
\address[T.M.~Carlsen]{Department of Science and Technology\\University of the Faroe Islands\\
Vestara Bryggja 15\\ 
FO-100 T\'orshavn\\
the Faroe Islands}
\email{toke.carlsen@gmail.com}

\keywords{Symbolic dynamics, shift spaces, groupoids, $\mathrm C^*$-algebras}
\thanks{The first named author is supported by the Danish National Research Foundation through the Centre for Symmetry and Deformation (DNRF92)
and the Carlsberg Foundation through an Internationalisation Fellowship,
while the second named author was supported by Research Council Faroe Islands through the project ``Using graph $\mathrm{C^*}$-algebras to classify graph groupoids''.}
\maketitle

\setcounter{tocdepth}{1}
\tableofcontents


\section*{Introduction}
In~\cite{CK80}, Cuntz and Krieger used finite type symbolic dynamical systems to construct a family of simple $\mathrm{C^*}$-algebras
today known as \emph{Cuntz--Krieger algebras}.
Such a dynamical system is up to conjugacy determined by a finite square $\{0,1\}$-matrix $A$,
and the $\mathrm{C^*}$-algebra $\OO_A$ comes equipped with a distinguished commutative subalgebra $\D_A$ called the \emph{diagonal}
and a circle action $\gamma\colon \T\curvearrowright \OO_A$ called the \emph{gauge action}.
This construction has allowed for new and fruitful discoveries in both symbolic dynamics and in operator algebras via translations of interesting problems and results.

One of the most important relations among two-sided subshifts besides conjugacy is flow equivalence.
Cuntz and Krieger showed that if the subshifts $\Lambda_A$ and $\Lambda_B$,
determined by irreducible matrices which are not permutations $A$ and $B$, are flow equivalent,
then there is a $^*$-isomorphism between the stabilized Cuntz--Krieger algebras $\OO_A\otimes \K \LRA \OO_B\otimes \K$
which maps $\D_A\otimes c_0$ onto $\D_B\otimes c_0$.
Here, $\K$ is the $\mathrm{C^*}$-algebra of compact operators on separable Hilbert space and $c_0$ is the maximal abelian subalgebra of diagonal operators.
The stabilized Cuntz--Krieger algebras together with their diagonal subalgebra therefore constitute an invariant of flow equivalence.
However, 
\[
    A = 
    \begin{pmatrix}
        1 & 1 \\
        1 & 1
    \end{pmatrix}, \qquad
    A' =
    \begin{pmatrix}
        1 & 1 & 0 & 0 \\
        1 & 1 & 1 & 0 \\
        0 & 1 & 1 & 1 \\
        0 & 0 & 1 & 1
    \end{pmatrix}
\]
are examples of irreducible and nonpermutation matrices which are not flow equivalent but whose Cuntz--Krieger algebras $\OO_A$ and $\OO_{A'}$ are $^*$-isomorphic, 
cf.~\cite[Lemma 6.4]{Rordam1995}.
This raised the question:
Is it possible to characterize flow equivalence in terms of the associated $\mathrm{C^*}$-algebras?

In the striking paper~\cite{MM14}, Matsumoto and Matui employ topological groupoids to answer this question:
Using Renault's groupoid reconstruction theory~\cite{Renault08} (which is based on work of Kumjian~\cite{Kumjian1986})
they prove that $\Lambda_A$ and $\Lambda_B$ (determined by irreducible and nonpermutation $\{0,1\}$-matrices $A$ and $B$) are flow equivalent
if and only if there is a $^*$-isomorphism $\Phi\colon \OO_A\otimes \K\LRA \OO_B\otimes \K$ satisfying $\Phi(\D_A\otimes c_0) = \D_B\otimes c_0$, cf.~\cite[Corollary 3.8]{MM14}.
In the particular case above, it follows that no $^*$-isomorphism $\OO_A\otimes \K \LRA \OO_{A'}\otimes \K$ will map $\D_A\otimes c_0$ onto $\D_{A'}\otimes c_0$.

In~\cite{Mat2010} (see also~\cite{Mat2010b}), Matsumoto introduces the notion of continuous orbit equivalence.
He proves that one-sided shift spaces $\X_A$ and $\X_B$ (determined by irreducible and nonpermutation $\{0,1\}$-matrices $A$ and $B$) 
are continuously orbit equivalent if and only if
there is a $^*$-isomorphism $\OO_A\LRA \OO_B$ which carries $\D_A$ onto $\D_B$.
For this reason, Matsumoto remarks that continuous orbit equivalence is a one-sided analog of flow equivalence.
These results on flow equivalence and continuous orbit equivalence are generalized to include all finite type shifts in~\cite[Corollaries 7.1 and 7.2]{CEOR}.

In the more general setting of directed graphs, the second-named author and Rout used groupoids to show that $\X_A$ and $\X_B$ 
(for $\{0,1\}$-matrices $A$ and $B$ with no zero rows and no zero columns) 
are one-sided eventually conjugate if and only if there is $^*$-isomorphism $\Phi\colon \OO_A\LRA \OO_B$ 
satisfying $\Phi(\D_A) = \D_B$ and $\Phi\circ \gamma^A = \gamma^B\circ \Phi$~\cite[Corollary 4.2]{Carlsen-Rout}.
Furthermore, they show that $\Lambda_A$ and $\Lambda_B$ are conjugate if and only if there is a $^*$-isomorphism $\Phi\colon \OO_A\otimes \K\LRA \OO_B\otimes \K$
satisfying $\Phi(\D_A\otimes c_0) = \D_B\otimes c_0$ and $\Phi\circ (\gamma^A\otimes \id) = (\gamma^B\otimes \id)\circ \Phi$~\cite[Corollary 5.2]{Carlsen-Rout}.
From this we understand that one-sided eventual conjugacy is a one-sided analog of two-sided conjugacy.
In a similar spirit, one-sided conjugacy for shifts of finite type was characterized using groupoids in terms of the Cuntz--Krieger algebra with its diagonal
and a certain completely positive map by the authors~\cite[Theorem 3.3]{BC}.
Orbit equivalence of general directed graphs were studied in~\cite{BCW2017, Arklint-Eilers-Ruiz2018, Carlsen-Winger}.

The aim of this paper is to study general shift spaces and provide similar characterizations in terms of groupoids and $\mathrm{C^*}$-algebras.
When $\X$ is a shift space which is not of finite type then the shift operation $\sigma_\X$ is not a local homeomorphism~\cite[Theorem 1]{Parry1966}
so $(\X, \sigma_\X)$ is not a Deaconu--Renault system (in the sense of~\cite[Section 8]{Sims-Williams}).
The Deaconu--Renault groupoid naturally associated to $\X$ then fails to be \'etale.
Therefore, a na\"{i}ve strategy to generalize Cuntz and Krieger's results does not work here.
The bulk of the work is therefore to circumvent this problem.

Matsumoto is the first to associate a $\mathrm{C^*}$-algebra to a general two-sided subshift 
and study its properties, see~\cite{Mat1997, Mat1998, Mat1999, Mat1999b, Mat1999c}.
Unfortunately, there was a mistake in one of the foundational results.
The second-named author and Matsumoto~\cite{Carlsen-Matsumoto2004} then provided a new construction  
which is in general not $^*$-isomorphic to Matsumoto's algebra.
This new construction lacks a universal property and therefore has the downside of not always admitting a gauge action.
The second-named author finally introduced a $\mathrm{C^*}$-algebra $\OO_\X$ associated to a general \emph{one-sided} shift space $\X$ 
using a Cuntz--Pimsner construction~\cite{Carlsen2008} which satisfies Matsumoto's results and admits a gauge action.
We refer the reader to~\cite{Carlsen-Matsumoto2004, Carlsen-Silvestrov2007, Dokuchaev-Exel2017}
for a more detailed description of the history of associating a $\mathrm{C^*}$-algebra to general subshifts.

The $\mathrm{C^*}$-algebra $\OO_\X$ has appeared in various guises throughout the literature.
In~\cite[Section 4.3]{Thomsen2010} (see also~\cite[Corollary 6.7]{Carlsen-Thomsen2012}), Thomsen realized it as a groupoid $\mathrm{C^*}$-algebra of a semi-\'etale groupoid,
Carlsen and Silvestrov describe it as one of Exel's crossed products~\cite[Theorem 10]{Carlsen-Silvestrov2007},
while Dokuchaev and Exel use partial actions~\cite[Theorem 9.5]{Dokuchaev-Exel2017}.
Matsumoto then took a slightly different approach and considered certain labeled Bratteli diagrams called $\lambda$-graph systems
and associated to each $\lambda$-graph system $\mathfrak{L}$ a $\mathrm{C^*}$-algebra $\OO_{\mathfrak{L}}$~\cite{Mat1999-lambda, Mat2002, Mat2006}.
Any two-sided subshift $\Lambda$ has a canonical $\lambda$-graph system $\mathfrak{L}^\Lambda$
and the spectrum of the diagonal subalgebra of $\OO_{\mathfrak{L}}$ is (homeomorphic to) the $\lambda$-graph $\mathfrak{L}^\Lambda$.
Matsumoto then studied orbit equivalence, eventual conjugacy and two-sided conjugacy of these $\lambda$-graphs
and how they are reflected in the $\mathrm{C^*}$-algebras~\cite{Mat2010b, Mat2019}.
Recently, Exel and Steinberg have further investigated semigroups of shift spaces and shown that there is a universal groupoid 
which can be suitably restricted to model either Matsumoto's $\mathrm{C^*}$-algebras or $\OO_\X$,~\cite[Theorem 10.3]{Exel-Steinberg}.

Our approach is based on~\cite[Chapter 2]{CarlsenPhD}:
To any one-sided shift space $(\X, \sigma_\X)$, we construct a cover $\tilde{\X}$ equipped with a local homeomorphism $\sigma_{\tilde{\X}}\colon \tilde{\X}\LRA \tilde{\X}$
and a surjection $\pi_\X\colon \tilde{\X}\LRA \X$ satisfying $\sigma_\X\circ \pi_\X = \pi_\X\circ \sigma_{\tilde{\X}}$.
The pair $(\tilde{\X}, \sigma_{\tilde{\X}})$ is a Deaconu--Renault system.
From $(\tilde{\X}, \sigma_{\tilde{\X}})$, we construct the Deaconu--Renault groupoid $\G_\X$ which is \'etale and consider the associated groupoid $\mathrm{C^*}$-algebra $\OO_\X$.
Starling constructed the space $\tilde{\X}$ as the tight spectrum of a certain inverse semigroup $\mathcal{S}_\X$ associated to $\X$
and showed that $\OO_\X$ is $^*$-isomorphic to the tight $\mathrm{C^*}$-algebra of $\mathcal{S}_\X$, cf.~\cite[Theorem 4.8]{Starling2016}.
The construction of $\tilde{\X}$ generalizes the left Krieger cover of a sofic shift 
(see~\cite{Krieger1984} where it is called \emph{the past state chains} or~\cite[Exercise 6.1.9]{Kitchens}).
From~\cite[Remark 3.8]{Carlsen2003}, we therefore know that for sofic shifts the $\mathrm{C^*}$-algebra $\OO_\X$ is $^*$-isomorphic to a Cuntz--Krieger algebra.

The paper is structured in the following way:
In Section~\ref{sec:approach}, we define the cover $\tilde{\X}$ and the associated groupoid $\G_\X$.
We characterize when $\G_\X$ is principal or effective, respectively, in terms of conditions on $\X$ (Propositions~\ref{prop:principal} and~\ref{prop:essentially-principal}).
In Section~\ref{sec:diagonal-preserving}, we show that any $^*$-isomorphism $\OO_\X\LRA \OO_\Y$
which maps $C(\X)$ onto $C(\Y)$ is in fact diagonal-preserving (Theorem~\ref{thm:diagonal-preserving}).
Sections~\ref{sec:one-sided-conjugacy},~\ref{sec:eventual-conjugacy} and~\ref{sec:two-sided-conjugacy}
give complete characterizations of one-sided conjugacy (Theorem~\ref{thm:one-sided-conjugacy}), 
one-sided eventual conjugacy (Theorem~\ref{thm:eventual-conjugacy}) 
and two-sided conjugacy (Theorem~\ref{thm:two-sided-conjugacy}), respectively,
in terms of isomorphism of groupoids and diagonal-preserving $^*$-isomorphism of $\mathrm{C^*}$-algebras.
As opposed to Matsumoto, our results are not limited to the case where the groupoid is effective,
and we characterize the relations on the shift spaces and not only the covers (or the $\lambda$-graphs).

In Section~\ref{sec:coe} we study continuous orbit equivalence:
We characterize stabilizer-preserving continuous orbit equivalence in terms of isomorphisms of groupoids which respect certain cocycles,
and $^*$-isomorphisms of $\mathrm{C^*}$-algebras which respect certain gauge actions (Theorem~\ref{thm:coe}).
Section~\ref{sec:flow} concerns flow equivalence:
We can characterize flow equivalence in terms of isomorphism of stabilized groupoids which respects certain cohomological data,
and $^*$-isomorphism of stabilized $\mathrm{C^*}$-algebras which respect certain gauge actions suitably stabilized (Theorem~\ref{thm:flow}).
When the groupoids involved are effective, some of the conditions simplify.
In particular, we obtain the following result related to the flow equivalence problem for sofic shifts.

\begin{theorem*}[Proposition~\ref{prop:essentially-principal}, Theorem~\ref{thm:flow-sofic}, Corollary~\ref{cor:coe-flow-sofic}]
	Let $\Lambda_\X$ and $\Lambda_\Y$ be two-sided sofic shift spaces such that $\G_\X$ and $\G_\Y$ are effective.
	Then $\Lambda_\X$ and $\Lambda_\Y$ are flow equivalent if and only if 
    there is a $^*$-isomorphism $\Phi\colon \OO_\X\otimes \K\LRA \OO_\Y\otimes \K$ satisfying $\Phi(C(\X)\otimes c_0) = C(\Y)\otimes c_0$.
	Furthermore, if $\X$ and $\Y$ are continuously orbit equivalent, then $\Lambda_\X$ and $\Lambda_\Y$ are flow equivalent.
\end{theorem*}

In most sections we prove our results by lifting a relation on the shift spaces to a similar relation on the covers.
We can then encode this relation into structure-preserving $^*$-isomorphisms of the $\mathrm{C^*}$-algebras using groupoids as an intermediate step.
The results of~\cite{CRST} then allow us to reconstruct the groupoid from the $\mathrm{C^*}$-algebras.

\section{Preliminaries}\label{sec:prelim}

We let $\Z$ denote the integers and let $\N = \{0,1,2,\ldots\}$ and $\N_+ = \{1,2,3,\ldots\}$ denote the nonnegative and positive integers, respectively.

\subsection{Symbolic dynamics}

    Let $\A$ be a finite set of symbols (the \emph{alphabet}) considered as a discrete space and let $|\A|$ denote its cardinality.
    Then
    \[
        \A^{\N} = \{ x = x_0x_1x_2\cdots \mid x_i\in \A, i\in \N\}
    \]
    is a second-countable, compact Hausdorff space when equipped with the subspace topology of the product topology on $\A^\N$.
    The \emph{shift-operation} $\sigma\colon \A^{\N}\LRA \A^{\N}$ is the continuous surjection given by $\sigma(x)_n = x_{n+1}$, for $x\in \A^\N$.
    A \emph{one-sided shift space} is a pair $(\X, \sigma_\X)$ in which $\X\subset \A^{\N}$ is closed and \emph{shift-invariant} 
    in the sense that $\sigma(\X) \subset \X$ (we do not assume equality) and where $\sigma_\X := \sigma|_\X\colon \X\LRA \X$.

    Let $\X$ be a one-sided shift space over the alphabet $\A$.
    If $x = x_0 x_1 x_2\cdots\in \X$, we write $x_{[i,j)} = x_i x_{i+1}\cdots x_{j-1}$ for $0\leq i<j$ and $x_{[i,\infty)} = x_i x_{i+1}\cdots$ for $0\leq i$.
    A finite \emph{word} $\mu = \mu_1\cdots \mu_k$ with $\mu_i\in \A$, for each $i=0,\ldots,k$, is \emph{admissible} in $\X$ 
    if $x_{[i, j)} = \mu$ for some $x\in \X$.
    Let $|\mu| = k$ denote the \emph{length} of $\mu$.
    The \emph{empty word} $\varepsilon$ is the unique word of length zero which satisfies $\varepsilon \mu = \mu = \mu \varepsilon$ for any word $\mu$.
    The collection of admissible words in $\X$ of length $l$ is denoted $\LL_l(\X)$
    and the \emph{language} of $\X$ is then the monoid consisting of the union $\LL(\X) = \bigcup_{l\geq 0} \LL_l(\X)$;
    the product being concatenation of words.

    The \emph{cylinder set} of a word $\mu\in \LL(\X)$ is the compact and open set
    \[
        Z_{\X}(\mu) = \{\mu x\in \X\mid x\in \X\},
    \]
    and the collection of sets of the form $Z_\X(\mu)$ constitute a basis for the topology of $\X$.
    A point $x\in \X$ is \emph{isolated} if there is a $k\in \N$ such that $\{x\} = Z_\X(x_{[0, k)})$.

    A point $x\in \X$ is \emph{periodic} if there exists $p\in \N_+$ such that $\sigma_\X^p(x) = x$
    and \emph{eventually periodic} if there is an $n\in \N$ such that $\sigma_\X^n(x)$ is periodic.
    The \emph{least period} of an eventually periodic point $x\in \X$ is
    \[
        \lp(x) = \min\{ p\in \N_+ \mid \exists n, m\in \N: p = n - m,~\sigma_\X^n(x) = \sigma_\X^m(x)\}.
    \]
    A point is \emph{aperiodic} if it is not eventually periodic.
    The \emph{stabilizer} of $x\in \X$ is the group $\Stab(x) = \{ p\in \Z \mid \exists k, l\in \N: p = k - l, \sigma_\X^k(x) = \sigma_\X^l(x)\}$.

    Following~\cite{Mat1999}, we define for every $x\in \X$ and $l\in \N$ the \emph{predecessor set} as
    \[
        P_{l}(x) = \{\mu\in \LL_l(\X) \mid \mu x\in \X\}.
    \]
    Two points $x, y\in \X$ are \emph{$l$-past equivalent} if $P_l(x) = P_l(y)$, in which case we write $x\sim_l y$.
    Let $[x]_l$ be the $l$-past equivalence class of $x$.
    A point $x\in \X$ is \emph{isolated in past equivalence} if there is an $l\in \N$ such that $[x]_l$ is a singleton.
    A shift space $\X$ satisfies \emph{Matsumoto's condition (I)}~\cite[p. 680]{Mat1999} if no points are isolated in past equivalence;
    this is a generalization of Cuntz and Krieger's condition (I).
    We shall also consider the slightly weaker condition that there are no \emph{periodic} points which are isolated in past equivalence.

    A \emph{two-sided shift space} is a subset $\Lambda \subset \A^\Z$ which is closed and shift invariant
    with respect to the shift operation $\sigma\colon \A^\Z\LRA \A^\Z$ 
    given by $\sigma(\textrm{x})_n = \textrm{x}_{n+1}$, for $\textrm{x} = \ldots x_{-1} x_0 x_1\ldots\in \Lambda$ and $n\in \Z$.
    Let $\sigma_\Lambda = \sigma|_\Lambda\colon \Lambda\LRA \Lambda$.
    A pair of two-sided shift spaces $(\Lambda_1, \bar{\sigma}_1)$ and $(\Lambda_2, \bar{\sigma}_2)$ are \emph{two-sided conjugate} if
    there is a homeomorphism $h\colon \Lambda_1 \LRA \Lambda_2$ satisfying $h\circ \sigma_1 = \sigma_2\circ h$.
    We shall consider conjugacy of two-sided shift spaces in Section~\ref{sec:two-sided-conjugacy}.
    
    Given a two-sided shift space $(\Lambda, \bar{\sigma}_\Lambda)$ there is a corresponding one-sided shift space defined by
    \[
        \X_\Lambda = \{\textrm{x}_{[0, \infty)} \in \A^\N \mid \textrm{x}\in \Lambda\}
    \]
    together with the obvious shift operation.
    Conversely, if $(\X, \sigma_\X)$ is a one-sided shift space and $\sigma_\X$ is surjective, then the pair consisting of the projective limit
    \[
        \Lambda_\X = \varprojlim(\X, \sigma_\X)
    \]
    together with the induced shift homeomorphism $\sigma_\X\colon \Lambda_\X\LRA \Lambda_\X$ given by $\sigma_\X(\textrm{x})_n = \textrm{x}_{n+1}$
    for $\textrm{x}\in \Lambda$ is the corresponding \emph{two-sided shift space}
    (this is called the \emph{natural extension} of $\X$ in~\cite[Section 9]{DHS1999}).
    The two operations are mutually inverse to each other.
    See~\cite{LM, Kitchens} for excellent introductions to the general theory of symbolic dynamics.

\subsection{$\mathrm{C}^*$-algebras of shift spaces}
To each shift space $\X$, there is a universal unital $\mathrm{C}^*$-algebra $\OO_\X$ which was first constructed as a Cuntz--Pimsner algebra~\cite[Definition 5.1]{Carlsen2008}.
In Section~\ref{sec:approach}, we follow~\cite[Chapter 2]{CarlsenPhD} and construct 
a second-countable, amenable, locally compact, Hausdorff and \'etale groupoid $\G_{\X}$ whose $\mathrm{C}^*$-algebra is canonically isomorphic to $\OO_{\X}$.
For an introduction to (\'etale) groupoid $\mathrm{C}^*$-algebras see~\cite{Renault80, Paterson} or the introductory notes~\cite{SimsNotes}.

We briefly recall the universal description of $\OO_\X$ given in~\cite[Remark 7.3]{Carlsen2008}.
Given words $\mu,\nu\in \LL(\X)$, consider the set
\[
    C_{\X}(\mu,\nu) := \{ \nu x\in \X\mid \mu x\in \X\}
\]
which is closed (but not necessarily open) in $\X$.
We shall refer to the commutative $\mathrm{C}^*$-algebra
\[
    \D_\X := \mathrm{C}^*\{ 1_{C_{\X}(\mu,\nu)} \mid \mu,\nu\in \LL(\X) \}
\]
inside the $\mathrm{C^*}$-algebra of bounded functions on $\X$ as the \emph{diagonal}.
The $\mathrm{C}^*$-algebra $\OO_{\X}$ is the universal unital $\mathrm{C}^*$-algebra generated by partial isometries $(s_{\mu})_{\mu\in \LL(\X)}$ satisfying
\[
    s_{\mu} s_{\nu} =
    \begin{cases}
        s_{\mu \nu} & \mu \nu\in \LL(\X), \\
        0 & \textrm{otherwise},
    \end{cases}
\]
and such that the map 
\[
    1_{C(\mu, \nu)} \LMT s_{\nu}s_{\mu}^*s_{\mu}s_{\nu}^*,
\]
for $\mu, \nu\in \LL(\X)$, extends to $^*$-homomorphism $\D_\X \LRA \mathrm{C}^*(s_\mu \mid \mu\in \LL(\X))$.
This map is injective and the projections $\{s_{\nu}s_{\mu}^*s_{\mu}s_{\nu}^*\}_{\mu, \nu}$ generate a commutative $\mathrm{C}^*$-subalgebra 
which is $^*$-isomorphic to $\D_{\X}$ via the above map.
We shall henceforth identify $\D_\X$ with this $\mathrm{C}^*$-subalgebra of $\OO_{\X}$.

The universal property ensures that there is a \emph{canonical gauge action} $\gamma^{\X}\colon \T\curvearrowright \OO_{\X}$ of the circle group $\T$ given by
\[
    \gamma^\X_z(s_{\mu}) = z^{|\mu|}s_{\mu},
\]
for every $z\in \T$ and $\mu\in \LL(\X)$.
The fixed point algebra under the gauge action is an AF-algebra which is denoted $\F_\X$.
Note that $\D_{\X} \subset \F_\X$.

\section{Basic constructions}\label{sec:approach}

Let $\X$ be a one-sided shift space.
In this section, we associate a cover $\tilde{\X}$ to $\X$ and build a groupoid $\G_\X$ from the cover and its dynamical properties.
This construction is due to the second-named author in~\cite[Chapter 2]{CarlsenPhD}.
The $\mathrm{C^*}$-algebra $\OO_\X$ is then constructed as a groupoid $\mathrm{C^*}$-algebra.

\subsection{The cover $\tilde{\X}$}
Consider the set $\I = \{(k,l)\in \N\times \N\mid k\leq l\}$ equipped with the partial order $\preceq$ given by
\[
    (k_1,l_1)\preceq (k_2,l_2) \iff k_1\leq k_2 \textup{ and } l_1 - k_1 \leq l_2 - k_2.
\]
For every $(k,l)\in \I$ we define an equivalence relation on $\X$ by
\[
    x\widesim{k,l} x' \iff x_{[0,k)} = x'_{[0,k)} \textup{ and } \bigcup_{l'\leq l} P_{l'}(\sigma_{\X}^k(x)) = \bigcup_{l'\leq l} P_{l'}(\sigma_{\X}^k(x')).
\]
The $(k,l)$-equivalence class of $x\in \X$ is denoted $_k[x]_l$ and each $_k\X_l = \{_k[x]_l \mid x\in \X\}$ is a finite set.
If $(k_1,l_1)\preceq (k_2,l_2)$, then 
\[
    x \widesim{k_2,l_2} x' \implies x\widesim{k_1,l_1} x',
\]
for every $x,x'\in \X$.
Hence there is a well-defined map $_{(k_1,l_1)} Q_{(k_2,l_2)}\colon {}_{k_2}\X_{l_2}\LRA {}_{k_1}{\X}_{l_1}$ given by
\[
    {}_{(k_1,l_1)} Q_{(k_2,l_2)}({}_{k_2}{[x]}_{l_2}) = {}_{k_1}{[x]}_{l_1},
\]
for every ${}_{k_2}{[x]}_{l_2}\in {}_{k_2}\X_{l_2}$.
When the context is clear, we shall omit the subscripts of the map.
The spaces ${}_k\X_l$ together with the maps $Q$ thus define a projective system.

\begin{definition}
    Let $\X$ be a one-sided shift space.
    The \emph{cover} of $\X$ is the second-countable compact Hausdorff space $\tilde{\X}$ defined as the 
    projective limit $\underset{(k,l)\in \I}\varprojlim ({}_k\X_l, Q)$.
    We identify this with
    \[
        \tilde{\X} = \bigg\{ {({}_k {[{}_k x_l]}_l)}_{(k,l)\in \I} \in \prod_{(k,l)\in \I} {}_k \X_l \mid (k_1,l_1)\preceq (k_2,l_2)\colon 
        {}_{k_1} {[{}_{k_1}x_{l_1}]}_{l_1} = {}_{k_1} {[{}_{k_2}x_{l_2}]}_{l_1} \bigg \}
    \]
    equipped with the subspace topology of the product topology of $\prod_{(k,l)\in \I} {}_k\X_l$.
\end{definition}

The topology of $\tilde{\X}$ is generated by compact open sets of the form
\[
    U(x,k,l) = \{ \tilde{x}\in \tilde{\X} \mid {}_k x_l \widesim{k,l}  x\},
\]
for $x\in \X$ and $(k,l)\in \I$.
In order to see that sets of the above form constitute a basis, let $\tilde{x}\in U(y,k_1,l_1)\cap U(z,k_2,l_2)$.
Set $k := \max\{k_1, k_2\}$ and $l := l_1 + l_2$.
The pair $(k,l)$ thus majorizes both $(k_1,l_1)$ and $(k_2,l_2)$, and
\[
    \tilde{x}\in U({}_kx_l,k,l) \subset U(y,k_1,l_1)\cap U(z,k_2,l_2),
\]
Given a word $\mu\in \LL(\X)$, we also consider the compact open sets
\[
    U_{\mu} := \bigcup_{x\in C(\mu)} U(x,|\mu|,|\mu|).
\]

We shall now determine a \emph{shift operation} on $\tilde{\X}$ endowing it with the structure of a dynamical system.
For any $(k,l)\in \I$ with $k\geq 1$, observe that
\[
    x\widesim{k,l} y \implies \sigma_{\X}(x) \widesim{k-1,l} \sigma_{\X}(y).
\]
Therefore, there is a well-defined map ${}_k\sigma_l\colon {}_k\X_l\LRA {}_{k-1}\X_l$ given by
\[
    {}_k\sigma_l({}_k{[x]}_l) = {}_{k-1}{[\sigma_{\X}(x)]}_l,
\]
for every ${}_k{[x]}_l\in {}_k\X_l$, $k\geq 1$.
When the context is clear, we shall omit the subscripts.
Furthermore, this shift operation intertwines the maps $Q$ in the sense that the diagram
\begin{center}
    \begin{tikzpicture}
        \matrix (m) [matrix of math nodes,row sep=2em, column sep=2em,minimum width=4em]
        {
            {}_{k_2}\X_{l_2} & {}_{k_2-1}\X_{l_2} \\
            {}_{k_1}\X_{l_1} & {}_{k_1-1}\X_{l_1} \\
        };
        \path[->]
        (m-1-1) edge node[above] {$\sigma$} (m-1-2)
                edge node[left] {$Q$} (m-2-1)
        (m-1-2) edge node[right] {$Q$} (m-2-2)
        (m-2-1) edge node[below] {$\sigma$} (m-2-2);
    \end{tikzpicture}
\end{center}
commutes for every $(k_1,l_1), (k_2,l_2)\in \I$ with $(k_1,l_1)\preceq (k_2,l_2)$ and $k_1\geq 1$.
It follows that there is an induced shift operation $\sigma_{\tilde{\X}}\colon \tilde{\X}\LRA \tilde{\X}$ given by
\[
    {}_k\sigma_{\tilde{\X}}(\tilde{x})_l = {}_{k+1}\sigma_l({}_{k+1}[{}_{k+1}x_l]_l) = {}_k[\sigma_{\X}({}_{k+1}x_l)]_l,
\]
for every $\tilde{x} = ({}_k[{}_k x_l]_l)_{(k,l)\in \I}\in \tilde{\X}$.
The pair $(\tilde{\X}, \sigma_{\tilde{\X}})$ is then a dynamical system.

There is a canonical continuous and surjective map $\pi_{\X}\colon \tilde{\X}\LRA \X$ given in the following way: 
If $\tilde{x}\in \tilde{\X}$, then $x=\pi_{\X}(\tilde{x})\in \X$ is the unique element with the property that $x_{[0,k)} = {({}_kx_l)}_{[0,k)}$,
for every $(k,l)\in \I$.
This map intertwines the shift operations in the sense that 
\[
    \sigma_\X\circ \pi_\X = \pi_\X\circ \sigma_{\tilde{\X}}.
\]
We shall refer to $\pi_\X$ as the \emph{canonical factor map} associated to $\X$.
It is injective (and thus a homeomorphism) if and only if $\X$ is a shift of finite type.

On the other hand, there is a \emph{function} $\iota_\X\colon \X\LRA \tilde{\X}$ given by sending $x\in \X$ to $\tilde{x}\in \tilde{\X}$
for which ${}_k x_l = x$, for every $(k,l)\in \I$.
This satisfies the relation $\pi_{\X}\circ \iota_{\X} = \textrm{id}_{\X}$.
If $x\in \X$ is isolated, then $\pi_\X^{-1}(x) = \{\iota_\X(x)\}$.
However, $\iota_{\X}$ is in general \emph{not} continuous.

\begin{example}\label{ex:even-example}
    The even shift $\X_{\mathrm{even}}$ is the strictly sofic one-sided shift space over the alphabet $\{0,1\}$
    determined by the forbidden words $\mathscr{F} = \{ 10^{2n+1} 1\mid n\in \N\}$
    (see, e.g.,~\cite[Section 3]{LM} for an introduction to sofic shifts).
    The space $\X_{\mathrm{even}}$ contains no isolated points, but $0^\infty$ is the unique element for which $P_2(0^\infty) = \{00, 10, 01\}$,
    so $0^\infty$ is isolated in past equivalence.
    Hence $\iota_{\mathrm{even}}(0^\infty)\in \tilde{\X}_{\mathrm{even}}$ is isolated and $\iota_{\mathrm{even}}$ is not continuous.
\end{example}

\begin{lemma}
    The shift operation $\sigma_{\tilde{\X}}\colon \tilde{\X}\LRA \tilde{\X}$ is a local homeomorphism.
\end{lemma}

\begin{proof}
    We show that $\sigma_{\tilde{\X}}$ is open and locally injective.
    For the first part, let $z\in \X$ and $(k,l)\in \I$ with $k\geq 1$ and suppose $a=z_0\in \A$.
    We claim that 
    \[
        \sigma_{\tilde{\X}}(U(z,k,l)) = U(\sigma_{\X}(z), k - 1, l).
    \]
    The left-to-right inclusion is straightforward.
    For the converse let $\tilde{x}\in U(\sigma_{\X}(z),k-1,l)$ and note that $(0,1)\preceq (k-1,l)$.
    Since ${}_{k-1}x_l \widesim{k-1,l} \sigma_{\X}(z)$ it thus follows that ${}_{k-1}x_l \widesim{0,1} \sigma_{\X}(z)$.
    As $a\in P_1(\sigma_{\X}(z))$, we see that $a {}_{k-1}x_l\in \X$.
    A similar argument shows that $a{}_{r} x_{s}\in \X$, for every $(r,s)\in \I$.
    Put ${}_r y_s = a {}_{r} x_{s+1}$, for every $(r,s)\in \I$.
    Now, if $(k_1,l_1)\preceq (k_2,l_2)$ in $\I$, then $a {}_{k_2}x_{l_2+1} \widesim{k_1,l_1} a {}_{k_1}x_{l_1+1}$ and so
    \[
        {}_{(k_1,l_1)}Q_{(k_2,l_2)} ({}_{k_2}[{}_{k_2}y_{l_2}]_{l_2} )
        = {}_{k_1}[a {}_{k_2}x_{l_2+1}]_{l_1}
        = {}_{k_1}[a {}_{k_1}x_{l_1+1}]_{l_1}
        = {}_{k_1}[{}_{k_1}y_{l_1}]_{l_1}.        
    \]
    Hence $\tilde{y} = ({}_r[{}_r y_s]_s)_{(r,s)\in \I} \in \tilde{\X}$.
    Observe now that
    \[
        {}_{k}y_l = a {}_k x_{l+1} \widesim{k,l} z,
    \]
    showing that $\tilde{y}\in U(z,k,l)$.
    Finally, we see that $\tilde{x} = \sigma_{\tilde{\X}}(\tilde{y}) \in \sigma_{\tilde{\X}}(U(z,k,l))$ so $\sigma_{\tilde{\X}}$ is open.

    In order to see that $\sigma_{\tilde{\X}}$ is locally injective let $z\in \X$ with $a = z_0\in \A$.
    We claim that $\sigma_{\tilde{\X}}$ is injective on $U(x,1,1)$.
    Indeed, suppose $\tilde{x},\tilde{y}\in U(x,1,1)$ and $\sigma_{\tilde{\X}}(\tilde{x}) = \sigma_{\tilde{\X}}(\tilde{y})$.
    In particular, $({}_kx_l)_0 = z_0 = ({}_k y_l)_0$ for every $(k,l)\in \I$.
    Hence
    \[
        {}_k x_l = a\sigma_{\X}({}_k x_l) \widesim{k,l} a\sigma_{\X}({}_k y_l) = {}_k y_l
    \]
    for every $(k,l)\in \I$ from which it follows that $\tilde{x}=\tilde{y}$.
    We conclude that $\sigma_{\tilde{\X}}$ is a local homeomorphism.
\end{proof}

\begin{remark}
    The cover $(\tilde{\X}, \sigma_{\tilde{\X}})$ is a \emph{Deaconu--Renault system} in the sense of~\cite[Section 8]{CRST},
    and the construction is a generalization of the left Krieger cover (see~\cite{Krieger1984} where it is called \emph{the past state chains} or~\cite[Exercise 6.1.9]{Kitchens}) of a sofic shift space.
    In particular, the cover $(\tilde{\X}, \sigma_{\tilde{\X}})$ of a sofic shift $(\X, \sigma_\X)$ is (conjugate to) a shift of finite type.
\end{remark}

The next lemma shows how the topologies of $\X$ and $\tilde{\X}$ interact.
\begin{lemma}\label{lem:continuity}
    Let $\X$ be a one-sided shift space and let $k\colon \X\LRA \N$ be a map.
    Then the map $k_{\tilde{\X}}\colon \tilde{\X}\LRA \N$ satisfying $k_{\tilde{\X}} = k\circ \pi_\X$ is continuous 
    if and only if $k$ is continuous.
\end{lemma}

\begin{proof}
    Define $k_{\tilde{\X}}\colon \tilde{\X}\LRA \N$ by $k_{\tilde{\X}} = k \circ \pi_\X$.
    If $k$ is continuous, then $k_{\tilde{\X}}$ is continuous.

    Suppose $k$ is not continuous.
    Then there is an element $x\in \X$ and a convergent sequence ${(x_n)}_n$ with limit $x$ such that $k(x_n) \neq k(x)$ for all $n\in \N$.
    In particular, the set
    \[
        C_i = \{x_n \mid n\in \N\} \cap Z(x_{[0, i)})
    \]
    is nonempty for each $i\in \N$.
    As $\pi_\X$ is surjective, $\tilde{C}_i = \pi_\X^{-1}(C_i)$ is nonempty.
    Choose $\tilde{t}_i\in \tilde{C}_i$ for each $i\in \N$.
    Then $\pi_\X(\tilde{t_i}) = x_{n_i}$ for some $n_i\in \N$ so $\tilde{k}(\tilde{t}_i) \neq k(x)$ for all $i\in \N$.
    Furthermore, the sequence ${(\tilde{t}_i)}_i$ has a convergent subsequence ${(\tilde{t}_{i_j})}_j$ with some limit $\tilde{x}$ which satisfies
    \[
        \pi_\X(\tilde{x}) = \pi_\X(\lim_{j\to \infty} \tilde{t}_{i_j}) = \lim_{j\to \infty} \pi_\X(\tilde{t}_{i_j}) = \lim_{j\to \infty} x_{n_{i_j}} = x,
    \]
    so $\tilde{x} \in \pi_\X^{-1}(x)$.
    Then $\tilde{t}_{i_j}\LRA \tilde{x}$ in $\tilde{\X}$ and $k_{\tilde{\X}}(\tilde{x}) = k(x) \neq k_{\tilde{\X}}(\tilde{t}_{i_j})$ for every $j\in \N$,
    so $k_{\tilde{\X}}$ is not continuous.
\end{proof}

The cover $\tilde{\X}$ may contain isolated points even if $\X$ does not, cf.~Example~\ref{ex:even-example}.
In~\cite[Lemma 4.3(1)]{CEOR}, it is shown that every isolated point in a shift of finite type is eventually periodic.
This is also the case for the class of sofic shift space but it need not be true in general.

\begin{lemma}\label{lem:isolated-points-sofic}
    Let $\X$ be a one-sided \emph{sofic} shift.
    If $x\in \X$ is isolated, then $x$ is eventually periodic.
\end{lemma}

\begin{proof}
    Let $x\in \X$ be isolated.
    Then $\tilde{x}\in \pi_\X^{-1}(x)$ is isolated. 
    The cover $\tilde{\X}$ is (conjugate to) a shift of finite type, so $\tilde{x}$ is eventually periodic, cf.~\cite[Lemma 4.3(1)]{CEOR}.
    Hence $x = \pi_\X(\tilde{x})$ is eventually periodic.
\end{proof}

\begin{example}
    Consider the shift space $\X_\omega$ over the alphabet $\{0,1\}$ generated by the sequence
    \[
        \omega = 01 010 0100 01000 \cdots.
    \]
    Since $\omega$ is not periodic, $\X_\omega$ is infinite.
    The shift operation $\sigma_\omega$ is not surjective and $\X_\omega$ is not minimal.
    We can identify $\X_\omega$ with the orbit of $\omega$ together with all its accumulation points, i.e.,
    \[
        \X_\omega = \{\sigma_{\omega}^i(\omega): i\in \N\} \cup \{0^n 10^\infty: n\in \N\} \cup \{0^\infty\}
    \]
    in which $\{\sigma^i(\omega): i\in \N\}$ are exactly the isolated points of $\X_\omega$.
    In particular, $\omega\in \X_\omega$ is isolated and aperiodic.
    It follows from Lemma~\ref{lem:isolated-points-sofic} that $\X_\omega$ is not sofic.
    Observe also that $0^\infty\in \X_\omega$ is periodic point isolated in past equivalence.
    In fact, every point in $\{0^n 10^\infty: n\in \N\}$ is isolated in past equivalence,
    so $\pi_{\X_\omega}^{-1}(x)$ contains an isolated point for every $x\in \X_\omega$.
\end{example}

\subsection{The groupoid $\G_\X$}\label{sec:groupoid}

The pair $(\tilde{\X}, \sigma_{\tilde{\X}})$ is a Deaconu--Renault system in the sense of~\cite[Section 8]{CRST}.
The associated Deaconu--Renault groupoid~\cite{De1995} is
\[
    \G_\X = \{ (\tilde{x}, p, \tilde{y})\in \tilde{\X}\times \Z\times \tilde{\X} \mid 
    \exists i, j\in \N\colon p = i - j,~\tilde{x},\tilde{y}\in \tilde{\X},~\sigma_{\tilde{\X}}^i(\tilde{x}) = \sigma_{\tilde{\X}}^j(\tilde{x}) \}.
\]
The product of $(\tilde{x}, p, \tilde{y})$ and $(\tilde{y'}, q, \tilde{z})$ is defined if and only if $\tilde{y} = \tilde{y'}$ in which case
\[
    (\tilde{x}, p, \tilde{y}) (\tilde{y'}, q, \tilde{z}) = (\tilde{x}, p + q,\tilde{z}),
\]
while inversion is given by $(\tilde{x}, p, \tilde{y})^{-1} = (\tilde{y}, -p, \tilde{x})$.
The range and source maps are given as
\[
    r(\tilde{x}, p, \tilde{y}) = (\tilde{x},0,\tilde{x}), \qquad s(\tilde{x}, p, \tilde{y}) = (\tilde{y}, 0, \tilde{y}),
\]
respectively, for $(\tilde{x}, p, \tilde{y})\in \G_{\X}$.
The topology of $\G_{\X}$ is generated by sets of the form
\[
    Z(U, i, j, V) = \{ (\tilde{x}, i - j,\tilde{y})\in \G_{\X} \mid (\tilde{x},\tilde{y})\in U\times V \}
\]
where $U,V\subset \tilde{\X}$ are open subsets such that $\sigma_{\tilde{\X}}^i|_U$ and $\sigma_{\tilde{\X}}^j|_V$ are injective 
and $\sigma_{\tilde{\X}}^i(U) = \sigma_{\tilde{\X}}^j(V)$.
We naturally identify the unit space $\G_{\X}^{(0)} = \{ (\tilde{x}, 0, \tilde{x})\in \G_\X\mid x\in \tilde{\X}\}$ with the space $\tilde{\X}$
via the map $(\tilde{x}, 0, \tilde{x}) \LMT \tilde{x}$.
Equipped with this topology, $\G_\X$ is topological groupoid which is second-countable, locally compact Hausdorff 
and \'etale (in the sense that $r,s\colon \G_\X\LRA \G_\X$ are local homeomorphism onto $\G_\X^{(0)}$), cf.~\cite[Lemma 3.1]{Sims-Williams}.
By~\cite[Lemma 3.5]{Sims-Williams}, $\G_\X$ is also amenable.

The \emph{isotropy} of a point $\tilde{x}\in \tilde{\X}$ is the set 
\[
    \Iso(\tilde{x}) = \{ (\tilde{x}, p, \tilde{x})\in \G_\X\} 
\]
which carries a natural group structure.
In our case, the group $\Iso(\tilde{x})$ is always (isomorphic to) $0$ or $\Z$.
The \emph{stabilizer} is $\Stab(\tilde{x}) = \{p\in \Z\mid (\tilde{x}, p, \tilde{x})\in \Iso(\tilde{x})\}$.
The \emph{isotropy subgroupoid} of $\G_\X$ is the group bundle
\[
    \Iso(\G_\X) = \bigcup_{\tilde{x}\in \tilde{\X}} \Iso(\tilde{x}).
\]

The groupoid $\G_\X$ is \emph{principal} if $\Iso(\G_\X) = \G_\X^{(0)}$, 
and \emph{effective} if ${\Iso(\G_\X)}^\circ = \G_\X^{(0)}$, where ${\Iso(\G)}^\circ$ denotes the interior of the isotropy subgroupoid.
Since $\G_\X$ is second-countable and Hausdorff the latter is equivalent to $\G_\X$ being topologically principal,
i.e., that the set of points with trivial isotropy is dense in the unit space, cf.~\cite[Propostion 3.6]{Renault08}.
Below we characterize when the groupoid $\G_\X$ is principal and effective in terms of $\X$.
First we need a lemma.

\begin{lemma}\label{lem:isotropy-lemma}
    Let $\X$ be a one-sided shift space and let $\tilde{x}, \tilde{y}\in \tilde{\X}$.
    \begin{enumerate}
        \item[(i)] If $\pi_\X(\tilde{x}) = \pi_\X(\tilde{y})$ and $\sigma_{\tilde{\X}}^k(\tilde{x}) = \sigma_{\tilde{\X}}^k(\tilde{y})$ for some $k\in \N$, 
            then $\tilde{x} = \tilde{y}$. 
        \item[(ii)] If $\pi_\X(\tilde{x}) = \pi_\X(\tilde{y})$ is aperiodic and $\sigma_{\tilde{\X}}^l(\tilde{x}) = \sigma_{\tilde{\X}}^{k}(\tilde{y})$ for some $k,l\in \N$, 
            then $\tilde{x} = \tilde{y}$.
    \end{enumerate}
\end{lemma}

\begin{proof}
    (i):
    Fix $k\in \N$ such that $\tilde{\sigma}^k_\X(\tilde{x}) = \tilde{\sigma}^k_\X(\tilde{y})$ and
    let $0\leq r\leq s$ be integers with $r+k\leq s$.
    An $(r,s)$-representative of $\tilde{\sigma}^k_\X(\tilde{x})$ and $\sigma_{\tilde{\X}}^k(\tilde{y})$
    is given by $\sigma^k_\X({}_{r + k} x_s)$ and $\sigma^k_\X({}_{r+k} y_s)$, respectively.
    So
    \[
        \sigma^k_\X({}_{r + k} x_s) \widesim{r,s} \sigma^k_\X({}_{r + k} y_s).
    \]
    Since $\pi_\X(\tilde{x}) = \pi_\X(\tilde{y})$ we also have ${}_{r + k} x_s \widesim{r + k, s} {}_{r + k} y_s$.
    It follows that $\tilde{x} = \tilde{y}$.
    
    (ii):
    Let $x = \pi_\X(\tilde{x}) = \pi_\X(\tilde{y})$ be aperiodic.
    If $\sigma_{\tilde{\X}}^l(\tilde{x}) = \sigma_{\tilde{\X}}^{k}(\tilde{y})$ for some $k, l\in \N$, then $\sigma_\X^l(x) = \sigma_\X^k(x)$, so $k = l$.
    Part (i) implies that $\tilde{x} = \tilde{y}$.
\end{proof}

Assertion (ii) may fail without the hypothesis of aperiodicity; this happens, e.g., for the even shift, cf.~Example~\ref{ex:even-example}.
It follows from Lemma~\ref{lem:isotropy-lemma}(ii) that the preimage under $\pi_\X$ of an aperiodic element consists only of aperiodic elements.
The preimage under $\pi_\X$ of an eventually periodic point contains an eventually periodic point but we do not know if it consists only of eventually periodic points.

\begin{proposition}\label{prop:principal}
    Let $\X$ be a one-sided shift space.
    The following conditions are equivalent:
    \begin{enumerate}
        \item[(i)] $\X$ contains no eventually periodic points;
        \item[(ii)] $\tilde{\X}$ contains no eventually periodic points;
        \item[(iii)] $\G_\X$ is principal.
    \end{enumerate}
\end{proposition}

\begin{proof}
    (i) $\iff$ (ii):
    It follows from Lemma~\ref{lem:isotropy-lemma}(ii) that if $x\in \X$ is aperiodic, then any $\tilde{x}\in \pi_\X^{-1}(x)\in \tilde{\X}$ is aperiodic.
    So if $\X$ consists only of aperiodic points, then $\tilde{\X}$ contains only aperiodic points.
    Conversely, if $x\in \X$ is eventually periodic, then $\iota_\X(x)\in \tilde{\X}$ is eventually periodic.

    The equivalence (ii) $\iff$ (iii) is obvious.
\end{proof}

\begin{proposition}\label{prop:essentially-principal}
    Let $\X$ be a one-sided shift space.
    The conditions
    \begin{enumerate}
        \item[(i)] $\X$ satisfies Matsumoto's condition (I);
        \item[(ii)] $\tilde{\X}$ contains no isolated points;
    \end{enumerate}
    are equivalent and strictly stronger that the following equivalent conditions
    \begin{enumerate}
        \item[(iii)] $\X$ contains no periodic points isolated in past equivalence;
        \item[(iv)] $\tilde{\X}$ has a dense set of aperiodic points;
        \item[(v)] $\G_\X$ is effective;
    \end{enumerate}
    which are strictly stronger than
    \begin{enumerate}
        \item[(vi)] $\X$ contains a dense set of aperiodic points.
    \end{enumerate}
\end{proposition}

\begin{proof}
    (i) $\iff$ (ii):
    Suppose $x\in \X$ is isolated in past equivalence so that $[x]_l = \{x\}$, for some $l\in \N$.
    Then $\{\iota(x)\} = U(x, 0, l)$ so $\iota(x)$ is isolated in $\tilde{\X}$.
    Conversely, if $\tilde{x}\in \tilde{\X}$ is isolated, say $\{\tilde{x}\} = U(x, r, s)$ for some integers $0\leq r\leq s$,
    then $\{\sigma_{\tilde{\X}}^r(\tilde{x})\} = U(\sigma_\X^r(x), 0, s)$, so $\pi_\X(\sigma_{\tilde{\X}}^r(\tilde{x}))\in \X$ is isolated in $s$-past equivalence.

    The implication (ii) $\implies$ (iii) is clear.

    (iii) $\implies$ (iv):
    Let $\mathcal{EP}(\tilde{\X})$ be the collection of eventually periodic points in $\tilde{\X}$ and set
    \[
        \mathcal{EP}_n^p = \{ \tilde{x}\in \mathcal{EP}(\tilde{\X}) \mid \sigma_{\tilde{\X}}^{n + p}(\tilde{x}) = \sigma_{\tilde{\X}}^n(\tilde{x}) \},
    \]
    for $n\in \N$ and $p\in \N_+$.
    Then $\mathcal{EP}(\tilde{\X}) = \bigcup_{n, p} \mathcal{EP}_n^p$.
    If there is an open set $U\subset \tilde{\X}$ consisting of eventually periodic points,
    then it follows from the Baire Category Theorem that $\mathcal{EP}_n^p$ has nonempty interior for some $n\in \N$ and $p\in \N_+$.
    In particular, there are an $x\in \X$ and integers $0\leq r\leq s$ with $r\leq n$ such that $U(x, r, s) \subset \mathcal{EP}_n^p$.
    Since $\iota_\X(x)\in U(x, r, s)$ it follows that $\sigma_\X^n(x)$ is $p$-periodic.
    We claim that $\sigma_\X^n(x)$ is isolated in past equivalence.

    Write $x = \mu \alpha^\infty$ for some words $\mu, \alpha\in \LL(\X)$ with $|\mu| = n$ and $|\alpha| = p$ and suppose 
    \[
        y \sim_{p + n + r - s} \sigma_\X^n(x) = \sigma_\X^{n + p}(x).
    \]
    Then $\mu \alpha y\in \X$ and $\iota_\X(\mu \alpha y) \in U(x, r, s)$, so $\alpha y = y$.
    Hence $y = \sigma_\X^p(x)$ as wanted.
   
    (iv) $\implies$ (iii):
    Suppose $x\in \X$ is a periodic point and there is an $l\in\N$ such that $[x]_l=\{x\}$.
    Then $U(x, 0, l) = \{\iota_\X(x)\}$ is an open set consisting of points with nontrivial isotropy.

    The equivalence (iv) $\iff$ (v) is obvious.

    (iv) $\implies$ (vi):
    Suppose $\X$ contains an open set $U$ consisting of eventually periodic points.
    Then
    \[
        U = \bigcup_{\alpha, \beta\in \LL(\X)} \{\alpha \beta^\infty\}\cap U,
    \]
    and by the Baire Category Theorem there are $\alpha, \beta\in \LL(\X)$ such that $\{\alpha \beta^\infty\}$ is an isolated eventually periodic point in $U$.
    Then $\iota_\X(\alpha \beta^\infty)$ is an isolated eventually periodic point in $\tilde{\X}$.

    To see that (iii) does not imply (i) observe that if $\X$ is the shift space generated by an aperiodic substitution, 
    then $\X$ contains no eventually periodic points, so $\G_\X$ is principal, cf.~\cite[Definition 5.15]{Queffelec}.
    However, if the substitution in addition to being aperiodic is also primitive and proper then, according to~\cite[Proposition 3.5]{Carlsen-Eilers2007}, 
    $\X$ contains a point which is isolated in past equivalence, so $\X$ does not satisfy Matsumoto’s condition (I).
    
    Finally, the even shift is an example of a shift with a dense set of aperiodic points but it contains a periodic point which is isolated in past equivalence, cf.~Example~\ref{ex:even-example}.
\end{proof}

Any groupoid homomorphism is assumed to be continuous and a groupoid isomorphism is assumed to be a homeomorphism.
A \emph{continuous cocycle} on $\G_\X$ is a groupoid homomorphism $\G_\X\LRA \Z$.
Let $B^1(\G_\X)$ be the collection of continuous cocycles on $\G_\X$.
There is a map $\kappa_\X\colon C(\X, \Z) \LRA B^1(\G_\X)$ given by
\begin{equation}\label{eq:kappa}
  \kappa_\X(f)(\tilde{x}, p, \tilde{y}) = \sum_{i = 0}^l f(\pi_\X(\sigma_{\tilde{\X}}^i(\tilde{x}))) - \sum_{j = 0}^k f(\pi_\X(\sigma_{\tilde{\X}}^j(\tilde{x}))),
\end{equation}
for $f\in C(\X, \Z)$, $(\tilde{x}, p, \tilde{y})\in \G_\X$ and
where $k, l\in \N$ satisfy $p = l - k$ and $\sigma_{\tilde{\X}}^l(\tilde{x}) = \sigma_{\tilde{\X}}^k(\tilde{y})$.
Observe that $\kappa_\X(f)$ is the unique cocycle satisfying 
\[
    \kappa_\X(f)(\tilde{x}, 1, \sigma_{\tilde{\X}}(\tilde{x})) = f(\pi_\X(\tilde{x})),
\]
for $\tilde{x}\in \tilde{\X}$.
The \emph{canonical continuous cocycle} $c_{\X}\colon \G_{\X}\LRA \Z$ is defined by 
\[
    c_\X(\tilde{x}, p, \tilde{y}) = p,
\]
for $(\tilde{x}, p, \tilde{y})\in \G_\X$.
Note that $c_\X = \kappa_\X(1)$ and $c_\X^{-1}(0) = \{(\tilde{x}, 0, \tilde{y})\in \G_\X\} \subset \G_\X$ is a clopen subgroupoid which is always principal.

\subsection{The $\mathrm{C}^*$-algebra $\mathrm{C}^*(\G_\X) = \OO_\X$}

The groupoid $\G_{\X}$ is second-countable, locally compact Hausdorff and \'etale, cf.~Section~\ref{sec:groupoid}.
Let $C_c(\G_\X)$ be the $^*$-algebra consisting of compactly supported and complex-valued maps with the convolution product.
As $\G_{\X}$ is also amenable, the full $\mathrm{C}^*(\G_\X)$ and the reduced $\mathrm{C}^*_{\mathrm{r}}(\G)$ groupoid $\mathrm{C}^*$-algebras 
are canonically $^*$-isomorphic, cf.~\cite[Theorem 4.1.4]{SimsNotes} or~\cite[Proposition 6.1.8]{AD-Renault2000}.
Therefore,~\cite[Proposition 3.3.3]{SimsNotes} allows us to consider $\mathrm{C}^*(\G_\X)$ as a subset of $C_0(\G_{\X})$.

There is a canonical $^*$-isomorphism $\OO_{\X}\LRA \mathrm{C}^*(\G_{\X})$ sending $s_a\LMT 1_{U_a}$ for each $a\in \A$, cf.~\cite[Chapter 2]{CarlsenPhD}.
According to~\cite[Proposition 1.9]{Phillips2005}
the canonical inclusions of $C(\X)$ and $C_c(c_\X^{-1}(0))$ into $C_c(\G_\X)$ extend to injective $^*$-homomorphisms of $C(\X)$ and $\mathrm{C}^*(c_\X^{-1}(0))$ into $\mathrm{C}^*(\G_\X)$. 
We will therefore simply identify $\OO_\X$ with $\mathrm{C}^*(\G_\X)$, $\D_\X$ with $C(\tilde{\X})$, and $\F_\X$ with $\mathrm{C}^*(c_\X^{-1}(0))$. 
The inclusion $\tilde{\X} \LRA \G_\X$ then induces a conditional expectation $p_\X\colon \OO_\X\LRA \D_\X$ given by restriction so that
\[
  p_\X(g)(\tilde{x}) = g(\tilde{x}, 0, \tilde{x}).
\]
for $g\in \OO_\X$ and $\tilde{x}\in \tilde{\X}$.

Any continuous cocycle $c\in B^1(\G_\X)$ induces a strongly continuous action $\beta^c\colon \T\curvearrowright \OO_\X$ satisfying
\[
    \beta^c_z(f) = z^n f
\]
for $z\in \T$ and $n\in \N$ and $f\in C_c(\G_\X)$ with $\supp(f) \subset c^{-1}(\{n\})$.
The canonical gauge action $\gamma^\X = \beta^{\kappa_\X(1)}$ is of the form
\[
    \gamma^\X_z(g)(\tilde{x}, p, \tilde{y}) = z^p g(\tilde{x}, p, \tilde{y}),
\]
for every $z\in \T$, $g\in C_c(\G_\X)$ and $(\tilde{x}, p, \tilde{y})\in \G_{\X}$.

Let $\K$ denote the $\mathrm{C}^*$-algebra of compact operators on separable Hilbert space
and let $c_0$ denote the canonical maximal abelian $\mathrm{C}^*$-subalgebra of diagonal operators in $\K$.

\section{Preserving the diagonal}\label{sec:diagonal-preserving}

Let $\X$ and $\Y$ be one-sided shift spaces.
A $^*$-isomorphism $\Phi\colon \OO_\X\LRA \OO_\Y$ is \emph{diagonal-preserving} if $\Phi(\D_\X) = \D_\Y$.
In this section we prove that any $^*$-isomorphism $\Phi\colon \OO_\X\LRA \OO_\Y$ satisfying $\Phi(C(\X)) = C(\Y)$ is diagonal-preserving (Theorem~\ref{thm:diagonal-preserving}).
First we need some preliminary results.
Recall that~\cite[Proposition 1.9]{Phillips2005} allows us to consider $\mathrm{C^*}(\Iso(\G_\X)^\circ)$ as a subalgebra of $\mathrm{C}^*(\G_{\X}) = \OO_\X$.

\begin{lemma}\label{lem:isotropy-commutator}
  Let $\X$ be a one-sided shift space. 
  We have $\mathrm{C^*}(\Iso(\G_\X)^\circ) = \D_\X' = C(\X)'$.
  If the groupoid $\G_\X$ is effective, then $\D_\X = \D_\X'$.
\end{lemma}

\begin{proof}
  The identification of $\mathrm{C^*(\Iso(\G_\X)^\circ)}$ and $\D_\X'$ follows from \cite[Corollary 5.3]{CRST},
  and $\D_\X'\subset C(\X)'$ is a consequence of $C(\X) \subset \D_\X$.
  It remains to verify the inclusion $C(\X)' \subset \D_\X'$.

  Let $\xi\in C(\X)'$ and observe that
  \[
    \xi(\tilde{x}, p, \tilde{y}) g(\pi_\X(\tilde{y})) = (\xi\star g)(\tilde{x}, p, \tilde{y})
    = (g\star \xi)(\tilde{x}, p, \tilde{y}),
    = g(\pi_\X(\tilde{x})) \xi(\tilde{x}, p, \tilde{y}),
  \]
  for $(\tilde{x}, p, \tilde{y})\in \G_\X$ and all $g\in C(\X)$.
  It follows that $\xi(\tilde{x}, p, \tilde{y}) \neq 0$ only if $\pi_\X(\tilde{x}) = \pi_\X(\tilde{y})$.
  Similarly, if $\xi(\tilde{x}, p, \tilde{y}) \neq 0$ implies that $\tilde{x} = \tilde{y}$, 
  then $\xi\star f = f\star \xi$ for all $f\in \D_\X$, i.e., $\xi\in \D_\X'$.

  Suppose $\xi(\tilde{x}, p, \tilde{y}) \neq 0$, for some $(\tilde{x}, p, \tilde{y})\in \G_\X$.
  Then $x :=\pi_\X(\tilde{x}) = \pi_\X(\tilde{y})$ since $\xi\in C(\X)'$.
  We will show that $\tilde{x} = \tilde{y}$.
  Pick nonnegative integers $k, l\in \N$ such that $p = k - l$ and $\sigma_{\tilde{\X}}^k(\tilde{x}) = \sigma_{\tilde{\X}}^l(\tilde{y})$.
  If $k = l$ then Lemma~\ref{lem:isotropy-lemma}(i) implies that $\tilde{x} = \tilde{y}$, so we may assume that $k > l$.
  It follows that $\sigma_\X^k(x) = \sigma_\X^l(x)$ in $\X$ so $x = \mu \alpha^\infty$, where $\mu$ and $\alpha$ are words with $|\mu| = l$ and $|\alpha| = p$.
  The support of $\xi$ is open in $\G_\X$ so it contains an open bisection of the form $Z(\tilde{U}, k, l, \tilde{V})$, 
  where $\tilde{U}$ and $\tilde{V}$ are open sets in $\tilde{\X}$.
  Since $\tilde{U}$ and $\tilde{V}$ are open neighborhoods around $\tilde{x}$ and $\tilde{y}$, respectively,
  we may assume that $\pi_\X(\tilde{U}), \pi_\X(\tilde{V}) \subset Z(\mu \alpha)$, since $\pi_\X$ is continuous,
  Given any $\tilde{u}\in \tilde{U}$, there is a (unique) element $\gamma = (\tilde{u}, p, \tilde{v})\in Z(\tilde{U}, k, l, \tilde{V})$ and
  $\sigma_{\tilde{\X}}^k(\tilde{u}) = \sigma_{\tilde{\X}}^l(\tilde{v})$.
  Since $\xi(\gamma) \neq 0$ it follows that $\pi_\X(\tilde{u}) = \pi_\X(\tilde{v}) = x$.
  
  The set $\iota(\X)\cap \tilde{U}$ is dense in $\tilde{U}$, and whenever $\iota(t)\in \tilde{U}$ we see from the above argument that
  \[
    t = \pi_\X(\iota(t)) = x.
  \]
  It follows that $\iota(x)$ is dense in $\tilde{U}$ so $\tilde{U} = \left\{ \iota(x) \right\}$.
  A similar argument shows that $\tilde{V} = \left\{ \iota(x) \right\}$.
  Hence $(\tilde{x}, p, \tilde{y}) = (\iota(x), p, \iota(x))$, so $\xi\in \D_\X'$.
   
  Finally, if $\G_\X$ is effective, then $\D_\X$ is maximally abelian in $\OO_\X$ so $\D_\X = \D_\X'$.
\end{proof}

Consider the equivalence relation $\sim$ on the space $\tilde{\X}\times \T$ given by
$(\tilde{x},\zeta) \sim (\tilde{y}, \theta)$ if and only if $\tilde{x} = \tilde{y}$ and $\zeta^p = \theta^p$ for all $p\in \Stab(\tilde{x})$.
Then the quotient $\tilde{\X}\times \T/\sim$ is compact and Hausdorff and as we shall see (homeomorphic to) the spectrum of $\mathrm{C}^*(\Iso(\G_{\X})^{\circ})$.

\begin{lemma}\label{lem:isotropy-spectrum}
    Let $\sim$ be the equivalence relation on $\tilde{\X}\times \T$ defined above.
    There is a $^*$-isomorphism $\Xi\colon \mathrm{C^*}(\Iso(\G_{\X})^{\circ}) \LRA C(\tilde{\X}\times \T/\sim)$ given by 
    \begin{equation} \label{eq:isotropy-spectrum}
        \Xi(f)([\tilde{x},\zeta]) = \sum_{p\in \Stab(\tilde{x})} f(\tilde{x}, p, \tilde{x}) \zeta^n,
    \end{equation}
    for $f\in C_c(\Iso(\G_{\X})^{\circ})$ and $[\tilde{x}, \zeta]\in \tilde{\X}\times \T/\sim$.
\end{lemma}

\begin{proof}
    The map $\Xi\colon C_c(\Iso(\G_\X)^\circ) \LRA C(\tilde{\X}\times \T/\sim)$ given in~\eqref{eq:isotropy-spectrum} 
    is well-defined by the definition of $\sim$ and linear.
    If $f,g\in C_c(\Iso(\G_{\X})^{\circ})$ and $[\tilde{x}, z]\in \tilde{\X}\times \T/\sim$, then
    \begin{align*}
        \Xi(f)([\tilde{x},\zeta]) \Xi(g)([\tilde{x},\zeta]) 
        &= \sum_{k,l\in \Stab(\tilde{x})} f(\tilde{x}, k, \tilde{x}) g(\tilde{x}, l, \tilde{x}) \zeta^{k+l} \\
        &= \sum_{n,m\in \Stab(\tilde{x})} f(\tilde{x}, n-m, \tilde{x}) g(\tilde{x}, m, \tilde{x}) \zeta^n \\
        &= \Xi(f\star g)([\tilde{x}, \zeta]),
    \end{align*}
    so $\Xi$ is multiplicative.
    It is straightforward to see that $\Xi$ also respects the $^*$-involution.
    
    Next, we show that $\Xi$ is isometric.
    It follows from~\cite[Lemma 5.1]{CRST} that $\mathrm{C^*}(\Iso(\G_\X)^\circ)$ is a $\D_\X$-algebra.
    In particular, 
    \[
      || f || = \sup_{\tilde{x}\in \tilde{\X}} || \pi_{\tilde{x}}(f) ||
    \]
    for $f\in \mathrm{C^*}(\Iso(\G_\X)^\circ)$,
    where $\pi_{\tilde{x}}\colon \mathrm{C^*}(\Iso(\G_\X)) \LRA \mathrm{C^*}(\Iso(\tilde{x}))$ is the $^*$-homomorphism given by
    \[
      \pi_{\tilde{x}}(f)(\tilde{x}, p, \tilde{x}) = f(\tilde{x}, p, \tilde{x}),
    \]
    for $(\tilde{x}, p, \tilde{x})\in \Iso(\tilde{x})$.
    Since $\Iso(\tilde{x})$ is (isomorphic to) the integers or the trivial group, we have
    \[
      || \pi_{\tilde{x}}(f) || = \sup_{\zeta\in \T} |\sum_{p\in \Stab(\tilde{x})} f(\tilde{x}, p, \tilde{x}) \zeta^p|
    \]
    from which it follows that $\Xi$ is isometric.

    We show that $\Xi$ separates points.
    First, if $[\tilde{x}, \zeta] \neq [\tilde{x}, \theta]$, then there is $p\in \Iso(\tilde{x})$ such that $\zeta^p \neq \theta^p$.
    Choose a compact open bisection $U \subset \G_\X$ satisfying $U\cap \Iso(\tilde{x}) = \{(\tilde{x}, p, \tilde{x})\}$ 
    and observe that $\Xi(1_U)([\tilde{x}, \zeta]) = \zeta^p$ and $\xi(1_U)([\tilde{x}, \theta]) = \theta^p$.
    Second, if $\tilde{x} \neq \tilde{y}$ in $\tilde{\X}$ 
    then we choose a compact open bisection $U$ satisfying $(\tilde{x}, 0, \tilde{x})\in U$ and $\Iso(\tilde{y})\cap U = \emptyset$.
    Then $\Xi(1_U)([\tilde{x}, \zeta]) = 1$ while $\Xi(1_U)([\tilde{y}, \theta]) = 0$.
    By the Stone--Weierstrass theorem, the image of $\Xi$ is dense in $C(\tilde{\X}\times \T/\sim)$
    and $\Xi$ thus extends to a $^*$-isomorphism as wanted.
\end{proof}

\begin{theorem}\label{thm:diagonal-preserving}
  Let $\X$ and $\Y$ be one-sided shift spaces.
  If $\Phi\colon \OO_\X\LRA \OO_\Y$ is a $^*$-isomorphism satisfying $\Phi(C(\X)) = C(\Y)$, then $\Phi(\D_\X) = \D_\Y$.
\end{theorem}

\begin{proof}
  If $\Phi\colon \OO_\X\LRA \OO_\Y$ is a $^*$-isomorphism satisfying $\Phi(C(\X)) = C(\Y)$, then $\Phi(C(\X)') = C(\Y)'$. 
  It follows from Lemmas~\ref{lem:isotropy-commutator} and~\ref{lem:isotropy-spectrum} that there is a homeomorphism 
    \[
        h\colon \tilde{\X}\times \T/\sim \LRA \tilde{\Y}\times \T/\sim
    \]
  such that $\Xi_\Y(\Phi(f)) = \Xi_\X(f)\circ h^{-1}$ for $f\in \mathrm{C^*}(\Iso(\G_{\X})^{\circ})$. 
  We see from \eqref{eq:isotropy-spectrum} that if $f\in \mathrm{C^*}(\Iso(\G_{\X})^{\circ})$, 
  then $f\in\D_\X$ if and only if $\Xi_\X(f)([\tilde{x}, z])=\Xi_\X(f)([\tilde{x}, 1])$ for all $\tilde{x}\in\tilde{\X}$ and all $z\in\T$. 
  Since $\X$ is a totally disconnected space, the connected component of $[\tilde{x}, z]\in \tilde{\X}\times \T/\sim$ is the set $\{[\tilde{x}, w] \mid w\in \T\}$. 
  We thus have that $f\in \mathrm{C^*}(\Iso(\G_{\X})^{\circ})$ belongs to $\D_\X$ if and only if $\Xi_\X(f)$ is constant on connected components. 
  Similarly, $\Phi(f)\in\D_\Y$ if and only if $\Xi_\Y(\Phi(f))$ is constant on connected components. 
  Since $h$ is a homeomorphism, it maps connected components onto connected components.
  We conclude that $\Phi(\D_\X) = \D_\Y$.
\end{proof}
   
\begin{remark}\label{rem:sofic}
  Let $\X$ be a strictly sofic one-sided shift space and let $\Y = \tilde{\X}$ be its cover.
  Then $\Y$ is (conjugate to) a shift of finite type.
  Although it is possible that $\X$ and $\Y$ are homeomorphic so that $C(\X)$ and $C(\Y)$ are $^*$-isomorphic,
  there is no $^*$-isomorphism $\Phi\colon \OO_\X \LRA \OO_\Y$ which satisfies $\Phi(C(\X)) = C(\Y)$.
  Indeed, if this were the case then Theorem~\ref{thm:diagonal-preserving} would imply that 
  \[
    \Phi(\D_\X) = \D_\Y = C(\Y) = \Phi(C(\X))
  \]
  so that $C(\X) = \D_\X$ inside $\OO_\X$, and $\pi_\X$ is a homeomorphism. 
  However, this is not possible when $\X$ is strictly sofic.
  Foreshadowing Theorem~\ref{thm:coe-cover} (below) this means that $\X$ and $\tilde{\X}$ do not admit a stabilizer-preserving continuous orbit equivalence.
\end{remark}

Below, we give a stabilized version of Theorem~\ref{thm:diagonal-preserving}.
Consider the product $\tilde{\X}\times \N\times \T$ equipped with the equivalence relation $\approx$ defined by
$(\tilde{x}, m_1, z) \approx (\tilde{y}, m_2, w)$ if and only if $\tilde{x} = \tilde{y}$ and $m_1 = m_2$ and $z^n = w^n$ for all $n\in \Iso(\tilde{x})$.
The spaces $\tilde{\X}\times \N\times \T/\approx$ and $(\tilde{\X}\times \T/\sim) \times \N$ are now homeomorphic.
An argument similar to the above then yields the following result.

\begin{corollary}\label{cor:diagonal-preserving-stable}
    Let $\X$ and $\Y$ be one-sided shift spaces
    and let $\Phi\colon \OO_\X\otimes \K\LRA \OO_\Y\otimes \K$ be a $^*$-isomorphism satisfying $\Phi(C(\X)\otimes c_0) = C(\Y)\otimes c_0$.
    Then $\Phi(\D_\X\otimes c_0) = \D_\Y\otimes c_0$.
\end{corollary}

\section{One-sided conjugacy}\label{sec:one-sided-conjugacy}


A pair of one-sided shift space $\X$ and $\Y$ are one-sided conjugate if there exists a homeomorphism 
$h\colon \X\LRA \Y$ satisfying $h\circ \pi_\X = \pi_\Y\circ h$.
A similar definition applies to the covers.
If $\X$ and $\Y$ are shifts of finite type, then they are conjugate if and only if the groupoids $\G_\X$ and $\G_\Y$ are isomorphic
in a way which preserves a certain endomorphism, if and only if the $\mathrm{C^*}$-algebras $\OO_\X$ and $\OO_\Y$ are $^*$-isomorphic
in a way which preserves a certain completely positive map~\cite[Theorem 3.3]{BC}.
In this section we characterize one-sided conjugacy of general one-sided shift spaces (Theorem~\ref{thm:one-sided-conjugacy}).

We start by lifting a one-sided conjugacy on the shift spaces to a conjugacy on the covers.
The cover construction is therefore \emph{canonical}, cf.~\cite[Theorem 2.13]{Krieger1984}.

\begin{lemma}[Lifting lemma]\label{lem:canonical_cover}
    Let $\X$ and $\Y$ be one-sided shift spaces and let $h\colon \X\LRA \Y$ be a homeomorphism.
    The following are equivalent:
    \begin{enumerate}
        \item[(i)] the map $h\colon \X\LRA \Y$ is a conjugacy;
        \item[(ii)] there is a conjugacy $\tilde{h}\colon \tilde{\X}\LRA \tilde{\Y}$ satisfying $h\circ \pi_{\X} = \pi_{\Y}\circ \tilde{h}$.
    \end{enumerate}
\end{lemma}

\begin{proof}
    (i) $\implies$ (ii):
    Let $h\colon \X\LRA \Y$ be a conjugacy and choose an integer $C\in \N$ such that
    \[
        x_{[0, C + r)} = x'_{[0, C + r)} \implies h(x)_{[0, r)} = h(x')_{[0, r)}
    \]
    for $r\in \N$ and $x, x'\in \X$.
    Given integers $0\leq r\leq s$, we show that
    \[
        \alpha x \widesim{C + r, C + s} \alpha x' \implies h(\alpha x) = h(\alpha x')
    \]
    for $\alpha x, \alpha x'\in \X$ with $|\alpha| = r$.
    Start by writing $h(\alpha x) = \mu y$ and $h(\alpha x') = \mu y'$ for some $y\in \Y$ and $\mu\in \LL(\Y)$ with $|\mu| = s$ 
    and observe that $h(x) = y$ and $h(x') = y'$ since $h$ is a conjugacy.
    Assume now that $\nu y\in \Y$ for some $\nu\in \LL(\Y)$ with $|\nu| \leq s$.
    We need to show that $\nu y'\in \Y$.

    Observe that $h^{-1}(\nu y) = \beta_\nu x$ for some $\beta_\nu\in \LL(\X)$ with $|\beta_\nu| = |\nu| \leq s$
    and, by hypothesis, $\beta_\nu x'\in \X$.
    It is now easily verified that $h(\beta_\nu x') = \nu y'$ so that $\nu y'\in \Y$ as wanted.
    
    This defines an induced map $\tilde{h}\colon \tilde{\X}\LRA \tilde{\Y}$ determined by
    \[
        \tilde{h}\colon {}_{C + r} [x]_{C + s} \LMT {}_r[h(x)]_s
    \]
    for integers $0\leq r\leq s$.
    It is readily verified that $\tilde{h}$ is a conjugacy satisfying $h\circ \pi_\X = \pi_\Y\circ \tilde{h}$ using that $h$ is a conjugacy.

    (ii) $\implies$ (i):
    Given $x\in \X$ and any $\tilde{x}\in \pi_{\X}^{-1}(x)\subset \tilde{\X}$, we observe that
    \[
        h(\sigma_{\X}(x)) = \pi_{\Y}(\tilde{h}(\sigma_{\tilde{\X}}(\tilde{x})))
        = \pi_{\Y}(\sigma_{\tilde{\Y}}(\tilde{h}(\tilde{x})))
        = \sigma_{\Y}h(x).
    \]
    This shows that $h$ is a conjugacy.
\end{proof}

Let $\X$ be a one-sided shift and let $\G_{\X}$ be the groupoid defined in Section~\ref{sec:approach}.
The map $\epsilon_{\X}\colon \G_{\X}\LRA \G_{\X}$ given by
\[
    \epsilon_{\X}(\tilde{x}, p, \tilde{y}) = (\sigma_{\tilde{\X}}(\tilde{x}), p, \sigma_{\tilde{\X}}(\tilde{y})),
\]
for $(\tilde{x}, p, \tilde{y})\in \G_{\X}$, is a continuous groupoid homomorphism.
There is an induced homomorphism $\epsilon_{\X}^*\colon C_c(\G_{\X})\LRA C_c(\G_{\X})$ given by $\epsilon_{\X}^*(f) = f\circ \epsilon_{\X}$,
for $f\in C_c(\G_{\X})$.
We also consider two completely positive maps on $\OO_{\X}$ as follows:
Let $\{s_a\}_{a\in \A}$ be the canonical generators of $\OO_{\X}$ and consider $\phi_{\X}\colon \OO_{\X}\LRA \OO_{\X}$ given by
\begin{equation}
    \phi_{\X}(y) = \sum_{a\in \A} s_a y s_a^*,
\end{equation}
for $y\in \OO_{\X}$, and map $\tau_{\X}\colon \OO_{\X}\LRA \OO_{\X}$ given by
\begin{equation}
    \tau_{\X}(y) = \sum_{a,b\in \A} s_b y s_a^*,
\end{equation}
for $y\in \OO_{\X}$.
The next lemma describes the relationship between these maps.
Recall that $p_\X\colon \OO_\X\LRA \D_\X$ is the conditional expectation onto the diagonal subalgebra induced by the inclusion $\tilde{\X} \LRA \G_\X$.
The proof of the lemma below is a straightforward computation, cf.~\cite[Lemma 3.1]{BC}.

\begin{lemma}
    We have $\tau_{\X}(f) = f\circ \epsilon_{\X}$ for $f\in C_c(\G_{\X})$.
    Hence $\tau_{\X}$ extends $\epsilon_{\X}^*$ to $\OO_{\X}$.
    Furthermore, $p_\X\circ \tau_\X|_{\D_\X} = \phi_\X|_{\D_\X}$.
\end{lemma}

For the next lemma, recall that $\F_\X$ is the AF core inside $\OO_\X$.
A similar result was presented in~\cite[Lemma 3.2]{BC} but we include a proof for the sake of completeness.

\begin{lemma}\label{lem:AFcore-preserving}
    Let $\X$ and $\Y$ be one-sided shift spaces.
    If $\Phi\colon \F_\X\LRA \F_\Y$ is a $^*$-isomorphism satisfying $\Phi(\D_\X) = \D_\Y$,
    then $\Phi\circ p_\X(f) = p_\Y\circ \Phi(f)$ on $\F_\X$.
    If, in addition, $\Phi\circ \tau_\X|_{\F_\X} = \tau_\Y\circ \Phi$, then $\Phi\circ \phi_\X = \phi_\Y\circ \Phi$ on $\D_\X$.
\end{lemma}

\begin{proof}
  Recall that $\F_\X = \mathrm{C^*}(c_\X^{-1}(0))$ and the subgroupoid $c_\X^{-1}(0) \subset \G_\X$ is principal.
  By~\cite[Proposition 4.13]{Renault08} (see also \cite[Theorem 3.3]{CRST}) and \cite[Proposition 5.7]{Ma2012a} and its proof, 
  there is a groupoid isomorphism $\kappa\colon c_{\Y}^{-1}(0)\LRA c_{\X}^{-1}(0)$ and a groupoid homomorphism $\xi\colon c_{\X}^{-1}(0)\LRA \T$ such that
    \[
        \Phi(f)(\gamma) = \xi(\kappa(\gamma)) f(\kappa(\gamma)),
    \]
    for $f\in \F_{\X}$ and $\gamma\in c_{\Y}^{-1}(0)$. 
    In particular, $\Phi(f)(\tilde{y}) = f(\kappa(\tilde{y}))$ for $\tilde{y}\in \tilde{\Y}$
    since $\xi$ maps any unit in $c_\X^{-1}(0)$ to $1\in \T$.
    Then 
    \[
      \Phi(p_\X(f))(\tilde{y}) = p_\X(f)(\kappa(\tilde{y})) = f(\kappa(\tilde{y})),
    \]
    so that $\Psi\circ d_\X = d_\Y\circ \Psi$ on $\F_\X$. 
    If, in addition, $\Psi\circ \tau_{\X}|_{\F_{\X}} = \tau_{\Y}\circ \Psi$, then
    \[
        \Phi(\phi_{\X}(f)) = \Phi(d_\X(\tau_{\X}(f))) = d_\Y(\tau_\Y(\Phi(f))) = \phi_{\Y}(\Phi(f)),
    \]
    for $f\in \D_\X$ by the above lemma.
\end{proof}

We now characterize one-sided conjugacy of general one-sided shift spaces.

\begin{theorem}\label{thm:one-sided-conjugacy}
    Let $\X$ and $\Y$ be one-sided shift spaces and let $h\colon \X\LRA \Y$ be a homeomorphism.
    The following are equivalent:
    \begin{enumerate}
        \item[(i)] the map $h\colon \X\LRA \Y$ is a one-sided conjugacy;
        \item[(ii)] there is a conjugacy $\tilde{h}\colon \tilde{\X}\LRA \tilde{\Y}$ satisfying $h\circ \pi_\X = \pi_\Y\circ \tilde{h}$;
        \item[(iii)] there is a groupoid isomorphism $\Psi\colon \G_\X\LRA \G_\Y$ satisfying $h\circ \pi_\X = \pi_\Y\circ \Psi^{(0)}$,
            $c_\X = c_\Y\circ \Psi$ and 
	    \begin{equation}\label{eq:epsilon-flet}
	      \Psi\circ \epsilon_\X = \epsilon_\Y\circ \Psi;
	    \end{equation}
        \item[(iv)] there is a groupoid isomorphism $\Psi\colon \G_\X\LRA \G_\Y$ satisfying $h\circ \pi_\X = \pi_\Y\circ \Psi^{(0)}$ and
	  \[
	    \Psi\circ \epsilon_\X = \epsilon_\Y\circ \Psi;
	  \]
        \item[(v)] there is a $^*$-isomorphism $\Phi\colon \OO_\X\LRA \OO_\Y$ satisfying $\Phi(C(\X)) = C(\Y)$ with $\Phi(g) = g\circ h^{-1}$ for $g\in C(\X)$,
            $\Phi\circ p_\X = p_\Y\circ \Phi$, 
            $\Phi\circ \gamma^\X_z = \gamma^\Y_z\circ \Phi$ for $z\in \T$,
            $\Phi\circ \phi_\X|_{\D_\X} = \phi_\Y\circ \Phi|_{\D_\X}$, and
            \begin{align}
                \Phi\circ \tau_\X = \tau_\Y\circ \Phi;
                \label{eq:tau-flet}
            \end{align}
        \item[(vi)] there is a $^*$-isomorphism $\Phi\colon \OO_\X\LRA \OO_\Y$ satisfying  
            $\Phi(C(\X)) = C(\Y)$ with $\Phi(g) = g\circ h^{-1}$ for $g\in C(\X)$, and
	    \[
	      \Phi\circ \tau_\X = \tau_\Y\circ \Phi; 
	    \]
        \item[(vii)] there is a $^*$-isomorphism $\Omega\colon \D_\X \LRA \D_\Y$ satisfying $\Omega(C(\X)) = C(\Y)$, $\Omega(g) = g\circ h^{-1}$ for $g\in C(\X)$
            and $\Omega\circ \phi_{\X}|_{\D_\X} = \phi_{\Y}\circ \Omega$.
    \end{enumerate}
\end{theorem}

\begin{proof}
    The equivalence (i)$\iff$(ii) is Lemma~\ref{lem:canonical_cover}.

    (ii) $\implies$ (iii):
    Let $\tilde{h}\colon \tilde{\X}\LRA \tilde{\Y}$ be a conjugacy satisfying $h\circ \pi_\X = \pi_\Y\circ \tilde{h}$.
    The map $\Phi\colon \G_{\X}\LRA \G_{\Y}$ given by
    \[
        \Phi(\tilde{x}, p, \tilde{y}) = (\tilde{h}(\tilde{x}), p, \tilde{h}(\tilde{y})),
    \]
    for $(\tilde{x}, p, \tilde{y})\in \tilde{\X}$, is a groupoid isomorphism. 
    Under the identification of $\Phi^{(0)}$ and $\tilde{h}$, we then have $\pi_\Y\circ \Psi^{(0)} = h\circ \pi_\X$, 
    $c_\X = c_\Y\circ \Psi$ and $\Psi\circ \epsilon_\X = \epsilon_\Y\circ \Psi$.

    The implications (iii) $\implies$ (iv) and (v) $\implies$ (vi) are clear.

    (iv) $\implies$ (vi) and (iii) $\implies$ (v):
    Let $\Psi\colon \G_\X\LRA \G_\Y$ be a groupoid isomorphism as in (iv).
    This induces a $^*$-isomorphism $\Phi\colon \OO_\X\LRA \OO_\Y$ satisfying $\Phi\circ p_\X = p_\Y\circ \Phi$
    and $\Phi(C(\X)) = C(\Y)$ with $\Phi(g) = g\circ h^{-1}$ for $g\in C(\X)$.
    The relation~\eqref{eq:epsilon-flet} ensures that $\Phi\circ \tau_\X = \tau_\Y\circ \Phi$.
    This is (vi).
    If, in addition, $c_\X = c_\Y\circ \Phi$, then $\Phi\circ \gamma^\X_z = \gamma^\Y_z\circ \Phi$ for $z\in \T$.
    In particular, $\Psi(\F_{\X}) = \F_{\Y}$ and Lemma~\ref{lem:AFcore-preserving} implies that $\Phi\circ \phi_\X|_{\D_\X} = \phi_\Y\circ \Psi|_{\D_\X}$.
    This is (v).

    (vi) $\implies$ (vii):
    As $\Phi$ satisfies~\eqref{eq:tau-flet} and $\F_\X$ is generated as a $\mathrm{C^*}$-algebra by $\bigcup_{k = 0}^\infty \tau_\X^k(\D_\X)$,
    we also have $\Phi(\F_\X) = \F_\Y$.
    Furthermore, $\Phi(\D_\X) = \D_\Y$ by Theorem~\ref{thm:diagonal-preserving}.
    It therefore follows from Lemma~\ref{lem:AFcore-preserving} that $\Phi\circ \phi_\X|_{\D_\X} = \phi_\Y\circ \Phi|_{\D_\X}$.

    (vii) $\implies$ (ii):
    Let $\tilde{h}\colon \tilde{\X}\LRA \tilde{\Y}$ be the homeomorphism induced by $\Omega$ via Gelfand duality.
    The relation $\Omega\circ \phi_\X|_{C(\X)} = \phi_\Y\circ \Omega$ 
    and the fact that $\phi_\X(f)(\tilde{x}) = f(\sigma_{\tilde{\X}}(\tilde{x}))$ for $f\in \D_\X$ and $\tilde{x}\in \tilde{\X}$ 
    ensures that $\tilde{h}$ is a conjugacy.
    The condition $\Omega(C(\X)) = C(\Y)$ entails that $h\circ \pi_\X = \pi_\Y\circ \tilde{h}$.
\end{proof}

\begin{corollary}
    Let $\X$ and $\Y$ be one-sided shift spaces.
    The following are equivalent:
    \begin{enumerate}
        \item[(i)] the systems $\X$ and $\Y$ are one-sided conjugate;
        \item[(ii)] there are a groupoid isomorphism $\Psi\colon \G_\X\LRA \G_\Y$ and a homeomorphism $h\colon \X\LRA \Y$ 
            satisfying $h\circ \pi_\X = \pi_\Y\circ \Psi^{(0)}$ and $\Psi\circ \epsilon_\X = \epsilon_\Y\circ \Psi$;
        \item[(iii)] there is a $^*$-isomorphism $\Phi\colon \OO_\X\LRA \OO_\Y$ satisfying  
            $\Phi(C(\X)) = C(\Y)$ and $\Phi\circ \tau_\X = \tau_\Y\circ \Phi$.
    \end{enumerate}
\end{corollary}

If $\X_A$ and $\X_B$ are one-sided shifts of finite type determined by finite square $\{0,1\}$-matrices $A$ and $B$ with no zero rows and no zero columns, 
respectively, then we recover~\cite[Theorem 3.3]{BC}.

\section{One-sided eventual conjugacy}\label{sec:eventual-conjugacy}
Matsumoto has studied one-sided eventual conjugacy of shifts of finite type~\cite{Mat2017-circle}.
A pair of shifts of finite type $\X$ and $\Y$ are eventually conjugate if and only if
the groupoids $\G_\X$ and $\G_\Y$ are isomorphic in a way which preserves the canonical cocycle,
if and only if the $\mathrm{C^*}$-algebras $\OO_\X$ and $\OO_\Y$ are $^*$-isomorphic in a way which preserves the canonical gauge actions, cf.~\cite[Theorem 4.1 and Corollary 4.2]{Carlsen-Rout}.
We characterize eventual conjugacy for general shift spaces in terms of groupoids and $\mathrm{C^*}$-algebras (Theorem~\ref{thm:eventual-conjugacy}).
We start by lifting an eventual conjugacy on the shift spaces to an eventual conjugacy on the covers.

\begin{definition}
    Two one-sided shift spaces $\X$ and $\Y$ are \emph{eventually conjugate}
    if there exist a homeomorphism $h\colon \X\LRA \Y$ and an integer $\ell \in \N$ such that
   \begin{align}
       \sigma_\Y^\ell (h(\sigma_\X(x))) &= \sigma_\Y^{\ell + 1} h(x),               \label{eq:eventual-conjugacy1} \\
       \sigma_\X^\ell (h^{-1}(\sigma_\Y(y))) &= \sigma_\X^{\ell + 1} h^{-1}(y),     \label{eq:eventual-conjugacy2}
   \end{align}
   for $x\in \X$ and $y\in \Y$.
   An eventual conjugacy $h$ is a conjugacy if and only if we can choose $\ell = 0$.
\end{definition}

A similar definition applies to the covers.

\begin{lemma}[Lifting lemma]\label{lem:lift-eventual-conjugacy}
    Let $\X$ and $\Y$ be one-sided shift spaces and let $h\colon \X\LRA \Y$ be a homeomorphism.
    The following are equivalent:
    \begin{enumerate}
        \item[(i)] the map $h\colon \X\LRA \Y$ is an eventual conjugacy;
        \item[(ii)] there is an eventual conjugacy $\tilde{h}\colon \tilde{\X}\LRA \tilde{\Y}$ satisfying $h\circ \pi_\X = \pi_\Y\circ \tilde{h}$.
    \end{enumerate}
\end{lemma}

\begin{proof}
    (i) $\implies$ (ii):
    Let $h\colon \X\LRA \Y$ be an eventual conjugacy and choose $\ell\in \N$ according to~\eqref{eq:eventual-conjugacy1} and~\eqref{eq:eventual-conjugacy2}.
    Then there is a continuity constant $C\in \N$ with the property that
    \[
        x_{[0, C + r)} = x'_{[0, C + r)} \implies h(x)_{[0, \ell + r)} = h(x')_{[0, \ell + r)},
    \]
    for $x, x'\in \X$ and $r\in \N$.
    Fix integers $0\leq r\leq s$ and put $K = C + 2\ell + s$. 
    We will show that
    \[
        \alpha x \widesim{K + r, K + s} \alpha x' \implies h(\alpha x) \widesim{r, s} h(\alpha x'),
    \]
    where $|\alpha| = \ell + r$.
    Since $K \geq C$, we can write $h(\alpha x) = \mu y$ and $h(\alpha x') = \mu y'$ for some $y, y'\in \Y$ and $\mu\in \LL(\Y)$ with $|\mu| = r$.
    In particular, $y_{[0, 2\ell)} = y'_{[0, 2\ell)}$.
    Assume now that $\nu y\in \Y$ for some $\nu\in \LL(\Y)$ with $|\nu| \leq s$.
    We need to show that $\nu y'\in \Y$.

    First observe that $h^{-1}(\nu y) = \beta_\nu x$ for some word $\beta_\nu\in \LL(\X)$ with $|\beta_\nu| = \ell + |\nu|$.
    This follows from the computation 
    \[
        x = \sigma_\X^{\ell + r}(\alpha x) =  \sigma_\X^{\ell + r}(h^{-1}(\mu y)) = \sigma_\X^\ell(h^{-1}(y)) = \sigma_\X^{\ell + |\nu|}(h^{-1}(\nu y)).
    \]
    By hypothesis, $\beta_\nu x'\in \X$ and we claim that $h(\beta_\nu x') = \nu y'$.

    In order to verify the claim first observe that
    \[
        \sigma_\Y^{2\ell + |\nu|} (h(\beta_\nu x')) = \sigma_\Y^\ell (h(x')) = \sigma_\Y^{2\ell + r} (h(\alpha x')) = \sigma_\Y^{2\ell}(y'),
    \]
    and since ${(\beta_\nu x)}_{[0, C + 2\ell + |\nu|)} = {(\beta_\nu x')}_{[0, C + 2\ell + |\nu|)}$ we have
    \[
        {h(\beta_\nu x)}_{[0, 2\ell + |\nu|)} = {h(\beta_\nu x')}_{[0, 2\ell + |\nu|)}  = {(\nu y)}_{[0, 2\ell + |\nu|)} = {(\nu y')}_{[0, 2\ell + |\nu|)}.
    \]
    Hence $h(\beta_\nu x') = \nu y'$.
    This shows that there is a well-defined map $\tilde{h}\colon \tilde{\X}\LRA \tilde{\Y}$ given by
    \[
        \tilde{h}\colon {}_{K + r}[x]_{K + s} \LMT {}_r [h(x)]_s,
    \]
    for all integers $0\leq r\leq s$.
    It is straightforward to check that $h\circ \pi_\X = \pi_\Y\circ \tilde{h}$ and that 
    $\tilde{h}$ is an eventual conjugacy using the fact that $h$ is an eventual conjugacy.

    (ii) $\implies$ (i): 
    Given $x\in \X$ and any $\tilde{x}\in \pi_\X^{-1}(x) \subset \tilde{\X}$, we have
    \[
        \sigma_\Y^\ell( h( \sigma_\X( x))) =
        \pi_\Y( \sigma_{\tilde{\Y}}^\ell ( \tilde{h}( \sigma_{\tilde{\X}}( \tilde{x})))) =
        \pi_\Y( \sigma_{\tilde{\Y}}^{\ell + 1}( \tilde{h}( \tilde{x}))) =
        \sigma_\Y^{\ell + 1}( h( x)),
    \]
    showing that $h$ is an $L$-conjugacy.
\end{proof}

We can now characterize one-sided eventual conjugacy of general one-sided shifts spaces, cf.~\cite[Theorem 1.4]{Mat2019}.
The proof uses ideas of~\cite{Carlsen-Rout}.

\begin{theorem}\label{thm:eventual-conjugacy}
    Let $\X$ and $\Y$ be one-sided shift spaces and let $h\colon \X\LRA \Y$ be a homeomorphism. 
    The following are equivalent:
    \begin{enumerate}
        \item[(i)] the map $h\colon \X\LRA \Y$ is a one-sided eventual conjugacy;
        \item[(ii)] there is an eventual conjugacy $\tilde{h}\colon \tilde{\X}\LRA \tilde{\Y}$ such that $h\circ \pi_\X = \pi_\Y\circ \tilde{h}$;
        \item[(iii)] there is a groupoid isomorphism $\Psi\colon \G_\X\LRA \G_\Y$ satisfying $h\circ \pi_\X = \pi_\Y\circ \Psi^{(0)}$
            and 
            \[
                c_\X = c_\Y\circ \Psi;
            \]
        \item[(iv)] there is a $^*$-isomorphism $\Phi\colon \OO_\X\LRA \OO_\Y$ satisfying $\Phi\circ p_\X = p_\Y\circ \Phi$,
            $\Phi(C(\X)) = C(\Y)$ with $\Phi(g) = g\circ h^{-1}$ for $g\in C(\X)$ and
            \begin{align}\label{eq:gauge-invariance}
                 \Phi \circ \gamma_z^\X = \gamma_z^\Y\circ \Phi,
            \end{align}
            for $z\in \T$;
        \item[(v)] there is a $^*$-isomorphism $\Phi\colon \OO_\X\LRA \OO_\Y$ satisfying  
            $\Phi(C(\X)) = C(\Y)$ with $\Phi(g) = g\circ h^{-1}$ for $g\in C(\X)$
            and~\eqref{eq:gauge-invariance}.
    \end{enumerate}
\end{theorem}

\begin{proof}
    The equivalence (i) $\iff$ (ii) is Lemma~\ref{lem:lift-eventual-conjugacy}.
  
    (ii) $\implies$ (iii):
    Let $\tilde{h}\colon \tilde{\X}\LRA \tilde{\Y}$ be an eventual conjugacy satisfying $h\circ \pi_\X = \pi_\Y\circ \tilde{h}$.
    The map $\Psi\colon \G_\X \LRA \G_\Y$ given by
    \[
        \Psi(\tilde{x}, p, \tilde{y}) = (\tilde{h}(\tilde{x}), p, \tilde{h}(\tilde{y}))
    \]
    for $(\tilde{x}, p, \tilde{y})\in \G_\X$ is a groupoid isomorphism.
    Under the identification $\Psi^{(0)} = \tilde{h}$, we have $h\circ \pi_\X = \pi_\Y\circ \Psi^{(0)}$ and $c_\X = c_\Y\circ \Psi$.

    (iii) $\implies$ (ii):
    Let $\Psi\colon \G_\X \LRA \G_\Y$ be a groupoid isomorphism satisfying $h\circ \pi_\X = \pi_\Y\circ \Psi^{(0)}$ and $c_\X = c_\Y\circ \Psi$.
    Identify $\tilde{\X} = \G_\X^{(0)}$, $\tilde{\Y} = \G_\Y^{(0)}$ and $\tilde{h} = \Psi^{(0)}$.
    Then $\Psi$ is of the form
    \[
        \Psi( \tilde{x}, p, \tilde{y}) = (\tilde{h}(\tilde{x}), p, \tilde{h}(\tilde{y})),
    \]
    for $(\tilde{x}, p, \tilde{y})\in \G_\X$,
    and $h\circ \pi_\X = \pi_\Y\circ \tilde{h}$.
    Let $\A$ be the alphabet of $\X$ and consider the compact open bisection
    \[
        A_a = ( \sigma_{\tilde{\X}}(U_a), 0, 1, U_a)
    \]
    for $a\in \A$.
    Here, $U_a = \bigcup_{x\in Z_\X(a)} U(x, 1,1)$ in $\tilde{\X}$.
    Then
    \[
        \Psi(A_a) = \{ (\tilde{h}(\sigma_{\tilde{\X}}(\tilde{x})), -1, \tilde{h}(\tilde{x})) \mid \tilde{x}\in U_a \}
    \]
    is compact and open and contained in $c_Y^{-1}(\{ - 1\})$.
    Therefore
    \[
        \Psi(A_a) = \bigcup_{j = 1}^n (V_j, k_j, k_j + 1, W_j),
    \]
    for some $n\in \N$ and some compact open and mutually disjoint subsets $V_1,\ldots, V_n$,
    and compact open and mutually disjoint subsets $W_1,\ldots,W_n$ of $\tilde{\X}$
    and integers $k_1,\ldots,k_n\in \N$.
    In particular, $\tilde{h}^{-1}(U_a)$ is the disjoint union $\tilde{h}^{-1}(U_a) = \bigcup_{j = 1}^n \tilde{h}^{-1}(W_j)$ and
    \[
        \sigma_{\tilde{\Y}}^{k_j + 1}(\tilde{h}(\tilde{x})) = \sigma_{\tilde{\Y}}^{k_j} (\tilde{h}(\sigma_{\tilde{\X}}(\tilde{x}))) 
    \]
    for $\tilde{x}\in \tilde{h}^{-1}(W_j) \subset U_a$.
    We can now define a continuous map $k_a\colon U_a \LRA \N$ by $k_a(\tilde{x}) = k_j$ for $\tilde{x}\in \tilde{h}^{-1}(W_j) \subset U_a$.
    Since $\tilde{\X}$ is the disjoint union of $U_a$, $a\in \A$, there is a continuous map $k\colon \tilde{\X}\LRA \N$ 
    given by $k(\tilde{x}) = k_a(\tilde{x})$ for $\tilde{x}\in U_a \subset \tilde{\X}$, and
    \[
        \sigma_{\tilde{\Y}}^{k(\tilde{x}) + 1}(\tilde{h}(\tilde{x})) = \sigma_{\tilde{\Y}}^{k(\tilde{x})} (\tilde{h}(\sigma_{\tilde{\X}}(\tilde{x}))),
    \]
    for $\tilde{x}\in \tilde{\X}$.
    Similarly, there is a continuous map $k'\colon \tilde{\Y}\LRA \N$ which satisfies
    \[
        \sigma_{\tilde{\X}}^{k'(\tilde{y}) + 1}(\tilde{h}^{-1}(\tilde{y})) = \sigma_{\tilde{\X}}^{k'(\tilde{y})}(\tilde{h}^{-1}(\sigma_{\tilde{\Y}}(\tilde{y}))),
    \]
    for $\tilde{y}\in \tilde{\Y}$.
    Let $\ell = \max\{ k(\tilde{\X}), k'(\tilde{\Y}))$.
    Then $\tilde{h}$ is an eventual conjugacy satisfying $h\circ \pi_\X = \pi_\Y\circ \tilde{h}$.

    (iii) $\implies$ (iv):
    A groupoid isomorphism $\Psi\colon \G_\X\LRA \G_\Y$ with $h\circ \pi_\X = \pi_\Y\circ \Psi^{(0)}$ induces a $^*$-isomorphism $\Phi\colon \OO_\X\LRA \OO_\Y$
    satisfying $\Phi\circ p_\X = p_\Y\circ \Phi$ and $\Phi(C(\X)) = C(\Y)$ with $\Phi(g) = g\circ h^{-1}$ for $g\in C(\X)$.
    The relation $c_\X = c_\Y\circ \Psi$ ensures that $\Phi \circ \gamma_z^\X = \gamma_z^\Y\circ \Phi$ for $z\in \T$.

    The implication (iv) $\implies$ (v) is obvious.

    (v) $\implies$ (iii):
    By Theorem~\ref{thm:diagonal-preserving}, $\Phi(\D_\X) = \D_\Y$.
    The reconstruction theorem~\cite[Theorem 6.2]{CRST} ensures the existence of a groupoid isomorphism $\Psi\colon \G_\X\LRA \G_\Y$
    satisfying $\Phi(f) = f\circ \tilde{h}^{-1}$ for $f\in \D_\X$,
    where $\Psi^{(0)} = \tilde{h}\colon \tilde{\X}\LRA \tilde{\Y}$ is the induced homeomorphism on the unit spaces, 
    and $c_\X = c_\Y\circ \Phi$.
    Since $\Phi(C(\X)) = C(\Y)$ with $\Phi(g) = g\circ h^{-1}$ for $g\in C(\X)$, the groupoid isomorphism $\Psi$ satisfies $h\circ \pi_\X = \pi_\Y\circ \Psi^{(0)}$.
\end{proof}

\begin{corollary}
    Let $\X$ and $\Y$ be one-sided shift spaces.
    The following are equivalent:
    \begin{enumerate}
        \item[(i)] the systems $\X$ and $\Y$ are one-sided eventually conjugate;
        \item[(ii)] there exist a groupoid isomorphism $\Psi\colon \G_\X\LRA \G_\Y$ and a homeomorphism $h\colon \X\LRA \Y$ 
            satisfying $h\circ \pi_\X = \pi_\Y\circ \Psi^{(0)}$ and $c_\X = c_\Y\circ \Psi$;
        \item[(iii)] there is $^*$-isomorphism $\Phi\colon \OO_\X\LRA \OO_\Y$ satisfying $\Phi(C(\X)) = C(\Y)$ 
            and $\Phi \circ \gamma_z^\X = \gamma_z^\Y\circ \Phi$ for $z\in \T$.
    \end{enumerate}
\end{corollary}

When $\X_A$ and $\X_B$ are one-sided shifts of finite type determined by finite square $\{0,1\}$-matrices $A$ and $B$ with no zero rows and no zero columns, 
respectively, we recover~\cite[Corollary 4.2]{Carlsen-Rout} (see also~\cite[Theorem 1.2]{Mat2017-circle}).

\section{Continuous orbit equivalence}\label{sec:coe}

The notion of continous orbit equivalence among one-sided shift spaces was introduced by Matsumoto in~\cite{Mat2010, Mat2010b}.
It is proven in \cite[Corollary 6.1]{CEOR} (see also \cite[Theorem 2.3]{MM14}) that if $\X$ and $\Y$ are shifts of finite type, 
then $\X$ and $\Y$ are continuously orbit equivalent if and only if $\G_\X$ and $\G_\Y$ are isomorphic, 
and if and only if there is a $^*$-isomorphism $\OO_\X \LRA \OO_\Y$ which maps $C(\X)$ onto $C(\Y)$.
In this section, we shall for general shift spaces $\X$ and $\Y$ look at the relationship between continuous orbit equivalence of $\X$ and $\Y$, 
isomorphism of $\G_\X$ and $\G_\Y$, and $^*$-isomorphisms $\OO_\X \LRA \OO_\Y$ which map $C(\X)$ onto $C(\Y)$ (Theorems~\ref{thm:coe} and~\ref{thm:coe-cover}).
When the groupoids are effective, the results simplify (Theorem~\ref{thm:coe-essentially-principal}).

\begin{definition}\label{def:oe}
    Two one-sided shift spaces $\X$ and $\Y$ are \emph{continuously orbit equivalent} if there exist a homeomorphism $h\colon \X\LRA \Y$
    and continuous maps $k_\X, l_\X\colon \X\LRA \N$ and $k_\Y, l_\Y\colon \Y\LRA \N$ satisfying
    \begin{align}
        \sigma_\Y^{l_\X(x)}(h(x)) &= \sigma_\Y^{k_\X(x)}( h(\sigma_\X(x))), \label{eq:coe1} \\
        \sigma_\X^{l_\Y(y)}(h^{-1}(y)) &= \sigma_\X^{k_\Y(y)}( h^{-1}(\sigma_\Y(y))), \label{eq:coe2}
    \end{align}
    for $x\in \X$ and $y\in \Y$.
    The underlying homeomorphism $h$ is a \emph{continuous orbit equivalence} and $(k_\X, l_\X)$ and $(k_\Y, l_\Y)$ are \emph{cocycle pairs} for $h$.
\end{definition}

Similar definitions apply to the covers of one-sided shift spaces. 
Our first aim is to show that a continuous orbit equivalence between $\X$ and $\Y$ can be lifted to a continuous orbit equivalence between $\tilde{\X}$ and $\tilde{\Y}$.

Observe that if $h\colon \X\LRA \Y$ is an orbit equivalence with cocycles $k_\X, l_\X\colon \X\LRA \N$ and we define
\[
    k_\X^{(n)}(x) = \sum_{i = 0}^{n - 1} k_\X(\sigma_\X^i(x)), \qquad
    l_\X^{(n)}(x) = \sum_{i = 0}^{n - 1} l_\X(\sigma_\X^i(x)),
\]
then $\sigma_\Y^{l_\X^{(n)}} (h(x)) = \sigma_\Y^{k_\X^{(n)}}( h(\sigma_\X^n (x)) )$, for $x\in \X$.

We need some additional terminology.
Let $\X$ and $\Y$ be one-sided shift spaces and let $h\colon \X\LRA \Y$ be a continuous orbit equivalence with continuous cocycles $k_\X, l_\X\colon \X\LRA \N$ and $k_\Y, l_\Y\colon \Y\LRA \N$. We say that $(h,l_\X,k_\X,l_\Y,k_\Y)$ is \emph{least period preserving}
if $h$ maps eventually periodic points to eventually periodic points,
\[
    \lp(h(x)) = l_\X^{(p)}(x) - k_\X^{(p)}(x), 
\]
for any periodic point $x\in \X$ with $\lp(x) = p$, 
$h^{-1}$ maps eventually periodic points to eventually periodic points, and
\[
    \lp(h^{-1}(y)) = l_\Y^{(q)}(y) - k_\Y^{(q)}(y), 
\]
for any periodic point $y\in \Y$ with $\lp(y) = q$.
We say that $(h,l_\X,k_\X,l_\Y,k_\Y)$ is \emph{stabilizer-preserving}
if $h$ maps eventually periodic points to eventually periodic points,
\[
    \lp(h(x)) = |l_\X^{(p)}(x) - k_\X^{(p)}(x)|, 
\]
for any periodic point $x\in \X$ with $\lp(x) = p$, 
$h^{-1}$ maps eventually periodic points to eventually periodic points, and
\[
    \lp(h^{-1}(y)) = |l_\Y^{(q)}(y) - k_\Y^{(q)}(y)|, 
\]
for any periodic point $y\in \Y$ with $\lp(y) = q$.
cf.~\cite[p. 1093]{CEOR} and \cite[Definition 8.1]{CRST}.
There are analogous definitions for a continuous orbit equivalence between covers.

\begin{remark}
    Not every cocycle pair of a continuous orbit equivalence (even between finite type shifts) is least period preserving, cf.~Remark~\ref{rem:stab-preserving}.
    However, we do not know if there is a continuous orbit equivalence which does not admit a least period/stabilizer-preserving cocycle pair.
    In Example~\ref{ex:coe-notLPP}, we exhibit an example of a continuous orbit equivalence between shifts of finite type 
    which does not admit a cocycle pair which is least period preserving on all \emph{eventually periodic} points.
\end{remark}

\begin{lemma}[Lifting lemma]\label{lem:lift-coe}
    Let $\X$ and $\Y$ be one-sided shift spaces and let $h\colon \X\LRA \Y$ be a stabilizer-preserving continuous orbit equivalence
    with continuous cocycles $l_\X, k_\X\colon \X\LRA \N$ and $l_\Y, k_\Y\colon \Y\LRA \N$.
    Then there is a stabilizer-preserving continuous orbit equivalence $\tilde{h}\colon \tilde{\X}\LRA \tilde{\Y}$
    with continuous cocycles $l_{\tilde{\X}} = l_\X\circ \pi_\X, k_{\tilde{\X}} = k_\X\circ \pi_\X \colon \tilde{\X}\LRA \N$ and
    $l_{\tilde{\Y}} = l_\Y\circ \pi_\Y, k_{\tilde{\Y}} = k_\Y\circ \pi_\Y \colon \tilde{\Y}\LRA \N$.
\end{lemma}

\begin{proof}
    We first verify two claims which will allow us to define the map $\tilde{h}\colon \tilde{\X}\LRA \tilde{\Y}$.
    Then we show that $\tilde{h}$ with the prescribed cocycles is stabilizer-preserving.

    Let $x\in \X$ and $\mu, \nu\in \LL(\X)$ and suppose $h(x) \in C_\Y(\mu, \nu)$.
    \begin{center}
        \emph{Claim 1.} There are integers $0\leq k\leq l$ such that $x' \widesim{k, l} x \implies h(x')\in C_\Y(\mu, \nu)$.
    \end{center}
    Let $y := \sigma_\Y^{|\nu|}(h(x))$.
    Then $\nu y, \mu y\in \Y$ and $h(x) = \nu y$.
    From the cocycle relations~\eqref{eq:coe1} and~\eqref{eq:coe2} we have
    \[
        \sigma_\X^{l_\Y^{(|\nu|)}(\nu y)} (h^{-1}\nu y)) = \sigma_\X^{k_\Y^{(|\nu|)}(\nu y)} (h^{-1}(y)),
    \]
    and
    \[
        \sigma_\X^{l_\Y^{(|\mu|)}(\mu y)} (h^{-1}\mu y)) = \sigma_\X^{k_\Y^{(|\mu|)}(\mu y)} (h^{-1}(y)).
    \]
    Hence if
    \[
        \alpha' = {h^{-1}(\mu y)}_{[0, l_\Y^{(|\mu|)}(\mu y) + k_\Y^{(|\nu|)}(\nu y))}, \quad
        \beta' = {(h^{-1}(\nu y))}_{[0, l_\Y^{(|\nu|)}(\nu y) + k_\Y^{(|\mu|)}(\mu y))},
    \]
    then $h^{-1}(\mu y) = \alpha' z'$ and $h^{-1}(\nu y) = \beta' z'$, for some $z'\in \X$.
    If $z'$ is eventually periodic, pick $q\in \N$ such that $z := \sigma_\X^q(z')$ is periodic and $\gamma := z'_{[0, q)}$;
    if $z'$ is aperiodic, let $q = 0$ and let $\gamma$ be the empty word.
    Set $\alpha := \alpha' \gamma$ and $\beta := \beta' \gamma$ and observe that $h(\alpha z) = \mu y$ and $h(\beta z) = \nu y$, and $x = \beta z$.

    By the cocycle relations we have
    \[
        \sigma_\Y^{l_\X^{(|\alpha|)}(\alpha z)}(\mu y) = \sigma_\Y^{l_\X^{(|\alpha|)}(\alpha z)}(h(\alpha z)) = \sigma_\Y^{k_\X^{(|\alpha|)}(\alpha z)}(h(z)),
    \]
    and 
    \[ 
        \sigma_\Y^{l_\X^{(|\beta|)}(\beta z)}(\nu y) = \sigma_\Y^{l_\X^{(|\beta|)}(\beta z)}(h(\beta z)) = \sigma_\Y^{k_\X^{(|\beta|)}(\beta z)}(h(z)),
    \]
    from which we deduce that 
    \[
        \sigma_\Y^{k_\X^{(|\alpha|)}(\alpha z)+|\mu|+l_\X^{(|\beta|)}(\beta z)}(h(z))
        = \sigma_\Y^{l_\X^{(|\alpha|)}(\alpha z)+l_\X^{(|\beta|)}(\beta z)}(y) 
        = \sigma_\Y^{k_\X^{(|\beta|)}(\beta z)+|\nu|+l_\X^{(|\alpha|)}(\alpha z)}(h(z)).
    \]
    It now follows that
    In particular,
    \[
        k_\X^{(|\alpha|)}(\alpha z)+|\mu|+l_\X^{(|\beta|)}(\beta z) - \big( k_\X^{(|\beta|)}(\beta z)+|\nu|+l_\X^{(|\alpha|)}(\alpha z) \big)
    \]
    is a multiple of $\lp(h(z))$ --- if $h(z)$ is aperiodic, we set $\lp(h(z)) = 0$.
    Without loss of generality, we may assume there is a nonnegative integer $m$ such that
    \[
        k_\X^{(|\alpha|)}(\alpha z)+|\mu|+l_\X^{(|\beta|)}(\beta z) - \big( k_\X^{(|\beta|)}(\beta z)+|\nu|+l_\X^{(|\alpha|)}(\alpha z) \big)
            = m \lp(h(z))
            = m\big( l_\X^{(\lp(z))}(z) - k_\X^{(\lp(z))}(z) \big).
    \]
    The final equality follows from the hypothesis that $h$ is stabilizer-preserving.
    Set 
    \[
        N := l_\X^{(|\beta|)}(\beta z) + l_{\X}^{(|\alpha|)}(\alpha z) + m l_\X^{(\lp(z))}(z).
    \]
    Pick $r\in \N$ such that $k_\X^{(|\beta|)}$ and $l_\X^{(|\beta|)}$ are constant on $Z_\X(x_{[0, r)})$ and
    \[
        h(Z_\X( x_{[0, r)}))) \subset Z_\Y( (\nu y)_{[0, |\nu| + N}).
    \]
    Pick also $s\in \N$ such that $k_\X^{(|\alpha|)}$ and $l_\X^{(|\alpha|)}$ are constant on $Z_\X( (\alpha z)_{[0, s)})$ and
    \[
            h(Z_\X( (\alpha z)_{[0, s)})) \subset Z_\Y( (\mu y)_{[0, |\mu| + N)}),
    \]
    and such that $l_\X^{(\lp(z))}$ and $k_\X^{(\lp(z))}$ are constant on $Z_\X(z_{[0, s)})$.
    Set $k := r + s + |\beta|$ and $l := r + s + |\beta| + |\alpha|$.

    Let $x'\in \X$ and suppose $x' \widesim{k, l} x$.
    Then $x'\in Z_\X(x_{[0, r)})$, so $h(x')\in Z_\Y( (\nu y)_{[0, |\nu| + N)})$ and $k_\X^{(|\alpha|)}(x') = k_\X^{(|\alpha|)}(x)$
    and $l_\X^{(|\alpha|)}(x') = l_\X^{(|\alpha|)}(x)$.
    Put $x'' := \alpha z_{[0, m \lp(z))} \sigma_\X^{(|\beta|)}(x')$.
    Then $\sigma_\X^{|\alpha| + m \lp(z)}(x'') = \sigma_\X^{(|\beta|)}(x')$ and $x''\in Z_\X( (\alpha z)_{[0, s)})$ so $h(x'') \in Z_\Y( (\mu y)_{[0, |\mu| + N)})$.
    \begin{align*} 
        \sigma_\Y^{|\nu| + N}(h(x'))
        &= \sigma_\Y^{|\nu| + l_\X^{(|\beta|)}(x') + l_\X^{(|\alpha|)}(\alpha z) + m l_\X^{(\lp(z))}(z)}(h(x'))\\
        &= \sigma_\Y^{|\nu| + k_\X^{(|\beta|)}(x') + l_\X^{(|\alpha|)}(\alpha z) + m l_\X^{(\lp(z))}(z)}(h(\sigma_\X^{|\beta|}(x')))\\
        &= \sigma_\Y^{|\nu| + k_\X^{(|\beta|)}(\beta z) + l_\X^{(|\alpha|)}(\alpha z) + m l_\X^{(\lp(z))}(z)}(h(\sigma_\X^{|\beta|}(x'))),
    \end{align*}
    and 
    \begin{align*}
        \sigma_\Y^{|\mu| + N}(h(x''))
        &= \sigma_\Y^{|\mu| + l_\X^{(|\beta|)}(\beta z) + l_\X^{(|\alpha|)}(x'') + m l_\X^{(\lp(z))}(z)}(h(x''))\\
        &= \sigma_\Y^{|\mu| + l_\X^{(|\beta|)}(\beta z) + k_\X^{(|\alpha|)}(x'') + m k_\X^{(\lp(z))}(z)}(h(\sigma_\X^{|\alpha| + m\lp(z)}(x'')))\\
        &= \sigma_\Y^{|\mu| + l_\X^{(|\beta|)}(\beta z) + k_\X^{(|\alpha|)}(\alpha z) + m k_\X^{(\lp(z))}(z)}(h(\sigma_\X^{|\beta|}(x')))\\
        &= \sigma_\Y^{|\nu| + k_\X^{(|\beta|)}(\beta z) + l_\X^{(|\alpha|)}(\alpha z) + m l_\X^{(\lp(z))}(z)}(h(\sigma_\X^{|\beta|}(x'))).
    \end{align*}  
    Thus $\sigma_\Y^{|\nu| + N}(h(x')) = \sigma_\Y^{|\mu| + N}(h(x''))$ so 
    \[
        h(x'')\in C_\Y( (\mu y)_{[0, |\mu| + N)}, (\nu y)_{[0, |\nu| + N)}) \subset C_\Y(\mu, \nu)
    \]
    and this proves Claim 1.

    \begin{center}
        \emph{Claim 2.} For each $(k, l)\in \I$ there is $(m(k, l), n(k,l))\in \I$ such that
        \[
            x \widesim{m(k,l), n(k,l)} x' \implies h(x) \widesim{k, l} h(x').
        \]
    \end{center}
    Let $(k, l)\in \I$ and take $\mu, \nu\in \LL(\X)$ with $|\nu| = k$ and $|\mu|\leq l$ and $x\in C_\X(\mu, \nu)$.
    By Claim 1, we may choose $\big(r(\mu, \nu, x), s(\mu, \nu, x) \big) \in \I$ such that
    \[
        h({}_{r(\mu, \nu, x)} [x]_{s(\mu, \nu, x)}) \subset C_\Y(\mu, \nu).
    \]
    The topology on $\X$ generated by the sets $\{ {}_r [x]_s \mid x\in \X, (r, s)\in \I\}$ is compact,
    so there is a finite set $F \subset \LL(\X)\times \LL(\X)\times \X$ such that
    \[
        \bigcup_{(\mu, \nu, x)\in F} {}_{r(\mu, \nu, x)} [x]_{s(\mu, \nu, x)} = \X.
    \]
    Set $m(k, l) := \max \{ r(\mu, \nu, x) \mid (\mu, \nu, x)\in F\}$ and $n(k, l) := \max\{ s(\mu, \nu, x) \mid (\mu, \nu, x)\in F\}$.
    Then the implication of Claim 2 holds.

    We are now ready to prove the lemma.
    Let $(k, l)\in \I$ and set
    \[
        \tilde{m}(k,l) := \max\{ m(k',l') \mid (k',l')\preceq (k,l)\}, \quad
        \tilde{n}(k,l) := \max\{ n(k',l') \mid (k',l')\preceq (k,l)\}.
    \]
    Then there is a well-defined and continuous map $\tilde{h}\colon \tilde{\X}\LRA \tilde{\Y}$ given by
    \[
        {}_k(\tilde{h}(\tilde{x}))_l = h({}_{\tilde{m}(k,l)} [x]_{\tilde{n}(k,l)}),
    \]
    for $(k,l)\in \I$ and $\tilde{x} = {({}_r [x]_s)}_{(r, s)\in \I}$.
    A similar argument shows that there for $(k,l)\in \I$ is $(m'(k,l), n'(k,l))\in \I$ such that
    \[
        y \widesim{m'(k,l), n'(k,l)} y' \implies h^{-1}(y) \widesim{k,l} h^{-1}(y'),
    \]
    and that there is a continuous map $\tilde{h}'\colon \tilde{\Y}\LRA \tilde{\X}$ given by 
    \[ 
        {}_k (h'(\tilde{y}))_l = h^{-1}({}_{\tilde{m}'(k,l)} [y]_{\tilde{n}'(k,l)}),
    \]
    for $(k,l)\in\I$ and $\tilde{y} = ({}_r [y]_s)_{(r,s)\in\I}\in \tilde{\Y}$, 
    where $\tilde{m}'(k,l)=\max\{m'(k',l')\mid (k',l')\preceq (k,l)\}$ and $\tilde{n}'(k,l)=\max\{n'(k',l')\mid (k',l')\preceq (k,l)\}$. 
    Since $h'$ is the inverse of $\tilde{h}$, the latter map is a homeomorphism.

    It is straightforward to check that $h\circ\pi_\X = \pi_\Y\circ\tilde{h}$.
    Define $k_{\tilde{\X}}, l_{\tilde{\X}}\colon \tilde{\X}\LRA \N$ and $k_{\tilde{\Y}}, l_{\tilde{\Y}}\colon \tilde{\Y}\LRA \N$ 
    by $k_{\tilde{\X}} = k_\X\circ \pi_\X$, $l_{\tilde{\X}} = l_\X\circ\pi_\X$ and $k_{\tilde{\Y}} = k_\Y\circ\pi_\Y$, $l_{\tilde{\Y}} = l_\Y\circ\pi_\Y$. 
    They are continuous.
    It is straightforward to check that 
    $\sigma_{\tilde{\Y}}^{l_{\tilde{\X}}(\tilde{x})}(\tilde{h}(\tilde{x})) = \sigma_{\tilde{\Y}}^{k_{\tilde{\X}}(\tilde{x})}(\tilde{h}(\sigma_{\tilde{\X}}(\tilde{x})))$
    for $\tilde{x}\in\tilde{\X}$, and that 
    $\sigma_{\tilde{\X}}^{l_{\tilde{\Y}}(\tilde{y})}(\tilde{h}^{-1}(\tilde{y})) = \sigma_{\tilde{\X}}^{k_{\tilde{\Y}}(\tilde{y})}(\tilde{h}^{-1}(\sigma_{\tilde{\X}}(\tilde{x})))$
    for $\tilde{y}\in\tilde{\Y}$. 
    Thus, $(\tilde{h},l_{\tilde{\X}},k_{\tilde{\X}},l_{\tilde{\Y}},k_{\tilde{\Y}})$ is a continuous orbit equivalence. 

    We will now show that $(\tilde{h},l_{\tilde{\X}},k_{\tilde{\X}},l_{\tilde{\Y}},k_{\tilde{\Y}})$ is stabilizer-preserving. 
    Pick a periodic element $\tilde{x}\in \tilde{\X}$ and let $x = \pi_\X(\tilde{x})\in \X$.
    Then $x$ is periodic and if $\lp(x) = p$, then $\lp(\tilde{x}) = n p$ for some $n\in \N_+$. 
    Since $(h,l_\X,k_\X,l_\Y,k_\Y)$ is stabilizer-preserving, $h(x)\in \Y$ is eventually periodic and $|l_\X^{(p)}(x) - k_\X^{(p)}(x)| = \lp(h(x))$.
    Furthermore, 
    \[
        |l_{\tilde{\X}}^{(\lp(\tilde{x}))}(\tilde{x}) - \tilde{k}_{\X}^{(\lp(\tilde{x}))}(\tilde{x})| = n \lp(h(x))
    \]
    is a period for $\tilde{h}(\tilde{x})$.
    In particular, $\tilde{h}(\tilde{x})$ is eventually periodic and as above $\lp(\tilde{h}(\tilde{x})) = m \lp(h(x))$ for some $m\in \N_+$.
    The above computation shows that $m$ divides $n$.
    A similar argument using $h^{-1}$ instead of $h$ shows that $n$ divides $m$ and thus that $n = m$. 
    This shows that $\lp(\tilde{h}(\tilde{x})) = |l_{\tilde{\X}}^{(\lp(\tilde{x}))}(\tilde{x}) - k_{\tilde{\X}}^{(\lp(\tilde{x}))}(\tilde{x})|$. 
    Since $\tilde{h}$ maps periodic points to eventually periodic points, it also maps eventually periodic points to eventually periodic points.
\end{proof}

We shall next find conditions on $\G_\X$ and $\G_\Y$ and for $\OO_\X$ and $\OO_\Y$ that are equivalent to the existence of a 
stabilizer-preserving continuous orbit equivalence between $\X$ and $\Y$.
Recall the definition of $\kappa_\X\colon C(\X, \Z) \LRA B^1(\G_\X)$ from \eqref{eq:kappa}.

\begin{theorem}\label{thm:coe}
    Let $\X$ and $\Y$ be one-sided shift spaces, let $h\colon \X\LRA \Y$ be a homeomorphism and let $d_\X\colon \X\LRA\Z$ and $d_\Y\colon \Y\LRA\Z$ be continuous maps.
    The following conditions are equivalent:
    \begin{enumerate}
        \item[(i)] there are continuous maps $k_\X, l_\X\colon \X\LRA \N$ and $k_\Y, l_\Y\colon \Y\LRA \N$ with $d_\X = l_\X - k_\X$ and $d_\Y = l_\Y - k_\Y$
            such that $(h,l_\X,k_\X,l_\Y,k_\Y)$ is a stabilizer-preserving continuous orbit equivalence;
        \item[(ii)] there are continuous maps $k_\X, l_\X\colon \X\LRA \N$ and $k_\Y, l_\Y\colon \Y\LRA \N$ with $d_\X = l_\X - k_\X$ and $d_\Y = l_\Y - k_\Y$
            and continuous maps $k_{\tilde{\X}}, l_{\tilde{\X}} \colon \tilde{\X}\LRA \N$ and $k_{\tilde{\Y}}, l_{\tilde{\Y}} \colon \tilde{\Y}\LRA \N$ 
            with $l_{\tilde{\X}} = l_\X\circ\pi_\X$, 
            $k_{\tilde{\X}} = k_\X\circ\pi_\X$, 
            $l_{\tilde{\Y}} = l_\Y\circ\pi_\Y$, 
            $k_{\tilde{\Y}} = k_\Y\circ\pi_\Y$,
            and a homeomorphism $\tilde{h}\colon\tilde{\X}\LRA\tilde{\Y}$ such that
            $(\tilde{h},l_{\tilde{\X}},k_{\tilde{\X}},l_{\tilde{\Y}},k_{\tilde{\Y}})$ is a stabilizer-preserving continuous orbit equivalence 
            satisfying $h\circ\pi_\X = \pi_\Y\circ\tilde{h}$;
        \item[(iii)] there are
            \begin{itemize}
                \item a groupoid isomorphism $\Psi\colon \G_\X\LRA \G_\Y$ such that $h\circ \pi_\X = \pi_\Y\circ \Psi^{(0)}$ and
                    $\kappa_\X(d_\X) = \kappa_\Y(1)\circ\Psi$; and
                \item a groupoid isomorphism $\Psi'\colon\G_\Y\to \G_\X$ such that $h^{-1}\circ\pi_\Y = \pi_\X\circ(\Psi')^{(0)}$ 
                    and $\kappa_\Y(d_\Y) = \kappa_\X(1)\circ\Psi'$;
            \end{itemize}           
        \item[(iv)] there are
            \begin{itemize}
                \item a $^*$-isomorphism $\Phi\colon \OO_\X\LRA \OO_\Y$ such that $\Phi(C(\X)) = C(\Y)$, $\Phi(f)=f\circ h^{-1}$ for $f\in C(\X)$ 
                    and $\Phi\circ\gamma^\X_z = \beta^{\kappa_\Y(d_\Y)}_z\circ\Phi$ for each $z\in\mathbb{T}$; and
                \item a $^*$-isomorphism $\Phi'\colon \OO_\Y\LRA \OO_\X$ such that $\Phi'(C(\Y)) = C(\X)$, $\Phi'(f) = f\circ h$ for $f\in C(\Y)$ and
                    $\Phi'\circ\gamma^\Y_z = \beta^{\kappa_\X(d_\X)}_z\circ\Phi'$ for each $z\in\mathbb{T}$.
            \end{itemize}
    \end{enumerate}
\end{theorem}

\begin{proof}
    (i) $\implies$ (ii):
    This is Lemma~\ref{lem:lift-coe}.

    (ii) $\implies$ (iii):
    It follows from~\cite[Proposition 8.3]{CRST} that there is a groupoid isomorphism $\Psi\colon \G_\X\LRA \G_\Y$ satisfying
    \[
        \Psi((\tilde{x}, 1, \sigma_{\tilde{\X}}(\tilde{x}))) = 
        (\tilde{h}(\tilde{x}), l_{\tilde{\X}}(\tilde{x}) - k_{\tilde{\Y}}(\tilde{x}), \tilde{h}(\sigma_{\tilde{\X}}(\tilde{x})))
    \]
    for $\tilde{x}\in\tilde{\X}$,
    and a groupoid isomorphism $\Phi'\colon\G_\Y\to \G_\X$ satisfying 
    \[
        \Phi'((\tilde{y},1,\sigma_{\tilde{\Y}}(\tilde{y}))) = 
        (\tilde{h}^{-1}(\tilde{y}), l_{\tilde{\Y}}(y) - k_{\tilde{\Y}}(\tilde{y}), \tilde{h}^{-1}(\sigma_{\tilde{\Y}}(\tilde{y})))
    \]
    for $\tilde{y}\in\tilde{\Y}$. 
    We then have that $h\circ \pi_\X = \pi_\Y\circ \Phi^{(0)}$, $\kappa_\X(d_\X) = \kappa_\Y(1)\circ\Phi$, 
    $h^{-1}\circ\pi_\Y = \pi_\X\circ(\Phi')^{(0)}$ and $\kappa_\Y(d_\Y) = \kappa_\X(1)\circ\Phi'$.

    (iii) $\implies$ (i):
    Let $\Psi\colon \G_\X \LRA \G_\Y$ be a groupoid isomorphism satisfying $h\circ \pi_\X = \pi_\Y\circ \Psi^{(0)}$ and $\kappa_\X(d_\X) = \kappa_\Y(1)\circ\Psi$.
    Put $\tilde{h} := \Psi^{(0)}$.
    Then $\Psi(\tilde{x}, 1, \sigma_{\tilde{\X}}(\tilde{x})) = (\tilde{h}(\tilde{x}), d_\X(\pi_\X(\tilde{x})), \tilde{h}(\sigma_{\tilde{\X}}(\tilde{x})))$ and
    it follows from~\cite[Lemma 8.4]{CRST} that the map $l_{\tilde{\X}}\colon \tilde{\X}\LRA \N$ given by
    \[
        l_{\tilde{\X}}(\tilde{x}) = 
        \min\{ n\in \N \mid n\geq d_\X(\pi_\X(\tilde{x})),~ 
        \sigma_{\tilde{\Y}}^{n}(\tilde{h}(\tilde{x})) = \sigma_{\tilde{\Y}}(\tilde{h}(\sigma_{\tilde{\X}}(\tilde{x}))) \}
    \]
    is continuous.
    We claim that 
    \begin{align}\label{eq:l-minimum}
        l_{\tilde{\X}}(\tilde{x}) = 
        \min\{ n\in \N \mid n\geq d_\X(\pi_\X(\tilde{x})),~ 
        \sigma_{\tilde{\Y}}^{n}(\tilde{h}(\pi_\X(\tilde{x}))) = \sigma_{\tilde{\Y}}(\tilde{h}(\sigma_{\tilde{\X}}(\pi_\X(\tilde{x})))) \}.
    \end{align}
    By applying $\pi_\Y$, it is easy to see that the left hand side is less than the right hand side. 
    For the converse inequality, fix $\tilde{x}\in \tilde{\X}$ and suppose the right hand side of~\eqref{eq:l-minimum} equals $n$.
    Set $\tilde{y} := \sigma_{\tilde{\Y}}^n(\tilde{h}(\tilde{x}))$ and 
    $\tilde{y}' := \sigma_{\tilde{\Y}}^{n - d_\X(\pi_\X(\tilde{x}))}(\tilde{h}(\sigma_{\tilde{\X}}(\tilde{x})))$.
    Then $\pi_\Y(\tilde{y}) = \pi_\Y(\tilde{y}')$ by hypothesis, 
    and since $(\tilde{h}(\tilde{x}), d_\X(\pi_\X(\tilde{x})), \tilde{h}(\sigma_{\tilde{\X}}(\tilde{x}))) \in \G_\Y$
    there is an $m\in \N$ such that $\sigma_{\tilde{\Y}}^m(\tilde{y}) = \sigma_{\tilde{\Y}}(\tilde{y}')$.
    It now follows from Lemma~\ref{lem:isotropy-lemma}(i) that $\tilde{y} = \tilde{y'}$.
    This means that there is a map $l_\X\colon \X\LRA \N$ such that $l_{\tilde{\X}} = l_\X\circ \pi_\X$.
    This map is continuous by Lemma~\ref{lem:continuity}.
    Set $k_\X := d_\X - l_\X$.
    Then $k_\X$ is a continuous map satisfying $d_\X = l_\X - k_\X$ and $\sigma_\Y^{l(x)}(h(x)) = \sigma_\Y^{k(x)}(h(\sigma_\X(x)))$ for $x\in \X$.
    A similar argument shows that there are continuous maps $l_\Y, k_\Y\colon \Y\LRA \N$ satisfying $d_\Y = l_\Y - k_\Y$ and 
    $\sigma_\X^{l_\Y(y)}(h^{-1}(y)) = \sigma_\X^{k_\Y(y)}(h^{-1}(\sigma_\Y(y)))$ for $y\in \Y$. 
    Then $(h,l_\X,k_\X,l_\Y,k_\Y)$ is a continuous orbit equivalence. 

    Finally, we show that $(h,l_\X,k_\X,l_\Y,k_\Y)$ is stabilizer-preserving.
    Observe first that an argument similar to the one used in the proof of~\cite[Lemma 8.6]{CRST} shows that $(\tilde{h},l_{\tilde{\X}},k_{\tilde{\X}},l_{\tilde{\Y}},k_{\tilde{\Y}})$ is stabilizer-preserving.
    Fix an eventually periodic element $x\in \X$.
    Then $\tilde{x} = \iota_\X(x)\in \tilde{\X}$ is eventually periodic, so $\tilde{h}(\tilde{x})$ is eventually periodic.
    Hence $h(x) = \pi_\Y(\tilde{h}(\tilde{x}))\in \Y$ is eventually periodic.
    Now suppose $x$ is periodic with $\lp(x) = p$.
    Then $\iota_\X(x)\in \tilde{\X}$ is periodic with $\lp(\iota_\X(x)) = p$.
    Since $\tilde{h}(\iota_\X(x)) = \iota_\Y(h(x))$ we also have $\lp(h(x)) = \lp(\tilde{h(\tilde{x})})$,
    and using that $\tilde{h}$ is stabilizer-preserving in the middle equality below we see that
    \[
        |l_\X^{(p)}(x) - k_\X^{(p)}(x)| = |l_{\tilde{\X}}^{(p)}(\tilde{x}) - k_{\tilde{\X}}^{(p)}(\tilde{x})| = \lp(\tilde{h}(\tilde{x})) = \lp(h(x))
    \]
    which shows that $(h,l_\X,k_\X,l_\Y,k_\Y)$ is stabilizer-preserving.

    (iii) $\implies$ (iv):
    It follows from~\cite[Theorem 6.2]{CRST} that there is $^*$-isomorphism $\Phi\colon \OO_\X\LRA \OO_\Y$ such that $\Phi(\mathcal{D}_\X) = \mathcal{D}_\Y$, 
    $\Phi(f) = f\circ\tilde{h}^{-1}$ for $f\in\mathcal{D}_\X$, and $\Phi\circ\gamma^\X_z = \beta^{\kappa_\Y(d_\Y)}_z\circ\Phi$ for each $z\in\mathbb{T}$. 
    Since $h\circ\pi_\X = \pi_\Y\circ\Phi^{(0)}$, it follows that $\Phi(C(\X)) = C(\Y)$ and $\Psi(f) = f\circ h^{-1}$ for $f\in C(\X)$. 
    Similarly, there is a $^*$-isomorphism $\Phi'\colon \OO_\Y\LRA \OO_\X$ such that $\Phi'(C(\Y)) = C(\X)$, $\Phi'(f) = f\circ h$ for $f\in C(\Y)$, 
    and $\Phi'\circ\gamma^\Y_z = \beta^{\kappa_\X(d_\X)}_z\circ\Phi'$ for each $z\in\mathbb{T}$.

    (iv) $\implies$ (iii):
    An application of~\cite[Theorem 6.2]{CRST} shows that there is a groupoid isomorphism $\Psi\colon \G_\X\LRA \G_\Y$ such that $\kappa_\X(d_\X) = \kappa_\Y(1)\circ\Psi$. 
    Since $\Phi(C(\X)) = C(\Y)$ and $\Psi(f) = f\circ h^{-1}$ for $f\in C(\X)$, it follows that $h\circ\pi_\X = \pi_\Y\circ\Psi^{(0)}$. 
    Similarly, there is a groupoid isomorphism $\Psi'\colon\G_\Y\LRA \G_\X$ such that $h^{-1}\circ\pi_\Y = \pi_\X\circ(\Psi')^{(0)}$ and $\kappa_\Y(d_\Y) = \kappa_\X(1)\circ\Psi'$.
\end{proof}

\begin{remark}\label{rem:stab-preserving}
  It is natural to ask if in Theorem~\ref{thm:coe}(iii) the groupoid isomorphisms $\Psi$ and $\Psi'$ can be chosen to be inverses of each other,
  and if in (iv) the $^*$-isomorphisms $\Phi$ and $\Phi'$ can be chosen to be inverses of each other.
  This is not the case in general.

  Let $\X = \Y$ be the shift space with only one point.
  Then $(\id,1,0,0,1)$ is a stabilizer-preserving continuous orbit equivalence from $\X$ to $\Y$ (which is not least period preserving), where $d_\X = 1$ and $d_\Y = -1$.
  The groupoid $\G_\X$ is canonically isomorphic to the integer group $\Z$.
  If $\Psi\colon \G_\X \LRA \G_\Y$ is a group isomorphism, then the conditions $\kappa_\X(d_\X) = \kappa_\Y(1)\circ\Psi$ and $\kappa_\Y(d_\Y) = \kappa_\X(1)\circ\Psi^{-1}$ 
  imply that $\Psi$ maps the generator $1$ to both $1$ and $-1$ and this cannot be the case.
  Similarly, there is no $^*$-isomorphism $\Phi\colon \OO_\X \LRA \OO_\Y$ satisfying $\Phi(C(\X)) = C(\Y)$ and $\Phi\circ \gamma_z^\X = \beta_z^{\kappa_\Y(d_\Y)}\circ\Phi$ 
  and $\Phi^{-1}\circ\gamma_z^\Y = \beta_z^{\kappa_\X(d_\X)}\circ \Phi^{-1}$.
    
  We do not know if there are similar examples where $\X$ and $\Y$ are not of finite type.
\end{remark}

In~\cite[p. 61]{Mat2010b} (see also~\cite[p. 2]{Mat2019}), Matsumoto introduces the notion of a continuous orbit equivalence 
between factor maps of two one-sided shift spaces $\X$ and $\Y$ satisfying condition (I)
(implying that the groupoids $\G_\X$ and $\G_\Y$ are effective).
His factor maps can be more general than our $\pi_\X$ and $\pi_\Y$.
In this case, he proves a result (\cite[Theorem 1.2]{Mat2010b} and~\cite[Theorem 1.3]{Mat2019}) which is similar to the theorem below.
Our results applies to all one-sided shifts.

\begin{theorem}\label{thm:coe-cover}
    Let $\X$ and $\Y$ be one-sided shift spaces and let $h\colon \X\LRA \Y$ be a homeomorphism.
    The following conditions are equivalent:
    \begin{enumerate}
        \item[(i)] there is a stabilizer-preserving continuous orbit equivalence $\tilde{h}\colon \tilde{\X}\LRA \tilde{\Y}$ 
            satisfying $h\circ \pi_\X = \pi_\Y\circ \tilde{h}$;
        \item[(ii)] there is a groupoid isomorphism $\Psi\colon \G_\X\LRA \G_\Y$ satisfying $h\circ \pi_\X = \pi_\Y\circ \Psi^{(0)}$;
        \item[(iii)] there is a $^*$-isomorphism $\Phi\colon \OO_\X\LRA \OO_\Y$ satisfying $\Phi(C(\X)) = C(\Y)$ with $\Phi(f) = f\circ h^{-1}$ for $f\in C(\X)$.
    \end{enumerate}
    Moreover, if $h\colon \X\LRA \Y$ is a stabilizer-preserving continuous orbit equivalence, then the equivalent conditions above hold.
\end{theorem}

\begin{proof}
    (i)$\implies$(ii):
    Let  $\tilde{h}\colon \tilde{\X}\LRA \tilde{\Y}$ be a continuous orbit equivalence 
    and let $k_{\tilde{\X}}, l_{\tilde{\X}}\colon \tilde{\X} \LRA \N$ be continuous cocycles for $\tilde{h}$.
    There is a groupoid homomorphism $\Phi\colon \G_\X\LRA \G_\Y$ given by
    \[
        \Phi(\tilde{x}, m - n, \tilde{y}) = 
        \big( \tilde{h}(\tilde{x}), l_{\tilde{\X}}^{(m)}(\tilde{x}) - k_{\tilde{\X}}^{(m)}(\tilde{x}) - l_{\tilde{\X}}^{(n)}(\tilde{y}) + k_{\tilde{\X}}^{(n)}(\tilde{y}),
        \tilde{h}(\tilde{y})) \big),
    \]
    for $(\tilde{x}, m - n, \tilde{y})\in \G_\X$, 
    The assumption that $\tilde{h}$ be stabilizer-preserving ensures that $\Phi$ is bijective, cf.~\cite[Lemma 8.8 and Proposition 8.3]{CRST}.

    (ii)$\implies$(i):
    A groupoid isomorphism $\Phi\colon \G_\X\LRA \G_\Y$ restricts to a homeomorphism $\tilde{h} = \Phi^{(0)}\colon \tilde{\X}\LRA \tilde{\Y}$.
    If $c_\X$ is the canonical cocycle for $\G_\X$ and $c_\Y$ is the canonical cocycle for $\G_\Y$, then the maps 
    \begin{align*}
        l_{\tilde{\X}}(\tilde{x}) &= 
        \min \{ 
            l\in \N \mid \sigma_{\tilde{\Y}}^{l}(\tilde{h}(\tilde{x})) = 
            \sigma_{\tilde{\Y}}^{l - c_\Y\circ \Phi(\tilde{x}, 1, \sigma_{\tilde{\X}}(\tilde{x}))}(\tilde{h}(\sigma_{\tilde{\X}}(\tilde{x})))\}, \\
        k_{\tilde{\X}}(\tilde{x}) &= l_{\tilde{\X}}(\tilde{x}) - c_\Y\circ \Phi(\tilde{x}, 1, \sigma_{\tilde{\X}}(\tilde{x})),\\
        l_{\tilde{\Y}}(\tilde{y}) &= \min \{ 
            l\in \N \mid \sigma_{\tilde{\X}}^{l}(\tilde{h}^{-1}(\tilde{y})) = 
            \sigma_{\tilde{\X}}^{l - c_\X\circ \Phi^{-1}(\tilde{y}, 1, \sigma_{\tilde{\Y}}(\tilde{y}))}(\tilde{h}^{-1}(\sigma_{\tilde{\Y}}(\tilde{y})))\}, \\
            k_{\tilde{\Y}}(\tilde{y}) &= l_{\tilde{\Y}}(\tilde{y}) - c_\X\circ \Phi^{-1}(\tilde{y}, 1, \sigma_{\tilde{\Y}}(\tilde{y}))
    \end{align*}
    constitute continuous cocycles for $\tilde{h}$ such that $(\tilde{h},l_{\tilde{\X}},k_{\tilde{\X}},l_{\tilde{\Y}},k_{\tilde{\Y}})$ 
    is a stabilizer-preserving continuous orbit equivalence, cf.~\cite[Lemma 8.5 and Lemma 8.6]{CRST}.
    The condition $h\circ \pi_\X = \pi_\Y\circ \Phi^{(0)}$ implies that $h\circ \pi_\X = \pi_\Y\circ \tilde{h}$.

    The equivalence (ii) $\iff$ (iii) is~\cite[Theorem 8.2]{CRST}.
    If $\Phi\colon \G_\X\LRA \G_\Y$ is a groupoid isomorphism, then the condition $h\circ \pi_\X = \pi_\Y\circ \Phi^{(0)}$ translates to the condition 
    that the $^*$-isomorphism $\Psi\colon \OO_\X\LRA \OO_\Y$ satisfies $\Psi(C(\X)) = C(\Y)$.
    The latter condition implies that $\Psi(\D_\X) = \D_\Y$ by Theorem~\ref{thm:diagonal-preserving}.

    The final remark follows from Lemma~\ref{lem:lift-coe}.
\end{proof}

\begin{remark}
We do not know if there exist shift spaces $\X$ and $\Y$ such that the conditions in Theorem~\ref{thm:coe-cover} are satisfied, 
but there is no stabilizer-preserving continuous orbit equivalence between $\X$ and $\Y$. 
\end{remark}

Next, we show that we can relax the conditions in Theorem~\ref{thm:coe} for certain classes of shift spaces.
First we need a preliminary result concerning eventually periodic points.

In~\cite[Proposition 3.5]{MM-zeta}, Matsumoto and Matui show that any continuous orbit equivalence between shift spaces containing 
a dense set of aperiodic points maps eventually periodic points to eventually periodic points.
The result is only stated for shifts of finite type associated with irreducible and nonpermutation $\{0,1\}$-matrices,
but the proof holds in this generality, cf.~\cite[Remark 3.1]{CEOR}.
Below, we give a \emph{pointwise} version of this result applicable to all shift spaces.
We do not know if any continuous orbit equivalence between shift spaces preserves eventually periodic points,
but we show that this problem hinges on whether there exists a continuous orbit equivalence which maps an aperiodic isolated point 
to an eventually periodic isolated point.

\begin{proposition}\label{prop:nonisolated-ev-per}
    Let $\X$ and $\Y$ be one-sided shift spaces and let $h\colon \X\LRA \Y$ be a continuous orbit equivalence.
    Then $h$ maps nonisolated eventually periodic points to nonisolated eventually periodic points.
\end{proposition}

\begin{proof}
    It suffices to verify the claim for nonisolated periodic points.
    Suppose $x = \alpha^\infty\in \X$ for some word $\alpha\in \LL(\X)$ with $|\alpha| = p\in \N_+$,
    and let ${(x_n)}_n$ be a sequence in $\X$ converging to $x$.
    We may assume that $x_n \in Z_\X(\alpha)$ for all $n$.

    Suppose now that $k := k_\X^{(p)}(x) = l_\X^{(p)}(x)$.
    The cocycles $k_\X, l_\X\colon \X\LRA \N$ for $h$ are continuous, so there exists $N\in \N$ such that $k_\X^{(p)}(x_n) = l_\X^{(p)}(x_n) = k$ for $n\geq N$.
    In particular,
    \begin{align}\label{eq:nonisolated-relation}
        \sigma_\Y^k( h( x_n)) = \sigma_\Y^k( h( \sigma_\X^{p}( x_n))),
    \end{align}
    for $n\geq N$.
    The sequences ${h(x_n)}_n$ and ${(h(\sigma_\X^p(x_n)))}_n$ both converge to $h(x)$ in $\Y$,
    so there exists an integer $M\geq N$ such that 
    \[
        {h(x_n)}_{[0, k)} = {h(x)}_{[0, k)} = {h(\sigma_\X^p(x_n))}_{[0, k)},
    \]
    whenever $n\geq M$.
    This together with~\eqref{eq:nonisolated-relation} means that $h(x_n) = h(\sigma_\X^p(x_n))$ and hence $x_n = \sigma_\X^p(x_n)$, for $n\geq M$.
    Since $x_n \in Z_\X(\alpha)$, the sequence ${(x_n)}_n$ is eventually equal to $x$, so we conclude that $x$ is an isolated point.
    If $x$ is not isolated, then $l_\X^{(p)}(x) \neq k_\X^{(p)}(x)$, and the observation
    \[
        \sigma_\Y^{l_\X^{(p)}(x)}( h(x)) = \sigma_\Y^{k_\X^{(p)}(x)}( h( \sigma_\X^{(p)}( x))) = \sigma_\Y^{k_\X^{(p)}(x)}( h(x))
    \]
    shows that $h(x)$ is eventually periodic.
\end{proof}

A similar result holds for covers.

Since any homeomorphism respects isolated points, we obtain the corollary below.
Sofic shifts contain no aperiodic isolated points (cf.~Lemma~\ref{lem:isolated-points-sofic}), so this result resolves the problem for this class of shift spaces.

\begin{corollary}
    Let $\X$ and $\Y$ be one-sided shift spaces 
    either containing no aperiodic isolated points or no isolated eventually periodic points,
    then any continuous orbit equivalence $h\colon \X\LRA \Y$ maps eventually periodic points to eventually periodic points.
    In particular, this applies to sofic shift spaces.
\end{corollary}

If $\X$ and $\Y$ contain no periodic points isolated in past equivalence,
then $\X$ and $\Y$ as well as the covers $\tilde{\X}$ and $\tilde{\Y}$ contain dense sets of aperiodic points.
Hence the condition that a continuous orbit equivalence be stabilizer-preserving is superfluous.

\begin{theorem}\label{thm:coe-essentially-principal}
    Let $\X$ and $\Y$ be one-sided shift spaces with no periodic points isolated in past equivalence
    and let $h\colon \X\LRA \Y$ be a homeomorphism.
    The following are equivalent:
    \begin{enumerate}
        \item[(i)] the map $h\colon \X\LRA \Y$ is a continuous orbit equivalence;
        \item[(ii)] there is a continuous orbit equivalence $\tilde{h}\colon \tilde{\X}\LRA \tilde{\Y}$ satisfying $h\circ \pi_\X = \pi_\Y\circ \tilde{h}$;
        \item[(iii)] there is a groupoid isomorphism $\Psi\colon \G_\X \LRA \G_\Y$ satisfying $h\circ \pi_\X = \pi_\Y\circ \Psi^{(0)}$;
        \item[(iv)] there is a $^*$-isomorphism $\Phi\colon \OO_\X \LRA \OO_\Y$ satisfying $\Phi(C(\X)) = C(\Y)$ and $\Phi(f) = f\circ h^{-1}$ for $f\in C(\X)$.
    \end{enumerate}
\end{theorem}

\begin{proof}
    (i) $\implies$ (ii):
    Suppose $h\colon \X\LRA \Y$ is a continuous orbit equivalence with continuous cocycles $k_\X, l_\X\colon \X\LRA \N$.
    Since $\X$ and $\Y$ contain no periodic points isolated in past equivalence, it follows from Proposition~\ref{prop:essentially-principal} that 
    $\X$ and $\Y$ contain dense sets of aperiodic points.
    The proof of Theorem~\ref{thm:coe} (i) $\implies$ (ii) shows that there is a continuous orbit equivalence $\tilde{h}\colon \tilde{\X} \LRA \tilde{\Y}$
    with continuous cocycles $k_{\tilde{\X}} = k_\X\circ \pi_\X$ and $l_{\tilde{\X}} = l_\X\circ \pi_\X$ which satisfies $h\circ \pi_\X = \pi_\Y\circ \tilde{h}$.

    (ii) $\implies$ (iii):
    Let  $\tilde{h}\colon \tilde{\X}\LRA \tilde{\Y}$ be a continuous orbit equivalence 
    and let $k_{\tilde{\X}}, l_{\tilde{\X}}\colon \tilde{\X} \LRA \N$ be continuous cocycles for $\tilde{h}$.
    The map $\Phi\colon \G_\X\LRA \G_\Y$ given by
    \[
        \Phi(\tilde{x}, m - n, \tilde{y}) = 
        \big( \tilde{h}(\tilde{x}), l_{\tilde{\X}}^{(m)}(\tilde{x}) - k_{\tilde{\X}}^{(m)}(\tilde{x}) - l_{\tilde{\X}}^{(n)}(\tilde{y}) + k_{\tilde{\X}}^{(n)}(\tilde{y}),
        \tilde{h}(\tilde{y})) \big),
    \]
    for $(\tilde{x}, m - n, \tilde{y})\in \G_\X$, is a groupoid isomorphism satisfying $h\circ \pi_\X = \pi_\Y\circ \Psi^{(0)}$.

    (iii) $\implies$ (i):
    Let $\Psi\colon \G_\X \LRA \G_\Y$ be a groupoid isomorphism satisfying $h\circ \pi_\X = \pi_\Y\circ \Psi^{(0)}$.
    Then $\tilde{h} := \Psi^{(0)}\colon \tilde{\X} \LRA \tilde{\Y}$ is a continuous orbit equivalence with continuous cocycles 
    $l_{\tilde{\X}}, k_{\tilde{\X}}\colon \tilde{\X} \LRA \N$ and $l_{\tilde{\Y}}, k_{\tilde{\Y}}\colon \tilde{\Y} \LRA \N$
    given as in the proof of Theorem~\ref{thm:coe-cover} (ii) $\implies$ (i).
    We will show that there are continuous maps $l_\X, k_\X\colon \X\LRA \N$ and $l_\Y, k_\Y\colon \Y\LRA \N$
    which are continuous cocycles for $h$ such that $l_{\tilde{\X}} = l_\X\circ \pi_\X$, $k_{\tilde{\X}} = k_\X\circ \pi_\X$,
    $l_{\tilde{\Y}} = l_\Y\circ \pi_\Y$ and $k_{\tilde{\Y}} = k_\Y\circ \pi_\Y$.
    By Proposition~\ref{prop:essentially-principal}, $\X$ and $\Y$ have dense sets of aperiodic points, 
    so it suffices to show that $l_{\tilde{\X}}$ and $k_{\tilde{\X}}$ are constant on $\pi_\X^{-1}(x)$ for an aperiodic $x\in \X$.

    Let $x\in \X$ be aperiodic and take $\tilde{x}, \tilde{x}'\in \pi_\X^{-1}(x)$.
    Set $c_{\tilde{\X}} = l_{\tilde{\X}} - k_{\tilde{\X}}$ and $k := \max\{ k_{\tilde{\X}}(\tilde{x}), k_{\tilde{\X}}(\tilde{x}')\}$.
    If $x$ is isolated, then $\pi_\X^{-1}(x)$ is a singleton, so we may assume that $x$ is not isolated.
    Since $\sigma_{\tilde{\Y}}^{c_{\tilde{\X}}(\tilde{x}) + k}(\tilde{h}(\tilde{x})) = \sigma_{\tilde{\Y}}^k(\tilde{h}(\sigma_{\X}(\tilde{x})))$
    and $\sigma_{\tilde{\Y}}^{c_{\tilde{\X}}(\tilde{x}') + k}(\tilde{h}(\tilde{x}')) = \sigma_{\tilde{\Y}}^k(\tilde{h}(\sigma_{\X}(\tilde{x}')))$,
    it follows that
    \[
        \sigma_\Y^{c_{\tilde{\X}}(\tilde{x}) + k}(h(x)) = \sigma_{\tilde{\Y}}^{c_{\tilde{\X}}(\tilde{x}') + k}(h(x)).
    \]
    By Proposition~\ref{prop:nonisolated-ev-per}, we know that $h(x)$ is aperiodic, so $c_{\tilde{\X}}(\tilde{x}) = c_{\tilde{\X}}(\tilde{x}')$.

    Set
    \[
        \tilde{v} := \sigma_{\tilde{\Y}}^{l_{\tilde{\X}}(\tilde{x}')}(\tilde{h}(\tilde{x})), \quad
        \tilde{w} := \sigma_{\tilde{\Y}}^{l_{\tilde{\X}}(\tilde{x}') - c_{\tilde{\X}}(\tilde{x})}(\tilde{h}(\sigma_{\tilde{\X}}(\tilde{x}))).
    \]
    Since $c_{\tilde{\X}}(\tilde{x}) = c_{\tilde{\X}}(\tilde{x}')$ we have 
    $\tilde{w} = \sigma_{\tilde{\Y}}^{k_{\tilde{\X}}(\tilde{x}')}(\tilde{h}(\sigma_{\tilde{\X}}(\tilde{x})))$, and $\pi_\Y(\tilde{v}) = \pi_\Y(\tilde{w})$.
    The point $x$ is aperiodic, so Lemma~\ref{lem:isotropy-lemma} implies that $\tilde{v} = \tilde{w}$.
    By minimality in the definition of $l_{\tilde{\X}}$, it follows that $l_{\tilde{\X}}(\tilde{x}) \leq l_{\tilde{\X}}(\tilde{x}')$.
    A symmetric argument shows that $l_{\tilde{\X}}(\tilde{x}') \leq l_{\tilde{\X}}(\tilde{x})$.
    Hence $l_{\tilde{\X}}$ and $k_{\tilde{\X}}$ are constant on $\pi_\X^{-1}(x)$.
    There are therefore cocycles $l_\X, k_\X\colon \X\LRA \N$ for $h$ which satisfy $l_{\tilde{\X}} = l_\X\circ \pi_\X$ and $k_{\tilde{\X}} = k_\X\circ \pi_\X$,
    and they are continuous by Lemma~\ref{lem:continuity}.
    A similar argument shows that there are continuous cocycles $l_\Y, k_\Y\colon \Y\LRA \N$ satisfying 
    $l_{\tilde{\Y}} = l_\Y\circ \pi_\Y$ and $k_{\tilde{\Y}} = k_\Y\circ \pi_\Y$.
    Hence $h$ is a continuous orbit equivalence.

    (iii) $\iff$ (iv):
    This is~\cite[Theorem 8.2]{CRST}.
    Note that if $\Phi\colon \OO_\X \LRA \OO_\Y$ is a $^*$-isomorphism as in (iv),
    then $\Phi(\D_\X) = \D_\Y$ by Theorem~\ref{thm:diagonal-preserving}.
\end{proof}

\begin{corollary}\label{cor:coe-sofic}
    Let $\X$ and $\Y$ be one-sided shift spaces with no periodic points which are isolated in past equivalence.
    The following are equivalent:
    \begin{enumerate}
        \item[(i)] the systems $\X$ and $\Y$ are continuously orbit equivalent;
        \item[(ii)] there is a groupoid isomorphism $\Psi\colon \G_\X\LRA \G_\Y$ and a homeomorphism $h\colon \X\LRA \Y$ such that $h\circ \pi_\X = \pi_\Y\circ \Psi^{(0)}$;
        \item[(iii)] there is a $^*$-isomorphism $\Phi\colon \OO_\X\LRA \OO_\Y$ satisfying $\Phi(C(\X)) = C(\Y)$.
    \end{enumerate}
\end{corollary}

\begin{remark}
  In~\cite{Carlsen2003}, the second-named author showed that when $\X$ is a sofic shift, then $\OO_\X$ is $^*$-isomorphic to a Cuntz--Krieger algebra even in a diagonal-preserving way.
  However, Corollary~\ref{cor:coe-sofic} shows that for a strictly sofic shift $\X$ the $\mathrm{C^*}$-algebra $\OO_\X$ together with $C(\X)$ 
  can be distinguished from a Cuntz--Krieger algebra, cf.~Remark~\ref{rem:sofic}.
\end{remark}

\subsection{Examples}

We consider a few examples.

\begin{example}\label{ex:coe-notLPP}
    Let $\X_E$ and $\X_F$ be the vertex shifts of the reducible graphs 
    
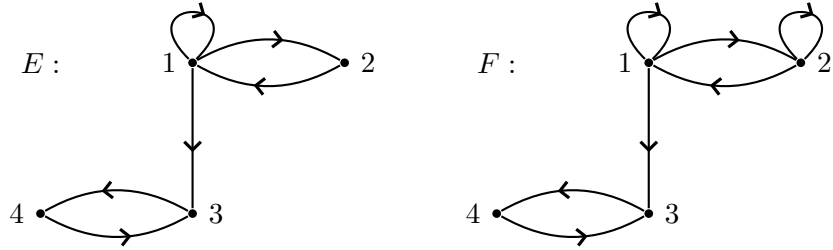
\begin{figure}[H]
    \begin{center}
    \begin{tikzpicture}
    [scale=5, ->-/.style={thick, decoration={markings, mark=at position 0.6 with {\arrow{Straight Barb[line width=0pt 1.5]}}},postaction={decorate}},
    node distance =2cm,
    thick,
    vertex/.style={inner sep=0pt, circle, fill=black}]
        \node (E) {$E:$};
        \node[vertex, right of = E, label=left:{1}] (E1) {.};
        \node[vertex, right of = E1, label=right:{2}] (E2) {.};
        \node[vertex, below of = E1, label=right:{3}] (E3)  {.};
        \node[vertex, left of = E3, label=left:{4}] (E4)  {.};

        \draw[->-, looseness=30, out=135, in=45] (E1) to (E1);
        \draw[->-, bend left] (E1) to  (E2);
        \draw[->-] (E1) to (E3);
        \draw[->-, bend left] (E2) to  (E1);
        \draw[->-, bend right] (E3) to (E4);
        \draw[->-, bend right] (E4) to (E3);

        \node (F) [right of = E2] {$F:$};
        \node[vertex, right of = F, label=left:{1}] (F1)  {.};
        \node[vertex, right of = F1, label=right:{2}] (F2)  {.};
        \node[vertex, below of = F1, label=right:{3}] (F3)  {.};
        \node[vertex, left of = F3, label=left:{4}] (F4)  {.};

        \draw[->-, looseness=30, out=135, in=45] (F1) to (F1);
        \draw[->-, bend left] (F1) to  (F2);
        \draw[->-] (F1) to  (F3);
        \draw[->-, looseness=30, out=135, in=45] (F2) to (F2);
        \draw[->-, bend left] (F2) to (F1);
        \draw[->-, bend right] (F3) to (F4);
        \draw[->-, bend right] (F4) to (F3);
    \end{tikzpicture}
    \end{center}
\end{figure}
Define a map $h\colon \X_E\LRA \X_F$ by exchanging the word $(21)$ with the word $2$
except in the case $h(21 {(34)}^\infty) = 21 {(34)}^\infty$.
Furthermore, $1 {(34)}^\infty$ is fixed by $h$ and $h({(34)}^\infty) = {(43)}^\infty$ and $h({(43)}^\infty) = {(34)}^\infty$.
This is a homeomorphism.
Consider the cocycles $k_E, l_E\colon \X_E\LRA \N$ and $k_F, l_F\colon \X_F\LRA \N$ given by
\[
    \begin{cases}
        k_E|_{\X_E\setminus Z(2)} &= 0 \\
        k_E|_{Z(2)} &= 1 \\
        l_E(1 {(34)}^\infty) &= 2 \\
        l_E|_{Z(1)\setminus \{1 {(34)}^\infty\}} &= 1 \\
        l_E(21 {(34)}^\infty) &= 2 \\
        l_E|_{Z(2)\setminus \{ 21 {(34)}^\infty\}} &= 1 \\
        l_E( {(34)}^\infty) &= 1 \\
        l_E( {(43)}^\infty) &= 1 
    \end{cases}, \qquad
    \begin{cases}
        k_F &= 0 \\
        l_F(1 {(34)}^\infty) &= 2 \\
        l_F|_{Z(1)\setminus \{1 {(34)}^\infty\}} &= 1 \\
        l_F(21 {(34)}^\infty) &= 1 \\
        l_F|_{Z(2)\setminus \{21 {(34)}^\infty\}} &= 2 \\
        l_F({(34)}^\infty) &= 1 \\
        l_F({(43)}^\infty) &= 1.
    \end{cases}
\]
They are continuous and $h$ is a continuous orbit equivalence with the specified cocycles.
Hence $\X_A$ and $\X_B$ are continuously orbit equivalent.

We will show that no choice of continuous cocycles of $h$ can be least period preserving on eventually periodic points.
Let $k_E, l_E\colon \X_E\LRA \N$ be any choice of continuous cocycles for $h$.
Let $x = 1 {(34)}^\infty\in \X_E$ and $z = \sigma_E(x)$.
The computation
\[
    \sigma_F^{l_E(x)}(1 {(34)}^\infty) =
    \sigma_F^{l_E(x)}( h(x)) = 
    \sigma_F^{k_E(x)}(h(\sigma_E(x))) =
    \sigma_F^{k_E(x)}{(43)}^\infty
\]
shows that $k_E(x)$ and $l_E(x)$ have the same parity.
On the other hand,
\[
    \sigma_F^{l_A(z)}( {(43)}^\infty) =
    \sigma_F^{l_A(z)}( h(z)) = 
    \sigma_F^{k_A(z)}(h(\sigma_E(z))) =
    \sigma_F^{k_A(z)}{(34)}^\infty
\]
shows that $k_E(z)$ and $l_E(z)$ have different parity.
Then $l_E^{(2)}(x) - k_E^{(2)}(x)$ is odd while $\lp(x) = 2$.
\end{example}

Below we revisit an example of Matsumoto~\cite{Mat2010} of infinite and irreducible shifts of finite type that are continuously orbit equivalent.
We show that they are \emph{not} eventually conjugate.

\begin{example}\label{ex:Matsumoto-coe}
    Let $\X$ be the full shift on the alphabet $\{1,2\}$ and let $\Y$ be the golden mean shift determined by the single forbidden word $\{22\}$.
    Then $\X$ and $\Y$ are infinite and irreducible shifts of finite type which are continuously orbit equivalent, cf.~\cite[p. 213]{Mat2010}.

    Suppose $h\colon \X\LRA \Y$ is an eventual conjugacy and that $\ell\in \N$ is an integer in accordance 
    with~\eqref{eq:eventual-conjugacy1} and~\eqref{eq:eventual-conjugacy2}.
    Then both $\sigma_\Y^\ell(h(1^{\infty}))$ and $\sigma_\Y^\ell(h(2^{\infty}))$ are constant sequences in $\Y$, 
    so they are both equal to $1^{\infty}\in \Y$.
    However, then
    \[
        1^{\infty} = \sigma_\X^\ell (h^{-1}(\sigma_\Y^\ell(h(1^{\infty})))) = \sigma_\X^\ell (h^{-1}(\sigma_\Y^\ell(h(2^{\infty})))) = 2^{\infty},
    \]
    which cannot be the case.
    Therefore, $\X$ and $\Y$ are not eventually conjugate.
\end{example}

\begin{example}\label{ex:even-odd-coe}
    Let $\X = \X_{\mathrm{even}}$ and $\Y = \Y_{\mathrm{odd}}$ be the even and the odd shift 
    defined by the following sets of forbidden words
    \[
        \mathscr{F}_{\mathrm{even}} = \{1 0^{2n + 1} 1 : n\in \N\}, \qquad
        \mathscr{F}_{\mathrm{odd}}  = \{1 0^{2n} 1 : n\in \N\},
    \]
    respectively.
    The shift spaces are represented in the labeled graphs $E$ and $F$ below.

    
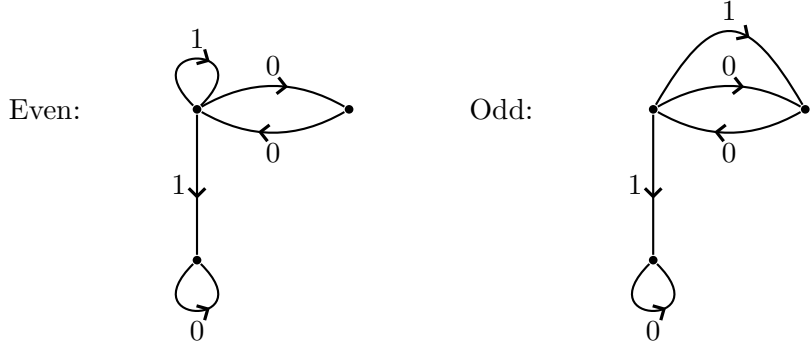
\begin{figure}[H]
    \begin{center}
    \begin{tikzpicture}
    [scale=5, ->-/.style={thick, decoration={markings, mark=at position 0.6 with {\arrow{Straight Barb[line width=0pt 1.5]}}},postaction={decorate}},
    node distance =2cm,
    thick,
    vertex/.style={inner sep=0pt, circle, fill=black}]
        \node (E) {Even:};
        \node[vertex] (E1) [right of = E] {.};
        \node[vertex] (E2) [right of = E1] {.};
        \node[vertex] (E3) [below of = E1] {.};

        \draw[->-, looseness=30, out=135, in=45] (E1) to node[above] {$1$} (E1);
        \draw[->-, bend left] (E1) to node[above] {$0$} (E2);
        \draw[->-] (E1) to node[left] {$1$} (E3);
        \draw[->-, bend left] (E2) to node[below] {$0$} (E1);
        \draw[->-, looseness=30, out=225, in=315] (E3) to node[below] {$0$} (E3);

        \node (F) [right of = E2] {Odd:};
        \node[vertex] (F1) [right of = F] {.};
        \node[vertex] (F2) [right of = F1] {.};
        \node[vertex] (F3) [below of = F1] {.};

        \draw[->-, bend left] (F1) to node[above] {$0$} (F2);
        \draw[->-, bend left = 60, looseness=2] (F1) to node[above] {$1$} (F2);
        \draw[->-] (F1) to node[left] {$1$} (F3);
        \draw[->-, bend left] (F2) to node[below] {$0$} (F1);
        \draw[->-, looseness=30, out=225, in=315] (F3) to node[below] {$0$} (F3);
    \end{tikzpicture}
    \end{center}
\end{figure}

    Define a map $h\colon \X\LRA \Y$ by exchanging the word $1$ by the word $(10)$.
    This is a homeomorphism.
    Furthermore, the cocycles $k_\X, l_\X\colon \X\LRA \N$ and $k_\Y, l_\Y\colon \Y\LRA \N$ given by
    \[
        \begin{cases}
            k_\X|_{Z(0)} = 0, &l_\X|_{Z(0)} = 1, \\
            k_\X|_{Z(1)} = 0, &l_\X|_{Z(1)} = 2
        \end{cases}, \qquad
        \begin{cases}
            k_\Y|_{Z(0)} = 0, &l_\Y|_{Z(0)} = 1, \\
            k_\Y|_{Z(1)} = 1, &l_\Y|_{Z(1)} = 1
        \end{cases}
    \]
    are continuous.
    Hence $h$ is a continuous orbit equivalence and $\X$ and $\Y$ are continuously orbit equivalent.
    An argument similar to that of Example~\ref{ex:Matsumoto-coe} shows that $\X$ and $\Y$ are not one-sided eventually conjugate.

    Observe that $h(0^\infty) = 0^\infty$ and $h(1^\infty) = {(10)}^\infty$ and
    \begin{align*}
        l_\X(0^\infty) - k_\X(0^\infty) &= 1 = \lp(0^\infty), \\
        l_\X(1^\infty) - k_\X(1^\infty) &= 2 = \lp({(10)}^\infty),
    \end{align*}
    so $(k_\X, l_\X)$ is least period preserving.
    A similar computation shows that $(k_\Y, l_\Y)$ is also least period preserving.
\end{example}

\begin{example}\label{ex:even-odd-cover-coe}
    Let $\X_E$ and $\X_F$ be the edge shifts determined by the reducible graphs
    
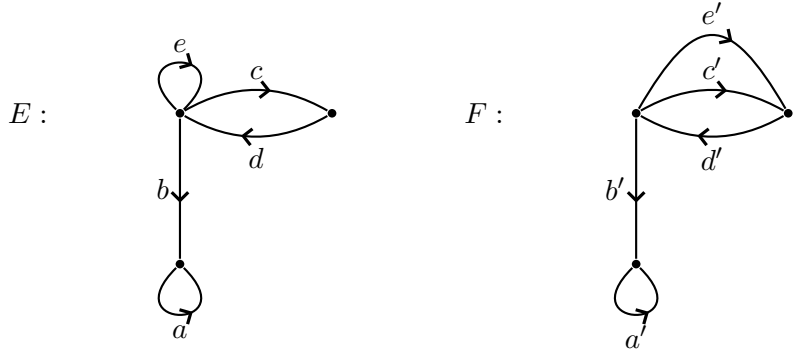
\begin{figure}[H]
    \begin{center}
    \begin{tikzpicture}
    [scale=5, ->-/.style={thick, decoration={markings, mark=at position 0.6 with {\arrow{Straight Barb[line width=0pt 1.5]}}},postaction={decorate}},
    node distance =2cm,
    thick,
    vertex/.style={inner sep=0pt, circle, fill=black}]
        \node (E) {$E:$};
        \node[vertex] (E1) [right of = E] {.};
        \node[vertex] (E2) [right of = E1] {.};
        \node[vertex] (E3) [below of = E1] {.};

        \draw[->-, looseness=30, out=135, in=45] (E1) to node[above] {$e$} (E1);
        \draw[->-, bend left] (E1) to node[above] {$c$} (E2);
        \draw[->-] (E1) to node[left] {$b$} (E3);
        \draw[->-, bend left] (E2) to node[below] {$d$} (E1);
        \draw[->-, looseness=30, out=225, in=315] (E3) to node[below] {$a$} (E3);

        \node (F) [right of = E2] {$F:$};
        \node[vertex] (F1) [right of = F] {.};
        \node[vertex] (F2) [right of = F1] {.};
        \node[vertex] (F3) [below of = F1] {.};

        \draw[->-, bend left] (F1) to node[above] {$c'$} (F2);
        \draw[->-, bend left = 60, looseness=2] (F1) to node[above] {$e'$} (F2);
        \draw[->-] (F1) to node[left] {$b'$} (F3);
        \draw[->-, bend left] (F2) to node[below] {$d'$} (F1);
        \draw[->-, looseness=30, out=225, in=315] (F3) to node[below] {$a'$} (F3);
    \end{tikzpicture}
    \end{center}
\end{figure}
Define a map $h\colon \X_E\LRA \X_F$ by exchanging any occurance of $e$ by $e'd'$.
This is a homeomorphism.
The maps $k_E, l_E\colon \X_E\LRA \N$ and $k_F, l_F\colon \X_F\LRA \N$ given by 
\[
    \begin{cases}
        k_E &= 0, \\
        l_F|_{Z(a)\cup Z(c)\cup Z(d)} &= 1, \\
        l_F|_{Z(b)\cup Z(e)} &= 2,
    \end{cases},
        \begin{cases}
            k_F|_{Z(a')\cup Z(c')\cup Z(d')} &= 0, \\
            k_F|_{Z(b')\cup Z(e')} &= 1,\\
            l_F &= 1,
        \end{cases}
\]
are continuous cocycles for $h$.
Hence $h$ is a continuous orbit equivalence and $\X_E$ and $\X_F$ are continuously orbit equivalent.
A computation shows that $(h, k_E, l_E, k_F, l_F)$ is least period preserving on periodic points but not on eventually periodic points.

In light of Example~\ref{ex:even-odd-coe} we can identify $\X_E$ and $\X_F$ with the covers $\tilde{\X}_{\mathrm{even}}$ and $\tilde{\Y}_{\mathrm{odd}}$, respectively,
and the cocycles above are induced from the cocycles on the even and odd shifts.
The maps $k_1, l_1, k_2, l_2\colon \X_E\LRA \N$ given by
\[
    \begin{cases}
        k_1 &= 0 \\
        l_1|_{Z(a)} &= 0 \\
        l_1|_{Z(b)\cup Z(c)\cup Z(d)} &= 1 \\
        l_1|_{Z(e)} &= 2
    \end{cases},
    \begin{cases}
        k_2 &= 0 \\
        l_2|_{\X_E\setminus Z(e)} &= 1 \\
        l_2|_{Z(e)} &= 2
    \end{cases}
\]
are continuous cocycles for $h$.
Then $(h, k_1, l_1, k_F, l_F)$ is not least period preserving and not constant on the preimages of $\pi_{\mathrm{even}}$
while $(h, k_2, l_2, k_F, l_F)$ is least period preserving on all eventually periodic points but not constant on the preimages under $\pi_{\mathrm{even}}$.
\end{example}

\section{Two-sided conjugacy}\label{sec:two-sided-conjugacy}

In~\cite[Corollary 5.2]{Carlsen-Rout}, the second-named author and Rout show that two-sided subshifts of finite type $\Lambda_\X$ and $\Lambda_\Y$
are conjugate if and only if the groupoids $\G_\X\times \R$ and $\G_\Y\times \R$ are isomorphic in a way which respects the canonical cocycle
and if and only if $\OO_\X\times \K$ and $\OO_\Y\otimes \K$ are $^*$-isomorphic in a way which intertwines the gauge actions suitably stabilized.
In this section, we characterize when a pair of general two-sided shift spaces are conjugate in terms of 
isomorphism of the groupoids $\G_\X\times \R$ and $\G_\Y\times \R$ and $^*$-isomorphism of $\OO_\X\otimes \K$ and $\OO_\Y\otimes \K$ (Theorem~\ref{thm:two-sided-conjugacy}).

Recall that if $\X$ is a one-sided shift space and $\sigma_\X$ is surjective, then the corresponding two-sided shift space $\Lambda_\X$ is constructed as the projective limit
\[
    \Lambda_\X = \varprojlim (\X, \sigma_\X).
\]
We shall write elements of $\X$ as $x, y, z\cdots$ and elements of $\Lambda_\X$ as $\textrm{x}, \textrm{y}, \textrm{z}\cdots$.

\begin{lemma}
    Let $\X$ be a one-sided shift space and let $\tilde{\X}$ be the associated cover.
    Then $\sigma_{\X}$ is surjective if and only if $\sigma_{\tilde{\X}}$ is surjective.
\end{lemma}

\begin{proof}
    If $\sigma_{\tilde{\X}}$ is surjective, 
    then the relation $\sigma_{\X}\circ \pi_{\X} = \pi_{\X}\circ \sigma_{\tilde{\X}}$ ensures that $\sigma_{\X}$ is surjective.
    On the other hand, suppose $\sigma_{\X}$ is surjective and let $\tilde{x}\in \tilde{\X}$.
    Take $x\in \X$ and integers $0\leq r < s$ such that $\tilde{x}\in U(x, r, s)$.
    Since $\sigma_{\X}$ is surjective, there exists $a\in \A$ such that $ax\in \X$.
    We have
    \[
        \tilde{x}\in U(x,r,s) = U(\sigma_\X(ax),r,s) = \sigma_{\tilde{\X}}( U(ax, r + 1, s)).
    \]
    In particular, we may pick $\tilde{y}\in U(ax,r + 1, s)$ such that $\sigma_{\tilde{\X}}(\tilde{y}) = \tilde{x}$.
\end{proof}

Following~\cite{Carlsen-Ruiz-Sims}, we let $\R$ be the full countable equivalence relation on $\N\times \N$.
The product of $(n,m),(n',m')\in \R$ is defined exactly when $m = n'$ in which case
\[
    (n,m)(n',m') = (n,m').
\]
Inversion is given as $(n,m)^{-1} = (m,n)$, and the source and range maps are
\[
    s(n,m) = m, \qquad r(n,m) = n,
\]
respectively, for $(n,m)\in \R$.

Given a one-sided shift space $\X$, we consider the product groupoid $\G_{\X}\times \R$ whose unit space we shall identify with $\tilde{\X}\times \N$
via the correspondence $( (\tilde{x},0,\tilde{x}), (0,0)) \LMT (\tilde{x},0)$.
The canonical cocycle is the continuous map $\bar{c}_{\X}\colon \G_{\X}\times \R\LRA \Z$ given by $\bar{c}_{\X}\big( (\tilde{x},k,\tilde{y}),(n,m)\big) = k$,
for $\big( (\tilde{x},k,\tilde{y}),(n,m)\big)\in \G_\X\times \R$.

We start by describing two-sided conjugacy in terms of sliding block codes on the corresponding one-sided shift spaces.
Recall that a sliding block code $\varphi\colon \X\LRA \Y$ between one-sided shift spaces $\X$ and $\Y$ 
is a continuous map satisfying $\varphi\circ \pi_\X = \pi_\Y\circ \varphi$.

\begin{definition}
    Let $\X$ and $\Y$ be one-sided shift spaces and let $\varphi\colon \X\LRA \Y$ be a sliding block code.
    We say that $\varphi$ is \emph{almost injective} (with lag $\ell_1$) if there exists $\ell_1\in \N$ such that
    \[
        \varphi(x) = \varphi(x') \implies \sigma_{\X}^{\ell_1}(x) = \sigma_{\X}^{\ell_1}(x'),
    \]
    for every $x,x'\in \X$.
    We say that $\varphi$ is \emph{almost surjective} (with lag $\ell_2$) if there exists $\ell_2\in \N$ such that
    for each $y\in \Y$ there exists $x\in \X$ such that $\sigma_{\Y}^{\ell_2}(\varphi(x)) = \sigma_{\Y}^{\ell_2}(y)$.
\end{definition}

Almost injective and almost surjective sliding block codes between covers is defined analogously.

\begin{lemma}\label{lem:two-sided-conj}
    Let $\Lambda_{\X}$ and $\Lambda_{\Y}$ be two-sided subshifts.
    If $\Lambda_{\X}$ and $\Lambda_{\Y}$ are two-sided conjugate, then there is a surjective sliding block code $\varphi\colon \X\LRA \Y$ which is almost injective.
    Conversely, if there exists a sliding block code $\varphi\colon \X\LRA \Y$ which is almost injective and almost surjective,
    then $\Lambda_{\X}$ and $\Lambda_{\Y}$ are conjugate.
\end{lemma}

\begin{proof}
    If $\Lambda_{\X}$ and $\Lambda_{\Y}$ are two-sided conjugate, 
    we may assume that there exist a two-sided conjugacy $H\colon \Lambda_{\X}\LRA \Lambda_{\Y}$ and $\ell\in \N$ such that
    \begin{align*}
        \textrm{x}_{[0,\infty)} &= \textrm{x}'_{[0,\infty)} \implies H(\textrm{x})_{[0,\infty)} = H(\textrm{x}')_{[0,\infty)}, \\
        \textrm{y}_{[0,\infty)} &= \textrm{y}'_{[0,\infty)} \implies H^{-1}(\textrm{y})_{[\ell,\infty)} = H^{-1}(\textrm{y}')_{[\ell,\infty)}
    \end{align*}
    for $\textrm{x},\textrm{x}'\in \Lambda_{\X}$, $\textrm{y},\textrm{y}'\in \Lambda_{\Y}$.
    Therefore, there is a well-defined map $\varphi\colon \X\LRA \Y$ given by
    \begin{align}\label{eq:surj-sbc-def}
        \varphi(x) = H(\textrm{x})_{[0,\infty)},
    \end{align}
    for every $x\in \X$ and $\textrm{x}\in \Lambda_{\X}$ with $x = \textrm{x}_{[0,\infty)}$. 
    The map $\varphi$ is a surjective sliding block code.
    Furthermore, if $\varphi(x) = \varphi(x')$ then $\sigma_{\X}^{\ell}(x) = \sigma_{\X}^{\ell}(x')$ for $x,x'\in \X$.

    Conversely, suppose $\varphi\colon \X\LRA \Y$ is an almost injective and almost surjective sliding block code with lag $\ell$.
    Define a map $h\colon \Lambda_{\X}\LRA \Lambda_{\Y}$ by 
    \[
        h(\ldots,x_2,x_1,x_0) = (\ldots, \varphi(x_2), \varphi(x_1), \varphi(x_0)),
    \]
    for $(\ldots,x_2,x_1,x_0)\in \Lambda_{\X}$.
    Note that $\sigma_{\Y}(\varphi(x_{i+1})) = \varphi(x_i)$ for $i\in \N$.
    Therefore $h$ is a well-defined sliding block code.
    We will show that $h$ is injective and surjective.

    Suppose first that $(\ldots,x_2,x_1,x_0), (\ldots,x'_2,x'_1,x'_0)\in \Lambda_{\X}$ and $\varphi(x_i) = \varphi(x'_i)$ for every $i\in \N$.
    Then 
    \[
        x_i = \sigma_{\X}^{\ell}(x_{i+\ell}) = \sigma_{\X}^{\ell}(x'_{i+\ell}) = x'_i
    \]
    for $i\in \N$ so $h$ is injective.
    
    Now let $(\ldots,y_2,y_1,y_0)\in \Lambda_{\Y}$ and choose $x,x'\in \X$ such that
    \[
        \sigma_{\Y}^{\ell}(\varphi(x)) = \sigma_{\X}^{\ell}(y_{2\ell}) = y_{\ell}, \qquad
        \sigma_{\Y}^{\ell}(\varphi(x')) = \sigma_{\X}^{\ell}(y_{2\ell + 1}) = y_{\ell +1}.
    \]
    Note that 
    \[
        \varphi(\sigma_{\X}^{\ell + 1}(x')) = y_{\ell} = \varphi(\sigma_{\X}^{\ell}(x))
    \]
    so $\sigma_{\X}^{2\ell + 1}(x') = \sigma_{\X}^{2\ell}(x)$ since $\varphi$ is almost surjective with lag $\ell$.
    Put $z_0 = \sigma_{\X}^{2\ell}(x)$ and $z_1 = \sigma_{\X}^{2\ell}(x')$ and observe that $\sigma_{\X}(z_1) = z_0$.
    Continuing this process inductively defines an sequence $(\ldots,z_2,z_1,z_0)\in \Lambda_{\X}$ which is sent to $(\ldots,y_2,y_1,y_0)$ via $h$.
    Hence $h$ is surjective and thus a two-sided conjugacy.
\end{proof}

Next we lift surjective sliding block codes on one-sided shift spaces to surjective sliding block codes on the covers.

\begin{lemma}\label{lem:lift-surj-sbc}
    Let $\X$ and $\Y$ be one-sided shift spaces and let $\varphi\colon \X\LRA \Y$ be a surjective sliding block code.
    Then there exists a surjective sliding block code $\tilde{\varphi}\colon \tilde{\X}\LRA \tilde{\Y}$ 
    satisfying $\varphi\circ \pi_\X = \pi_\Y\circ \tilde{\varphi}$.

    If, in addition, $\sigma_{\X}$ is surjective and $\varphi$ is almost injective with lag $\ell$, 
    then $\tilde{\varphi}$ is almost injective with lag $\ell$.
\end{lemma}

\begin{proof}
    Since $\varphi$ is a sliding block code there exists an integer $K\in \N$ such that
    \begin{equation}\label{eq:K-cts-const}
        x_{[0, r + K)} = x'_{[0, r + K)} \implies \varphi(x)_{[0,r)} = \varphi(x')_{[0,r)}
    \end{equation}
    for $r\in \N$ and $x,x\in \X$.
    We want to show that 
    \begin{equation}\label{eq:induced-sbc}
       x \widesim{r + K, s + K} x' \implies \varphi(x) \widesim{r,s} \varphi(x'),
    \end{equation}
    for $x,x'\in \X$ and integers $0\leq r\leq s$.
    
    Suppose $\nu \sigma_\Y^r(\varphi(x)) = \nu \varphi(\sigma_{\X}^k(x))$ where $\nu\in \LL(\Y)$ with $|\nu|\leq s$.
    We need to show that $\nu \sigma_\Y^k(\varphi(x'))\in \Y$.
    As $\varphi$ is surjective and commutes with the shift, there exists a word $\mu\in \LL(\X)$ with $|\mu| = |\nu|$ such that $\mu \sigma_{\X}^k(x)\in \X$ 
    and $\nu \varphi(\sigma_{\X}^k(x)) = \varphi(\mu \sigma_{\X}^k(x))$.
    By hypothesis, $\mu \sigma_\X^k(x')\in \X$ and we claim that $\varphi(\mu \sigma_\X^k(x')) = \nu \sigma_\Y^k(\varphi(x'))$.
    Indeed, ${\mu \sigma_{\X}^k(x)}_{[0, |\mu| + K)} = {\mu \sigma_{\X}^k(x')}_{[0, |\mu| + K)}$, so
    \[
        \varphi(\mu \sigma_{\X}^k(x'))_{[0,|\nu|)} = \varphi(\mu \sigma_{\X}^k(x))_{[0,|\nu|)} = |\nu|
    \]
    by the choice of $K$.
    This proves the claim.
    
    Define $\tilde{\varphi}\colon \tilde{\X}\LRA \tilde{\Y}$ by
    \[
        {}_r \tilde{\varphi}(\tilde{x})_s = \varphi({}_{r + K} x_{s + K}),
    \]
    for $\tilde{x}\in \tilde{\X}$ and integers $0\leq r\leq s$.
    It is straightforward to check that the induced map $\tilde{\varphi}\colon \tilde{\X}\LRA \tilde{\Y}$ is a surjective sliding block code 
    satisfying $\varphi\circ \pi_\X = \pi_\Y\circ \tilde{\varphi}$.
   
    Suppose now that $\sigma_{\X}$ is surjective and that there is $\ell\in \N$ 
    such that $\varphi(x) = \varphi(x')$ implies $\sigma_\X^{\ell}(x) = \sigma_\X^{\ell}(x')$ for all $x,x'\in \X$.
    Equivalently, there exists a surjective sliding block code $\rho\colon \Y\LRA \X$ satisfying $\sigma_\X^{\ell} = \rho\circ \varphi$.
    An argument similar to the one above shows that there is an induced surjective sliding block code $\tilde{\rho}\colon \tilde{\Y}\LRA \tilde{\X}$
    with $\rho\circ \pi_\Y = \pi_\X\circ \tilde{\rho}$.
    It is straightforward to verify that $\sigma_{\tilde{\X}}^{\ell} = \tilde{\rho}\circ \tilde{\varphi}$.
    Hence $\tilde{\varphi}$ is almost injective.
\end{proof}

We now arrive at the main theorem of this section which characterizes two-sided conjugacy of general shift spaces.
The proof uses ideas of~\cite{Carlsen-Rout}.
Let $\pi_{\X\times \N}\colon \tilde{\X}\times \N \LRA \X\times \N$ be the map $\pi_{\X\times \N}(\tilde{x}, n) = (\pi_\X(\tilde{x}), n)$,
for $(\tilde{x}, n)\in \tilde{\X}\times \N$.

\begin{theorem}\label{thm:two-sided-conjugacy}
    Let $\Lambda_\X$ and $\Lambda_\Y$ be two-sided shift spaces.
    The following are equivalent:
    \begin{enumerate}
        \item[(i)] there is a sliding block code $\varphi\colon \X\LRA \Y$ which is almost injective and almost surjective;
        \item[(ii)] there is a two-sided conjugacy $h\colon \Lambda_{\X}\LRA \Lambda_{\Y}$;
        \item[(iii)] there are a groupoid isomorphism $\Psi\colon \G_\X\times \R\LRA \G_\Y\times \R$ and a homeomorphism $\psi\colon \X\times \N\LRA \Y\times \N$
            satisfying $\psi\circ \pi_{\X\times \N} = \pi_{\Y\times \N}\circ \Psi^{(0)}$ and
            \[
                \bar{c}_\X = \bar{c}_\Y\circ \Psi;
            \]
        \item[(iv)] there is a $^*$-isomorphism $\Phi\colon \OO_\X\otimes \K\LRA \OO_\Y\otimes \K$ satisfying 
            $\Phi(C(\X)\otimes c_0) = C(\Y)\otimes c_0$ and
            \begin{align}\label{eq:stable-gauge-intertwining}
                \Phi\circ (\gamma^{\X}\otimes \id) = (\gamma^{\Y}\otimes \id)\circ \Phi.
            \end{align}
    \end{enumerate}
\end{theorem}

\begin{proof}
The equivalence (i) $\iff$ (ii) is Lemma~\ref{lem:two-sided-conj}.
 
(ii)$\implies$(iii):
Let $h\colon \Lambda_{\X}\LRA \Lambda_{\Y}$ be a conjugacy as in the proof of Lemma~\ref{lem:two-sided-conj} and let $\varphi\colon \X\LRA \Y$
be the surjective and almost injective sliding block code of~\eqref{eq:surj-sbc-def}.
By Lemma~\ref{lem:lift-surj-sbc} there exists a surjective and almost injective sliding block code $\tilde{\varphi}\colon \tilde{\X}\LRA \tilde{\Y}$
satisfying $\varphi\circ \pi_\X = \pi_\Y\circ \tilde{\varphi}$.
Since $\varphi$ is continuous, there exists $L\in \N$ such that
\[
    x_{[0,L)} = x'_{[0,L)} \implies \varphi(x)_{[0,\ell)} = \varphi(x')_{[0,\ell)},
\]
for $x,x'\in \X$.
Define an equivalence relation $\sim$ on words of length $L$ in the following way: 
Two words $\mu,\nu\in \LL_L(\X)$ are $\sim$-equivalent, if there are $x\in Z(\mu)$ and $x'\in Z(\nu)$ such that $\varphi(x) = \varphi(x')$.
Then $\varphi(x) = \varphi(x')$ if and only if $\sigma_{\X}^{\ell}(x) = \sigma_{\X}^{\ell}(x')$ and $x_{[0,L)} \sim x'_{[0,L)}$.
For every $\sim$-equivalence class $[\mu] = \{\nu \in \LL_L(\X)\mid \mu \sim \nu\}$, fix a partition
\[
    \N = \coprod_{\nu \in [\mu]} \N_{\nu}
\]
and bijections $f_{\nu}\colon \N_{\nu}\LRA \N$.
Define $\omega\colon \X\times \N \LRA \N$ by 
\[
    \omega(x, n) = f_{x_{[0, L)}}^{-1}(n),
\]
for $(x, n)\in \X\times \N$.
Then $\psi\colon \X\times \N\LRA \Y\times \N$ given by 
\[
    \psi(x,n) = (\varphi(x), \omega(x, n)),
\]
for $(x,n)\in \X\times \N$, is a homeomorphism.
Furthermore, the map $\Psi\colon \G_\X\times \R \LRA \G_\Y\times \R$ given by
\[
    \Psi\big( (\tilde{x}, k), n, (\tilde{y}, l) \big) =
    \big( (\tilde{\varphi}(\tilde{x}), \omega(\pi_\X(\tilde{x}), k)), n, (\tilde{\varphi}(\tilde{y}), \omega(\pi_\X(\tilde{y}), l)) \big)
\]
for $\big( (\tilde{x}, k), n, (\tilde{y}, l) \big)\in \G_\X\times \R$, is a groupoid isomorphism satisfying 
$\psi\circ \pi_{\X\times \N} = \pi_{\Y\times \N}\circ \Psi^{(0)}$.

(iii)$\implies$(i): 
Suppose $\Psi\colon \G_{\X}\times \R\LRA \G_{\Y}\times \R$ is a groupoid isomorphism satisfying the hypotheses of (iii).
Define a map $\tilde{\kappa}\colon \tilde{\X}\LRA \tilde{\Y}$ by $\Psi^{(0)} (\tilde{x},0) = (\tilde{\kappa}(\tilde{x}), n)$ 
for $\tilde{x}\in \tilde{\X}$ and some $n\in \N$.
Then $\tilde{\kappa}$ is well-defined and continuous since $\Psi^{(0)}$ is continuous.
By an argument similar to one in the proof of~\cite[Theorem 5.1]{Carlsen-Rout}, 
there exists $L\in \N$ such that $\tilde{\varphi} := \sigma_{\tilde{\Y}}^L\circ \tilde{\kappa}$ is a sliding block code which is almost injective and almost surjective,
say with lag $\ell$.
Define also $\kappa\colon \X\LRA \Y$ by $\psi(x,0) = (\varphi(x), m)$ for $x\in \X$ and some $m\in \N$.
Then $\kappa$ is continuous and $\kappa\circ \pi_{\X} = \pi_{\Y}\circ \tilde{\kappa}$.
It follows that $\varphi := \sigma_{\Y}^L\circ \kappa\colon \X\LRA \Y$ is a sliding block code.

Let $y\in \Y$ and choose $\tilde{y}\in \pi_{\Y}^{-1}(y)$.
Pick $\tilde{x}\in \tilde{\X}$ such that $\sigma_{\tilde{\Y}}^{\ell}(\tilde{\varphi}(\tilde{x})) = \sigma_{\tilde{\Y}}^{\ell}(\tilde{y})$.
If $x = \pi_{\X}(\tilde{x})$, then $\sigma_{\Y}^{\ell}(\varphi(x)) = \sigma_{\Y}^{\ell}(y)$ and $\varphi$ is almost surjective.

In order to see that $\varphi$ is almost injective, choose distinct $x,x'\in \X$ such that $y := \varphi(x) = \varphi(x')$.
Choose distinct $n,n'\in \N$ such that $\psi(x,0) = (y,n)$ and $\psi(x',0) = (y,n')$ and pick $\tilde{y}\in \pi_{\Y}^{-1}(y)$.
Since $\Phi^{(0)}$ is a homeomorphism, there are unique and distinct $\tilde{x},\tilde{x'}\in \X$
such that $\Phi^{(0)}(\tilde{x},0) = (\tilde{y}, n)$ and $\Phi^{(0)}(\tilde{x}',0) = (\tilde{y}, n')$.
It follows that $\sigma_{\tilde{\X}}^{\ell} (\tilde{x}) = \sigma_{\tilde{\X}}^{\ell} (\tilde{x}')$ since $\tilde{\varphi}$ is almost injective. 
Since $\psi\circ \pi_{\X\times \N} = \pi_{\Y\times \N}\circ \Phi^{(0)}$, we have $\pi_{\X}(\tilde{x}) = x$ and $\pi_{\X}(\tilde{x}') = x'$. 
Hence $\sigma_{\X}^{\ell}(x) = \sigma_{\X}^{\ell}(x')$ and $\varphi$ is almost injective.

(iii)$\implies$(iv):
A groupoid isomorphism $\Psi\colon \G_{\X}\times \R\LRA \G_{\Y}\times \R$ induces a $^*$-isomorphism $\Phi\colon \OO_\X\otimes \K\LRA \OO_\Y\otimes \K$
satisfying $\Phi(\D_\X\otimes c_0) = \D_\Y\otimes c_0$.
Since $\psi\circ \pi_{\X\times \N} = \pi_{\Y\times \N}\circ \Psi^{(0)}$, we also have $\Phi(C(\X)\otimes c_0) = C(\Y)\otimes c_0$.
The relation $\bar{c}_\X = \bar{c}_\Y\circ \Phi$ ensures that~\eqref{eq:stable-gauge-intertwining} is satisfied.

(iv) $\implies$ (ii):
By Corollary~\ref{cor:diagonal-preserving-stable}, $\Phi(\D_\X\otimes c_0) = \D_\Y\otimes c_0$.
From~\cite[Theorem 8.10]{CRST} there is a groupoid isomorphism $\Psi\colon \G_\X\times \R\LRA \G_\Y\times \R$ 
satisfying $\Phi(f) = f\circ \Psi^{-1}\in \D_\Y\otimes c_0$ for $f\in \D_\X\otimes c_0$, and $\bar{c}_\X = \bar{c}_\Y \circ \Psi $.
Since $\Phi(g) = g\circ \tilde{\psi}^{-1}\in C(\Y)\otimes c_0$ for $g\in C(\X)\otimes c_0$, there is a homeomorphism $\psi\colon \X\times \N\LRA \Y\times \N$
such that $\psi\circ \pi_{\X\times \N} = \pi_{\Y\times \N} \circ \Psi^{(0)}$.
\end{proof}

\begin{corollary}
    Let $\Lambda_\X$ and $\Lambda_\Y$ be two-sided shift spaces.
    The following are equivalent:
    \begin{enumerate}
        \item[(i)] the two-sided subshifts $\Lambda_{\X}$ and $\Lambda_{\Y}$ are two-sided conjugate;
        \item[(ii)] there is a groupoid isomorphism $\Psi\colon \G_\X\times \R\LRA \G_\Y\times \R$ and a homeomorphism $\psi\colon \X\times \N\LRA \Y\times \N$ 
            satisfying $\psi\circ \pi_{\X\times \N} = \pi_{\Y\times \N}\circ \Psi^{(0)}$ and $\bar{c}_\X = \bar{c}_\Y\circ \Psi$;
        \item[(iii)] there is a $^*$-isomorphism $\Phi\colon \OO_\X\otimes \K\LRA \OO_\Y\otimes \K$ satisfying 
            $\Phi(C(\X)\otimes c_0) = C(\Y)\otimes c_0$ and $\Phi\circ (\gamma^{\X}\otimes \id) = (\gamma^{\Y}\otimes \id)\circ \Phi$.
    \end{enumerate}
\end{corollary}

If $\Lambda_A$ and $\Lambda_B$ are the two-sided subshifts associated to finite square $\{0,1\}$-matrices with no zero rows and no zero columns,
then we recover~\cite[Corollary 5.2]{Carlsen-Rout}.
See also~\cite[Theorem 1.5]{Mat2019}.

\section{Flow equivalence}\label{sec:flow}

It is proven in~\cite[Corollary 6.3]{CEOR} (see also~\cite[Corollary 3.8]{MM14}) that if $\Lambda_\X$ and $\Lambda_\Y$ are two-sided subshifts of finite type, 
then $\Lambda_\X$ and $\Lambda_\Y$ are flow equivalent if and only if $\G_\X\times \R$ and $\G_\Y \times \R$ are isomorphic, 
and if and only if there is a $^*$-isomorphism $\OO_\X\otimes \K \LRA \OO_\Y\otimes\K$ which maps $C(\X)\otimes c_0$ onto $C(\Y)\otimes c_0$.
In this section, we shall for general shift spaces $\X$ and $\Y$ look at the relationship between flow equivalence of $\Lambda_\X$ and $\Lambda_\Y$, 
isomorphism of $\G_\X\times\R$ and $\G_\Y\times\R$, and $^*$-isomorphisms $\OO_\X\otimes\K \LRA \OO_\Y\otimes\K$ which map $C(\X)\otimes c_0$ onto $C(\Y)\otimes c_0$.

The \emph{ordered cohomology}~\cite[p. 868]{MM14} of $\X$ is the group
\[
    H^\X = C(\X, \Z)/\{ f - f\circ \sigma_\X : f\in C(\X, \Z)\},
\]
with the positive cone
\[
    H^\X_+ = \{[f]\in H^\X \mid f\geq 0\}.
\]
The ordered cohomology of the cover $\tilde{\X}$ is defined analogously.
An isomorphism of cohomology groups is \emph{positive} if it maps the positive cone onto the positive cone,
and two maps $f, f'\in C(\X, \Z)$ are \emph{cohomologous} if $[f] = [f']$ in $H^\X$.

\begin{remark}
Recall that $B^1(\G_\X)$ is the collection of groupoid homomorphisms from $\G_\X$ to $\Z$.
The first cohomology group of $\G_\X$ is the group
\[
    H^1(\G_\X) = B^1(\G_\X) / \{ \partial(f) \mid f\in C(\G_\X^{(0)}, \Z)\},
\]
where $\partial \colon C(\G_\X^{(0)}, \Z) \LRA B^1(\G_\X)$ is $\partial(f)(\gamma) = f(r(\gamma)) - f(s(\gamma))$, 
for $f\in C(\G_\X^{(0)}, \Z)$ and $\gamma\in \G_\X$, cf.~\cite[p. 870]{MM14}.
There is a canonical isomorphism $\Theta\colon H^1(\G_\X) \LRA H^{\tilde{\X}}$ given by $\Theta([f]) = [g]$, where
\[
    g(\tilde{x}) = f(\tilde{x}, 1, \sigma_{\tilde{\X}}(\tilde{x})),
\]
for $\tilde{x}\in \tilde{\X}$, cf.~\cite[Proposition 4,7]{CEOR}.
\end{remark}

The factor map $\pi_\X\colon \tilde{\X}\LRA \X$ induces a well-defined injective map $\pi_\X^*\colon H^\X \LRA H^{\tilde{\X}}$
given by $\pi_\X^*([f]) = [f\circ \pi_\X]$, for $f\in C(\X, \Z)$.
Note $\pi_\X^*(H^\X_+) \subset H^{\tilde{\X}}_+$ and $\pi_\X^*([1_\X]) = 1_{\tilde{\X}}$.
The ordered cohomology $(H^{\Lambda_\X}, H^{\Lambda_\X}_+)$ of a two-sided subshift $\Lambda_\X$ is defined analogously,
and there is a canonical isomorphism $(H^\X, H^\X_+) \cong (H^{\Lambda_\X}, H^{\Lambda_\X}_+)$.
This was shown in~\cite[Lemma 3.1]{MM14} for infinite irreducible shifts of finite type but as noted in~\cite[Section 2.5]{CEOR} the proof holds for general shifts.

If $\X$ is a one-sided shift space, then the \emph{stabilization of $\X$} is the space $\X\times \N$ with the shift operation $S_\X\colon \X\times \N\LRA \X\times \N$
given by
\[
    S_\X(x, n) = 
    \begin{cases}
        (x, n - 1) & \mathrm{if}~n > 0, \\
        (\sigma_\X(x), 0) & \mathrm{if}~n = 0,
    \end{cases}
\]
for $(x, n)\in \X\times \N$.
We define $S_{\tilde{\X}}\colon \tilde{\X}\times \N\LRA \tilde{\X}\times \N$ in a similar way.

The ordered cohomology for the stabilized system is the group
\[
    H^{\X\times\N} = C(\X\times\N, \Z)/\{ f - f\circ S_\X : f\in C(\X\times\N, \Z)\},
\]
with the positive cone
\[
    H^{\X\times\N}_+ = \{[f]\in H^{\X\times\N} \mid f\geq 0\}.
\]

The ordered cohomology is stable in the following sense.

\begin{lemma}
    Let $\X$ be a one-sided shift space and let $\iota_0\colon \X\LRA \X\times \N$ be the inclusion given by $\iota_0(x) = (x, 0)$.
    There is a surjective homomorphism $\iota_0^*\colon C(\X\times \N, \Z) \LRA C(\X, \Z)$ defined by $\iota_0^*(\xi)(x) = \xi(x, 0)$, 
    and an isomorphism $H(\iota_0)\colon H^{\X\times \N}\LRA H^\X$ such that $H(\iota_0)([\xi]) = [\iota_0^*(\xi)]$. 
    Moreover, $H(\iota_0)(H^{\X\times \N}_+) = H^\X_+$ and $H(\iota_0)([1_{\X\times \N}]) = [1_{\X}]$.
\end{lemma}

\begin{proof}
It is straightforward to check that $\iota_0^*\colon C(\X\times \N, \Z) \LRA C(\X, \Z)$ is a surjective homomorphism.
Since $\iota_0$ is a sliding block code, there is a well-defined surjective map $H(\iota_0)\colon H^{\X\times \N} \LRA H^\X$ 
given by $H(\iota_0)([\xi]) = [\iota_0^*(\xi)]$.
Any class $[\xi]\in H^{\X\times \N}$ can be represented by a map $\xi\in C(\X\times \N, \Z)$ which is supported on $\X\times \{0\}$.
If $\xi(x, 0) = b(x) - b(\sigma_\X(x))$, for some $b\in C(\X, \Z)$ and $x\in \X$, 
we may take $\eta\in C(\X\times \N, \Z)$ supported on $\X\times \{0\}$ such that $\eta(x, 0) = b(x)$.
Then $\xi = \eta - \eta\circ S_\X$ on $\X\times \{0\}$, so $H(\iota_0)$ is injective.

It is clear the $\iota_0^*(\xi)\geq 0$ if $\xi\geq 0$. 
Conversely, let $g\in C(\X, \Z)$ and assume that $g\geq 0$.
Take $\eta\in C(\X\times \N, \Z)$ supported on $\X\times \{0\}$ such that $\xi(x, 0) = g(x)$ for $x\in \X$, and note that $\iota_0^*(\xi) = g$ and $\xi\geq 0$.
Hence $H(\iota_0)(H^{\X\times \N}_+) = H^\X_+$.
Finally, $\iota_0^*(1_{\X\times \N}) = 1_\X$.
\end{proof}

We will write an element of $\G_\X\times\R$ as $\big((\tilde{x}, k), n, (\tilde{y}, l)\big)$ instead of $\big((\tilde{x}, n, \tilde{y}), (k, l)\big)$,
where $(\tilde{x}, k), (\tilde{y}, l)\in \tilde{\X}\times \N$ and $\sigma_{\tilde{\X}}^j(\tilde{x}) = \sigma_{\tilde{\X}}^i(\tilde{y})$
for some $i, j\in \N$ with $n = j - i$.
We then have that 
\begin{align*}
    \G_\X\times\R = 
    \{ \big( (\tilde{x}, k), n, (\tilde{y}, l) \big) \mid \exists i, j\in \N: n = j - i,~S_{\tilde{\X}}^{k + j}(\tilde{x}, k) = S_{\tilde{\X}}^{l + i}(\tilde{y}, l),
    ~\sigma_{\tilde{\X}}^j(\tilde{x}) = \sigma_{\tilde{\X}}^i(\tilde{y})\}.
\end{align*}

Let $\pi_{\X\times\N}\colon \tilde{\X}\times\N\LRA \X\times\N$ be the map defined by $\pi_{\X\times\N}(\tilde{x},n) = (\pi_\X(\tilde{x}),n)$,
for $(\tilde{x}, n)\in \tilde{\X}\times \N$.
There is an injective homomorphism $\kappa_{\X\times\N}\colon C(\X\times\N,\Z)\LRA B^1(\G_\X\times\R)$ defined by 
\[
    \kappa_{\X\times\N}(f)\big( (\tilde{x},k), n, (\tilde{y},l) \big)
    =\sum_{r = 0}^{j + k} f\big( \pi_{\X\times\N}(S_{\tilde{\X}}^r (\tilde{x},k)) \big) - \sum_{r = 0}^{i+l} f\big( \pi_{\X\times\N}(S_{\tilde{\X}}^r(\tilde{y},l)) \big) 
\]
where $i, j\in\N$ are such that $\sigma_{\tilde{\X}}^j(\tilde{x}) = \sigma_{\tilde{\X}}^i(\tilde{y})$ and $n = j - i$. 
In particular, $\kappa_{\X\times\N}(f)\colon \G_\X\times \R \LRA \Z$ it the unique cocycle satisfying 
\[
    \kappa_{\X\times\N}(f)\big( (\tilde{x}, k), 1, S_{\tilde{\X}}(\tilde{x}, k) \big) = f\big(\pi_{\X\times\N}(\tilde{x}, k)\big),
\]
for $(\tilde{x}, k)\in \tilde{\X}\times \N$.

If $\Lambda_\X$ and $\Lambda_\Y$ are conjugate subshifts, then they have isomorphic ordered cohomology.
We give a one-sided decription below.

\begin{lemma}\label{lem:conjugacy,SBC-homeo}
    Let $\Lambda_\X$ and $\Lambda_\Y$ be two-sided subshifts that are conjugate.
    Then there exist
    \begin{enumerate}
        \item[(i)] 
            \begin{itemize}
                \item a surjective and almost injective sliding block code $\varphi\colon \X\LRA \Y$
                    and an injective homomorphism $\varphi^*\colon C(\Y, \Z) \LRA C(\X, \Z)$ given by $\varphi^*(g) = g\circ \varphi$;
                \item a positive isomorphism $H(\varphi)\colon H^\Y\LRA H^\X$ satisfying $H(\varphi)([1_\Y]) = [1_\X]$ 
                    and $H(\varphi)[g] = [g\circ \varphi]$ for $g\in C(\Y, \Z)$;
            \end{itemize}
        \item[(ii)] a groupoid isomorphism $\Psi\colon \G_{\X}\times \R\LRA \G_\Y\times \R$ and a homeomorphism $\psi\colon \X\times \N \LRA \Y\times \N$
            satisfying $\psi\circ \pi_{\X\times \N} = \pi_{\Y\times \N}\circ \Psi^{(0)}$;
        \item[(iii)] 
            \begin{itemize}
                \item a homomorphism $\psi^*\colon C(\Y\times \N, \Z) \LRA C(\X\times \N, \Z)$ such that 
                    $\kappa_{\X\times\N}(\psi^*(\eta)) = \kappa_{\Y\times\N}(\eta)\circ\Psi$ for $\eta\in C(\Y\times \N, \Z)$,
                \item a homomorphism $\psi^\#\colon C(\X\times \N, \Z) \LRA C(\Y\times \N, \Z)$ such that
                    $\kappa_{\Y\times\N}(\psi^\#(\zeta)) = \kappa_{\X\times\N}(\zeta)\circ\Psi^{-1}$ for $\zeta\in C(\X\times \N, \Z)$; and
                \item a positive isomorphism $H(\psi)\colon H^{\Y\times \N} \LRA H^{\X\times \N}$ such that 
                    $H(\psi)([\eta]) = [\psi^*(\eta)]$ for $\eta\in C(\Y\times \N, \Z)$, 
                    $H(\psi)^{-1}([\zeta]) = [\psi^\#(\zeta)]$ for $\zeta\in C(\X\times \N, \Z)$ and $H(\varphi)\circ H(\iota_0) = H(\iota_0)\circ H(\psi)$.
            \end{itemize}
    \end{enumerate}
\end{lemma}

\begin{proof}
    (i): 
    Since $\Lambda_\X$ and $\Lambda_\Y$ are conjugate there is a surjective and almost injective sliding block code $\varphi\colon \X\LRA \Y$, 
    cf. Lemma~\ref{lem:two-sided-conj}.
    The map $\varphi^*\colon C(\Y, \Z) \LRA C(\X, \Z)$ given by $\varphi^*(g) = g\circ \varphi$ for $g\in C(\Y, \Z)$ is an injective homomorphism.

    Since $\varphi$ is a sliding block code the map $H(\varphi)\colon H^\Y\LRA H^\X$ given by $H(\varphi)[g] = [g\circ \varphi]$ is well-defined and injective.
    In order to see that $H(\varphi)$ is surjective, recall that $\varphi$ is almost injective and pick $\ell\in \N$ accordingly.
    Take $f\in C(\X, \Z)$.
    Define a map $g\colon \Y\LRA \N$ by $g(y) = f\circ \sigma_\X^\ell(\varphi^{-1}(y))$, for $y\in \Y$.
    Since $\varphi$ is almost injective with lag $\ell$ this is well-defined and $g$ is continuous,
    and $H(\varphi)[g] = [f\circ \sigma_\X^\ell] = [f]$.
    Hence $H(\varphi)$ is surjective.
    It is straightforward to verify that $H(\varphi)(H^\Y_+) = H^\X_+$ and $H(\varphi)([1_\Y]) = [1_\X]$.

    (ii): 
    By (the proof of) Theorem~\ref{thm:two-sided-conjugacy}, there is a surjective sliding block code $\tilde{\varphi}\colon \tilde{\X}\LRA\tilde{\Y}$ such that
    $\varphi\circ \pi_\X = \pi_\Y \circ \tilde{\varphi}$ and a map $\omega\colon \X\times\N \LRA \N$ such that the map 
    $\Psi\colon \G_\X\times \R \LRA \G_\Y\times \R$ defined by 
    \begin{align}\label{eq:omega}
        \Psi\big( (\tilde{x},k), n, (\tilde{y},l) \big) = 
        \big( (\tilde{\varphi}(\tilde{x}), \omega(\pi_\X(\tilde{x}),k)), n, (\tilde{\varphi}(\tilde{y}), \omega(\pi_\X(\tilde{y}),l)) \big)
    \end{align}
    is a groupoid isomorphism, and the map $\psi\colon \X\times \N\LRA \Y\times \N$ defined by 
    \[
        \psi(x, n) = (\varphi(x), \omega(x,n)),
    \]
    is a homeomorphism satisfying $\psi\circ \pi_{\X\times \N} = \pi_{\Y\times \N}\circ \Psi^{(0)}$.

    (iii): 
    Choose $\ell\in\N$ such that $\varphi(x) = \varphi(x') \implies \sigma_\X^\ell(x) = \sigma_\X^\ell(x')$
    and let $\omega\colon \X\times \N \LRA \N$ be the map from~\eqref{eq:omega}.
    Define $\omega'\colon \Y\times\N\LRA\N$ by letting $\omega'(y,n) = m$ where $\psi^{-1}(y,n) = (x,m)$ for some $x\in \X$. 
    Since $\psi$ is a homeomorphism, $\omega'$ is continuous. 
    Let $\psi^*\colon C(\Y\times \N, \Z) \LRA C(\X\times \N, \Z)$ be the map defined by
    \[
        \psi^*(\eta)(x, n) = \sum_{r = 0}^{\omega(x,n)} \eta(S_\Y^r(\psi(x, n))) - \sum_{r = 0}^{\omega(x,n-1)} \eta(S_\Y^r(\psi(x, n-1))),
    \]
    for $\eta\in C(\Y\times \N, \Z)$ and $(x, n)\in \X\times \N$ with $n\ge 1$, and  
    \[
        \psi^*(\eta)(x, 0) = \sum_{r = 0}^{\omega(x,0)+1} \eta(S_\Y^r(\psi(x, 0))) - \sum_{r = 0}^{\omega(\sigma_\X(x),0)} \eta(S_\Y^r(\psi(\sigma_\X(x),0))),
    \]
    for $\eta\in C(\Y\times \N, \Z)$ and $x\in\X$. 
    Let $\psi^\#\colon C(\X\times \N, \Z) \LRA C(\Y\times \N, \Z)$ be the map defined by
    \[
        \psi^\#(\zeta)(y, n) = \sum_{r = 0}^{\omega'(y,n) + \ell} \zeta(S_\X^r(\psi^{-1}(y, n))) - \sum_{r = 0}^{\omega'(y,n-1) + \ell} \zeta(S_\X^r(\psi^{-1}(y, n-1))),
    \]
    for $\zeta\in C(\X\times \N, \Z)$ and $(y, n)\in \Y\times \N$ with $n\ge 1$, and  
     \[
        \psi^\#(\zeta)(y, 0) = 
        \sum_{r = 0}^{\omega'(y,0) + \ell + 1} \zeta(S_\X^r(\psi^{-1}(y, 0))) - \sum_{r = 0}^{\omega'(\sigma_\Y(y),0) + \ell} \zeta(S_\X^r(\psi^{-1}(\sigma_\Y(y),0))),
    \]
    for $\zeta\in C(\X\times \N, \Z)$ and $y\in\Y$. 
    It is straightforward to check that $\psi^*$ and $\psi^\#$ are homomorphisms, 
    and that $\kappa_{\X\times\N}(\psi^*(\eta)) = \kappa_{\Y\times\N}(\eta)\circ\Psi$ for $\eta\in C(\Y\times \N, \Z)$, 
    and $\kappa_{\Y\times\N}(\psi^\#(\zeta)) = \kappa_{\X\times\N}(\zeta)\circ\Psi^{-1}$ for $\zeta\in C(\X\times \N, \Z)$.
    
    Let $\eta\in C(\Y\times \N, \Z)$ and observe that
    \[
        \psi^*(\eta - \eta\circ S_\Y)(x, n) = \eta(\omega(x,n)) - \eta(\omega(x,n-1)) =  (\eta-\eta\circ S_\X)(x,n)
    \]
    for $(x, n)\in \X\times \N$ with $n\ge 1$, and 
    \[
        \psi^*(\eta - \eta\circ S_\Y)(x, 0) = \eta(\omega(x,0)) - \eta(\omega(\sigma_\X(x),0)) = (\eta-\eta\circ S_\X)(x,0)
    \]
    for $x\in \X$.
    Hence $\psi^*$ induces a map $H(\psi)\colon H^{\Y\times \N} \LRA H^{\X\times \N}$ given by $H(\psi)([\eta]) = [\psi^*(\eta)]$ for $\eta\in C(\Y\times \N, \Z)$.

    Suppose $\eta\in C(\Y\times \N, \Z)$ is supported on $\Y\times \{0\}$. 
    Then
    \[
        \iota_0^*(\psi^*(\eta))(x) = \eta(\varphi(x),0) = \varphi^*(\iota_0^*(\eta))(x)
    \]
    for $x\in \X$. 
    Since any element in $H^{\Y\times \N}$ can be represented by a map $\eta\in C(\Y\times \N, \Z)$ which is supported on $\Y\times \{0\}$,
    it follows that $H(\varphi)\circ H(\iota_0) = H(\iota_0)\circ H(\psi)$.
    Therefore, $H(\psi)$ is a positive isomorphism.

    Suppose $\zeta\in C(\X\times \N, \Z)$ is is supported on $\Y\times \{0\}$. 
    Then
    \[
        \phi^*(\iota_0^*(\psi^\#(\zeta)))(x) = \psi^\#(\zeta)(\varphi(x),0) = \zeta(x,0) = \iota_0^*(\zeta)(x)
    \]
    for $x\in \X$. 
    It follows that $(H(\varphi)\circ H)(\iota_0)([\psi^\#(\zeta)]) = H(\iota_0)([\zeta])$. 
    Since $H(\varphi)\circ H(\iota_0) = H(\iota_0)\circ H(\psi)$ and $H(\iota_0)$ is an isomorphism, we conclude that $H(\psi)^{-1}([\zeta]) = [\psi^\#(\zeta)]$. 
    Any element in $H^{\Y\times \N}$ can be represented by a map $\eta\in C(\Y\times \N, \Z)$ which is supported on $\Y\times \{0\}$, 
    so it follows that $H(\psi)^{-1}([\zeta]) = [\psi^\#(\zeta)]$ for every $\zeta\in C(\X\times \N, \Z)$.
\end{proof}

Let $f\colon \X\LRA \N_+$ be a continuous map.
Following~\cite{Ma2012a}, we consider the space
\[
    \X_f = \{ (x, i)\in \X\times \N \mid i < f(x)\}
\]
with the shift operation $\sigma_f\colon \X_f\LRA \X_f$ given by
\[
    \sigma_f(x, i) = 
    \begin{cases}
        (x, i - 1) & i > 0, \\
        \big(\sigma_\X(x), f(\sigma_\X(x)) - 1\big) & i = 0,
    \end{cases}
\]
for $(x, i)\in \X_f$.
We equip $\X_f$ with the subspace topology of $\X\times \N$ with the product topology where $\N$ is endowed with the discrete topology.
Then $\X_f$ is compact and Hausdorff, and $\sigma_f$ is surjective if and only if $\sigma_\X$ is surjective.
If $\A$ is the alphabet of $\X$, then the pair $(\X_f, \sigma_f)$ is conjugate to a shift space $\X^f = j(\X)$ 
over $\A\times \{0, 1,\ldots,\max\{f(x) \mid x\in \X\} - 1\}$ where $j\colon \X_f\LRA {(\A\times \{0, 1,\ldots,\max\{f(x) \mid x\in \X\} - 1\})}^\N$ 
is the injective sliding block code given by
\[
    j(x, i) = (x_0, i)(x_0, i - 1)\cdots (x_0, 0) (x_1, f(\sigma_\X(x)) - 1) \cdots (x_1, 0) \cdots
\]
for $x = x_0 x_1 \cdots\in \X$ and $i = 0, 1,\ldots, f(x) - 1$.
By a slight abuse of notation, we shall identify $\X_f$ and $\X^f$ and consider the two-sided subshift $\Lambda_{\X_f}$ as well as the cover $\tilde{\X_f}$.
Note that $\Lambda_\X$ and $\Lambda_{\X_f}$ are flow equivalent, cf.~\cite[Section 5]{CEOR}.
A similar construction applies to two-sided subshifts.

We shall make use of the following characterization of flow equivalence.
This is probably known to experts but we have not been able to find a proper reference.

\begin{lemma}
    A pair of two-sided subshifts $\Lambda_\X$ and $\Lambda_\Y$ are flow equivalent, if and only if there are 
    continuous maps $f\in C(\X, \N_+)$ and $g\in C(\Y, \N_+)$ such that $\Lambda_{\X_f}$ and $\Lambda_{\Y_g}$ are conjugate.
\end{lemma}

\begin{proof}
Suppose first that there are continuous maps $f\in C(\X, \N_+)$ and $g\in C(\Y, \N_+)$ such that $\Lambda_{\X_f}$ and $\Lambda_{\Y_g}$ are conjugate.
It is well-known that $\Lambda_\X$ is flow equivalent to $\Lambda_{\X_f}$,
and that $\Lambda_\Y$ is flow equivalent to $\Lambda_{\Y_g}$, cf.~\cite[Section 5]{CEOR}, so it follows that $\Lambda_{\X_f}$ and $\Lambda_{\Y_g}$ are flow equivalent.

If $\Lambda_\X$ and $\Lambda_\Y$ are flow equivalent,
then there is a compact metric space $Z$ with a flow $\gamma\colon Z\times \mathbb{R}\LRA Z$
and cross sections $\mathcal{X}$ and $\mathcal{Y}$ which are conjugate to $\Lambda_\X$ and $\Lambda_\Y$, respectively,
cf., e.g.,~\cite{Parry-Sullivan, Boyle-Carlsen-Eilers}.
Let $h_\X\colon \Lambda_\X\LRA \mathcal{X}$ and $h_\Y\colon \Lambda_\Y\LRA \mathcal{Y}$ be such conjugacies.

Set $A = \mathcal{X} \cup \mathcal{Y}$.
Consider the return time function $\tau_{\mathcal{X}}\colon Z\LRA \mathbb{R}$ given by
\[
    \tau_{\mathcal{X}}(z) = \min \{ t > 0 \mid \gamma(z, t)\in \mathcal{X}\},
\]
for $z\in Z$, and define the map $\bar{f}\colon \Lambda_\X \LRA \N$ by
\[
    \bar{f}(\textrm{x}) = |\{ t\in (0,\tau_{\mathcal{X}}(h_\X(\textrm{x}))) \mid \gamma(h_\X(\textrm{x}), t)\in \mathcal{Y} \}|
\]
for $\textrm{x}\in \Lambda_\X$.
Then $\bar{f}$ is continuous and $f\geq 1$.
Moreover, $(\Lambda_\X)_{\bar{f}}$ is conjugate to $A$ by construction.
By continuity, there is an integer $n\in \N$ such that $\textrm{x}_{[-n, n]} = \textrm{x}_{[-n, n]}$ implies $\bar{f}(\textrm{x}) = \bar{f}(\textrm{x})$.
It follows that there is a well-defined continuous map $f\colon \X\LRA \N$ satisfying 
\[
    f(\textrm{x}_{[0, \infty)}) = \bar{f}(\sigma_{\Lambda_\X}^n(\textrm{x}))
\]
for $\textrm{x}\in \Lambda_\X$.
Then $(\Lambda_\X)_{\bar{f}}$ is conjugate to $(\Lambda_\X)_{\bar{f}\circ \sigma_{\Lambda_\X}^n}$,
and $(\Lambda_\X)_{\bar{f}\circ \sigma_{\Lambda_\X}^n}$ is conjugate to $\Lambda_{\X_f}$.
In particular, $\Lambda_{\X_f}$ is conjugate to $A$.

A similar argument shows that there is a continuous map $g\colon C(\Y, \N_+)$ such that $\Lambda_{\Y_g}$ is conjugate to $A$.
It follows that $\Lambda_{\X_f}$ and $\Lambda_{\Y_g}$ are conjugate.
\end{proof}

\begin{lemma}\label{lem:X_f}
    Let $\X$ be a one-sided shift space and let $f\colon \X\LRA \N_+$ be continuous.
    Then
    \begin{enumerate}
        \item[(i)] there are an injective sliding block code $\iota_f\colon \X\LRA \X_f$ and a surjective homomorphism $\iota_f^*\colon C(\X_f, \Z) \LRA C(\X, \Z)$ given by
            \begin{align}\label{eq:iotastar}
            \iota_f^*(\xi)(x) = \sum_{r = 0}^{f(\sigma_\X(x)) - 1} \xi(\sigma_f^r(\iota_f(x))),
        \end{align}
        and a positive isomorphism $H(\iota_f)\colon H^{\X_f}\LRA H^\X$ given by $H(\iota_f)([\xi]) = [\iota_f^*(\xi)]$;
        \item[(ii)] there are  
            \begin{itemize}
                \item a groupoid isomorphism $\Psi_f\colon \G_\X\times \R \LRA \G_{\X_f}\times \R$ and a homeomorphism $\psi\colon \X\times \N \LRA \X_f\times \N$
                    satisfying $\psi\circ \pi_{\X\times \N} = \pi_{\Y\times \N} \circ \Psi_f^{(0)}$;
                \item a homomorphism $\psi^*\colon C(\X_f\times \N, \Z) \LRA C(\X\times \N, \Z)$ 
                    satisfying $\kappa_{\X\times\N}(\psi^*(\xi)) = \kappa_{\X_f\times\N}(\xi)\circ\Psi_f$ for $\xi\in C(\X_f\times \N, \Z)$;
                \item a homomorphism $\psi^\#\colon C(\X\times \N, \Z) \LRA C(\X_f\times \N, \Z)$ 
                    satisfying $\kappa_{\X_f\times\N}(\psi^\#(\zeta)) = \kappa_{\X\times\N}(\zeta)\circ\Psi_f^{-1}$ for $\zeta\in C(\X\times \N, \Z)$;
                \item a positive isomorphism $H(\psi)\colon H^{\X_f\times \N} \LRA H^{\X\times \N}$
                    such that $H(\psi)([\xi])=[\psi^*(\xi)]$, $H(\psi)^{-1}([\zeta])=[\psi^\#(\zeta)]$,
                    and $H(\iota_f)\circ H(\iota_0) = H(\iota_0)\circ H(\psi)$.
            \end{itemize}
        \end{enumerate}
\end{lemma}

\begin{proof}
    (i):
    The inclusion $\iota_f\colon \X\LRA \X_f$ given by $\iota_f(x) = (x, 0)$ is an injective sliding block code,
    and $\iota_f^*\colon C(\X_f, \Z) \LRA C(\X, \Z)$ given by~\eqref{eq:iotastar} is a surjective homomorphism.
    Since 
    \[
        \iota_f^*(\xi - \xi\circ \sigma_\X)(x) = \xi(x, 0) - \xi(\sigma_\X(x), 0) = \iota_0^*(\xi)(x) - \iota_f^*(\xi)(\sigma_\X(x)),
    \]
    for $\xi\in C(\X_f, \Z)$ and $x\in \X$, 
    the map $\iota_f^*$ induces a well-defined surjective map $H(\iota_f)\colon H^{\X_f} \LRA H^\X$ given by
    $H(\iota_f)([\xi]) = [\iota_f^*(\xi)]$ for $\xi\in C(\X_f, \Z)$.
  
    To see that $H(\iota_f)$ is injective, notice that any element of $H^{\X_f}$ can be represented by a map $\xi\in C(\X_f, \Z)$ 
    which is supported on $\X\times \{0\} \subset \X_f$. 
    Suppose $\xi\in C(\X_f, \Z)$ is supported on $\X\times \{0\} \subset \X_f$ and $\iota_f^*(\xi)(x) = b(x) - b(\sigma_\X(x))$ for some $b\in C(\X, \Z)$. 
    Let $\eta\in C(\X_f, \Z)$ be given by $\eta(x, n) = 0$ for $n > 0$ and $\eta(x, 0) = b(x)$.
    Then $\xi(x, 0) = \eta(x, 0) - \eta\circ \sigma_f^{f(\sigma_\X(x))}(x, 0)$, so $\xi$ is cohomologous to zero.

    Note that $\iota_f^*(\xi)\geq 0$ when $\xi \geq 0$.
    Conversely, let $g\in C(\X, \Z)$ and take $\xi\in C(\X_f, \Z)$ such that $\xi(x, i) = 0$ for all $i > 0$ and $\xi(x, 0) = g(x)$.
    Then $\iota_f^*(\xi) = g$ and $\xi \geq 0$ if $g\geq 0$.
    Hence $H(\iota_f^*)$ is a positive isomorphism.

    (ii):
    Define $\psi\colon \X\times \N\LRA \X_f\times \N$ by
    \[
        \psi(x,j) = \big( (x,i),k \big)
    \]
    where $i,k\in \N$ with $i<f(x)$ and $j = kf(x) + i$. 
    Then $\psi$ is a homeomorphism. 
    
    Define $\Psi_f\colon \G_\X\times \R\LRA \G_{\X_f}\times \R$ by
    \[
        \Psi_f\left( (\tilde{x},j), p, (\tilde{x}',j') \right) = \left( \big( (\tilde{x},i),k \big), l - l', \big( (\tilde{x}',i'),k' \big) \right)
    \]
    for $\big( (\tilde{x},j), p, (\tilde{x}',j') \big)\in \G_\X\times \R$ and $s, s'\in \N$ 
    such that $\sigma_{\tilde{\X}}^s(\tilde{x}) = \sigma_{\tilde{\X}}^{s'}(\tilde{x}')$ and $p = s - s'$.
    Here, $i, i', k, k'\in \N$ with $i < f(\pi_\X(\tilde{x}))$ and $i'<f(\pi_\X(\tilde{x}'))$,
    and $j = k f(\pi_\X(\tilde{x})) + i$ and $j' = k' f(\pi_\X(\tilde{x}')) + i'$, and
    \[
        l  = i  + \sum_{r = 1}^s f(\sigma_\X^r(\pi_\X(\tilde{x}))), \quad
        l' = i' + \sum_{r = 1}^{s'} f(\sigma_\X^r(\pi_\X(\tilde{x}'))).
    \]
    Then $\Psi_f$ is a groupoid isomorphism such that $\psi\circ \pi_{\X\times \N} = \pi_{\Y\times \N} \circ \Psi_f^{(0)}$.
    
    (iii):
    Let $\psi^*\colon C(\X_f\times \N, \Z) \LRA C(\X\times \N, \Z)$ be defined by
    \[
        \psi^*(\xi)(x,j) = \sum_{r=0}^{k + 1} \xi(S_f^r(\psi(x,j))) - \sum_{r = 0}^{k} \xi(S_f^r(\psi(x,j - 1)))
    \]
    for $\xi\in C(\X_f\times \N, \Z)$ and $(x, j)\in \X\times \N$ with $j\ge 1$, where $k$ is the integer part of $j/f(x)$, and 
    \[
        \psi^*(\xi)(x,0) = \sum_{r = 0}^{f(\sigma_\X(x)) - 1} \xi(S_f^r(\psi(x,0)))
    \]
    for $\xi\in C(\X_f\times \N, \Z)$ and $x\in \X$. 
    Then $\psi^*$ is a homomorphism such that $\kappa_{\X\times\N}(\psi^*(\xi)) = \kappa_{\X_f\times\N}(\xi)\circ \Psi_f$ for $\xi\in C(\X_f\times \N, \Z)$.

    Define $\psi^\#\colon C(\X\times \N, \Z) \LRA C(\X_f\times \N, \Z)$ by
    \[
        \psi^\#(\zeta)((x,i),k) = \sum_{j = (k-1)f(x) + i + 1}^{k f(x) + i} \zeta(x,j)
    \]
    for $\zeta\in C(\X\times \N, \Z)$ and $\big( (x,i), k \big)\in \X_f\times \N$ with $k\ge 1$,
    \[
        \psi^\#(\zeta)\big( (x,i),0 \big) = \zeta(x,i)
    \]
    for $\zeta\in C(\X\times \N, \Z)$ and $(x,i)\in \X_f$ with $i\ge 1$, and
    \[
        \psi^\#(\zeta)\big( (x,0),0 \big) = \zeta(x,0) - \sum_{j = 1}^{f(\sigma_\X(x)) - 1} \zeta(\sigma_\X(x),j)
    \]
    for $\zeta\in C(\X\times \N, \Z)$ and $x\in\X$. 
    Then $\psi^\#$ is a homomorphism such that $\kappa_{\X_f\times\N}(\psi^\#(\zeta)) = \kappa_{\X\times\N}(\zeta)\circ\Psi_f^{-1}$ for $\zeta\in C(\X\times \N, \Z)$

    Since
    \[
        \psi^*(\xi - \xi\circ S_f)(x, j) = \xi((x, 0), 0) - \xi( (\sigma_\X(x), 0), 0) = \xi( (x, 0), 0) - \xi\circ S_\X^{f(\sigma_\X(x))}( (x, 0), 0),
    \]
    for $\xi\in C(\X_f\times \N, \Z)$ and $(x, j)\in \X\times \N$, 
    $\psi^*$ induces a well-defined map $H(\psi)\colon H^{\X_f\times \N}\LRA H^{\X\times \N}$ given by
    $H(\psi)([\xi]) = [\psi^*(\xi)]$ for $\eta\in C(\X_f\times \N, \Z)$. 
    Since $\iota_f^*\circ\iota_0^* = \iota_0^*\circ\psi^*$, it follows that $H(\iota_f)\circ H(\iota_0) = H(\iota_0)\circ H(\psi)$. 
    Since $H(\iota_0)$ and $H(\iota_f)$ are positive isomorphisms, $H(\psi)$ is also a positive isomorphism.

    Suppose $\zeta\in C(\X\times \N, \Z)$ is supported on $\X\times\{0\}$. 
    Then 
    \[
        \iota_f^*(\iota_0^*(\psi^\#(\zeta)))(x) = \zeta(x,0) = \iota_0^*(\zeta)(x),
    \]
    for every $x\in\X$. 
    Since every element of $H_+^{\X\times\N}$ can be represented by a map $\zeta\in C(\X\times \N, \Z)$ which is supported on $\X\times\{0\}$, 
    this shows that $H(\iota_f)\circ H(\iota_0)([\psi^\#(\zeta)]) = H(\iota_0)([\zeta])$ for every $\zeta\in C(\X\times \N, \Z)$. 
    Since $H(\iota_f)\circ H(\iota_0) = H(\iota_0)\circ H(\psi)$ and $H(\iota_0)$ is an isomorphism, 
    it follows that $H(\psi)^{-1}([\zeta])=[\psi^\#(\zeta)]$ for $\zeta\in C(\X\times \N, \Z)$. 
\end{proof}

Let us say that a stabilizer-preserving continuous orbit equivalence $(h,l_\X,k_\X,l_\Y,k_\Y)$ from $\X$ to $\Y$ is \emph{positive} 
if $[l_\X - k_\X]\in H^\X_+$ and $[l_\Y - k_\Y]\in H^\Y_+$. 

\begin{lemma}\label{lem:pos}
Let $\X$ and $\Y$ be one-sided shift spaces and let $(h,l_\X,k_\X,l_\Y,k_\Y)$ be a positive stabilizer-preserving continuous orbit equivalence from $\X$ to $\Y$. 
Then $(h,l_\X,k_\X,l_\Y,k_\Y)$ is least period preserving.
\end{lemma}

\begin{proof}
Since $[l_\X-k_\X]\in H^\X_+$ and $[l_\Y-k_\Y]\in H^\Y_+$, there are $b_\X\in C(\X,\Z)$ and $n_\X\in C(\X,\N)$ such that $l_\X - k_\X = n_\X + b_\X - b_\X\circ\sigma_\X$,
and $b_\Y\in C(\Y,\Z)$ and $n_\Y\in C(\Y,\N)$ such that $l_\Y - k_\Y = n_\Y + b_\Y - b_\Y\circ \sigma_\Y$. 
If $x\in\X$ is periodic with $\lp(x) = p$, then 
\begin{align*}
    l_\X^{(p)}(x) - k_\X^{(p)}(x)
    &= \sum_{i = 0}^{p - 1}\left( l_\X(\sigma_\X^i(x)) - k_\X(\sigma_\X^i(x)) \right)\\
    &= \sum_{i = 0}^{p - 1}\left( n_\X(\sigma_\X^i(x) +b_\X(\sigma_\X^i(x) -b_\X(\sigma_\X^{i+1}(x) \right)\\
    &= \sum_{i = 0}^{p - 1} n_\X(\sigma_\X^i(x)\geq 0.  
\end{align*}
Since $(h,l_\X,k_\X,l_\Y,k_\Y)$ is stabilizer-preserving, we thus have that 
\[
    l_\X^{(p)}(x) - k_\X(p)(x) = |l_\X^{(p)}(x) - k_\X(p)(x)| = \lp(h(x)).
\]
A similar argument shows that $l_\Y^{(\lp(y))}(y) - k_\Y(\lp(y))(y) = \lp(h^{-1}(y))$ for any periodic $y\in \Y$. 
\end{proof}

\begin{corollary}\label{cor:flow}
Let $\Lambda_\X$ and $\Lambda_\Y$ be two-sided subshifts and suppose there is a positive stabilizer-preserving continuous orbit equivalence from $\X$ to $\Y$. 
Then $\Lambda_\X$ and $\Lambda_\Y$ are flow equivalent. 
\end{corollary}

\begin{proof}
Let $(h,l_\X,k_\X,l_\Y,k_\Y)$ be a positive stabilizer-preserving continuous orbit equivalence from $\X$ to $\Y$. 
It follows from Lemma~\ref{lem:pos} that $(h,l_\X,k_\X,l_\Y,k_\Y)$ is least period preserving, 
and thus from~\cite[Proposition 3.2]{CEOR} that $\Lambda_\X$ and $\Lambda_\Y$ are flow equivalent. 
\end{proof}

The proof of~\cite[Theorem 5.11]{MM-zeta} shows that any continuous orbit equivalence between shifts of finite type with no isolated points
is least period preserving and positive. 
However, if $\X = \Y$ is the shift space with only one point, then $(\id,1,0,0,1)$ is a stabilizer-preserving continuous orbit equivalence from $\X$ to $\Y$
which is not positive. 
It follows from~\cite[Proposition 4.5 and Proposition 4.7]{CEOR} that if $\X$ and $\Y$ are shifts of finite type that are continuously orbit equivalent, 
then there is a least period preserving continuous orbit equivalence between $\X$ and $\Y$. 
We do not know if there are shifts spaces $\X$ and $\Y$ that are continuously orbit equivalent, 
but for which there is no positive stabilizer-preserving continuous orbit equivalence between $\X$ and $\Y$.

\begin{remark}\label{remark:CRST}
  Suppose $\G$ is a second-countable locally compact Hausdorff \'etale groupoid such that $\Iso(\G)^\circ$ is abelian and torsion-free, $\Gamma$ is an abelian group,
  and that $c\colon \G\LRA\Gamma$ is a cocycle. 
  Then $c$ induces an action $\beta^c$ of the dual $\widehat{\Gamma}$ of $\Gamma$ on $\mathrm{C^*_r}(\G)$ such that $\beta^c_\gamma(f)(\eta) = \gamma(c(\eta))f(\eta)$ 
  for $\gamma\in\widehat{\Gamma}$, $f\in \mathrm{C^*_r}(\G)$, and $\eta\in\G$. 
  In~\cite[Section 4]{CRST}, a groupoid $\mathcal{H}(\mathrm{C^*_r}(\G),C_0(\G^{(0)}),\beta^c)$ consisting of equivalence classes of pairs $(n,\phi)$, 
  where $n$ is normalizer of $C_0(\G^{(0)})$ in $\mathrm{C^*_r}(\G)$ that is homogeneous with respect to $\beta^c$, and $\phi$ is a character of $C_0(\G^{(0)})$, is constructed, 
  and it is shown in~\cite[Proposition 6.5]{CRST} that there is an isomorphism $\theta_{(\mathrm{C^*_r}(\G),C_0(\G^{(0)}),\beta^c)}\colon \G \LRA \mathcal{H}(\mathrm{C^*_r}(\G),C_0(\G^{(0)}),\beta^c)$ 
  (it is in~\cite{CRST} not assumed that $\Gamma$ is abelian and $\beta^c$ is a coaction of $\Gamma$ rather than an action of $\widehat{\Gamma}$). 

  This is used in~\cite[Theorem 6.2]{CRST} to prove that if $\G'$ is another second-countable locally compact Hausdorff étale groupoid, and $d\colon \G'\to\Gamma$ is a cocycle 
  such that there is a $^*$-isomorphism $\Phi\colon \mathrm{C^*_r}(\G)\LRA \mathrm{C^*_r}(\G')$ such that $\Phi(C_0(\G^{(0)})) = C_0((\G')^{(0)})$ and 
  $\beta_\gamma^d\circ\Phi = \Phi\circ\beta_\gamma^c$ for all $\gamma\in\hat{\Gamma}$, then there is an isomorphism $\Psi\colon \G\LRA\G'$ such that $d\circ\Psi = c$.

  If we let $c_0$ denote the unique cocycle from $\G$ to the abelian group $\{0\}$, then any normalizer of $C_0(\G^{(0)}$ in $\mathrm{C^*_r}(\G)$ is homogeneous with respect to $c_0$. 
  In particular, a normalizer $n$ that is homogeneous with respect to $c$, is also homogeneous with respect to $c_0$, 
  and there is a homomorphism $\Phi_\pi\colon \mathcal{H}(\mathrm{C^*_r}(\G),C_0(\G^{(0)}),\beta^c)\LRA \mathcal{H}(\mathrm{C^*_r}(\G),C_0(\G^{(0)}),\beta^{c_0})$ that sends $[n,\phi]$ to $[n,\phi]$. 
  Since $\theta_{(\mathrm{C^*_r}(\G),C_0(\G^{(0)}),\beta^{c_0})} = \Phi_\pi\circ\theta_{(\mathrm{C^*_r}(\G),C_0(\G^{(0)}),\beta^c)}$, it follows that $\Phi_\pi$ is an isomorphism. 

  Therefore, the isomorphism $\Psi'\colon \G\LRA\G'$ constructed in~\cite[Theorem 3.3]{CRST} is equal to the isomorphism $\Psi\colon \G\LRA\G'$ constructed 
  in~\cite[Theorem 6.2]{CRST} such that $d\circ \Psi = c$.
\end{remark}

We are now ready to characterize flow equivalence of general two-sided subshifts.
The equivalence (i) $\iff$ (iv) in Theorem~\ref{thm:flow} below is a generalization of~\cite[Theorem 5.3 (5) $\iff$ (6)]{CEOR}
which is formulated for shifts of finite type.

\begin{theorem}\label{thm:flow}
    Let $\Lambda_\X$ and $\Lambda_\Y$ be two-sided subshifts.
    The following are equivalent:
    \begin{enumerate}
        \item[(i)] the two-sided subshifts $\Lambda_\X$ and $\Lambda_\Y$ are flow equivalent;
        \item[(ii)] there are
            \begin{itemize}
                \item a groupoid isomorphism $\Psi\colon \G_\X\times \R \LRA \G_\Y\times \R$ and 
                    a homeomorphism $\psi\colon \X\times \N \LRA \Y\times \N$ such that $\psi\circ \pi_{\X\times \N} = \pi_{\Y\times \N}\circ \Psi^{(0)}$;
                \item a homomorphism $\psi^*\colon C(\Y\times\N,\Z)\LRA C(\X\times\N,\Z)$ such that
                    \begin{align}\label{eq:flow-psistar}
                        \kappa_{\X\times\N}(\psi^*(\eta)) = \kappa_{\Y\times\N}(\eta)\circ\Psi,
                    \end{align}
                    for $\eta\in C(\Y\times\N,\Z)$;
                \item a homomorphism $\psi^\#\colon C(\X\times\N,\Z)\LRA C(\Y\times\N,\Z)$ such that 
                    \begin{align}\label{eq:flow-psihash}
                        \kappa_{\Y\times\N}(\psi^\#(\zeta)) = \kappa_{\X\times\N}(\zeta)\circ\Psi^{-1},
                    \end{align}
                    for $\zeta\in C(\X\times\N,\Z)$; and
                \item a positive isomorphism $H(\psi)\colon H^{\Y\times\N}\LRA H^{\X\times\N}$ such that 
                    $H(\psi)([\eta]) = [\psi^*(\eta)]$ $\eta\in C(\Y\times\N,\Z)$, and $H(\psi)^{-1}([\zeta]) = [\psi^\#(\zeta)]$ for $\zeta\in C(\X\times\N,\Z)$.
            \end{itemize}     
        \item[(iii)] there are
            \begin{itemize}
                \item a $^*$-isomorphism $\Phi\colon \OO_{\X}\otimes \K\LRA \OO_{\Y}\otimes \K$ such that $\Phi(C(\X)\otimes c_0) = C(\Y)\otimes c_0$;
                \item a homomorphism $\psi^*\colon C(\Y\times\N,\Z)\LRA C(\X\times\N,\Z)$ such that
                    \[
                        \Phi\circ\beta_z^{\kappa_{\X\times\N}(\psi^*(\eta))} = \beta_z^{\kappa_{\Y\times\N}(\eta)}\circ\Phi,
                    \]
                    for $\eta\in C(\Y\times\N,\Z)$ and $z\in\T$;
                \item a homomorphism $\psi^\#\colon C(\X\times\N,\Z)\LRA C(\Y\times\N,\Z)$ such that 
                    \[
                        \Phi\circ\beta_z^{\kappa_{\X\times\N}(\zeta)} = \beta_z^{\kappa_{\Y\times\N}(\psi^\#(\zeta))}\circ\Phi,
                    \]
                    for $\zeta\in C(\X\times\N,\Z)$ and $z\in\T$; and
                \item a positive isomorphism $H(\psi)\colon H^{\Y\times\N}\LRA H^{\X\times\N}$ such that 
                    $H(\psi)([\eta]) = [\psi^*(\eta)]$ $\eta\in C(\Y\times\N,\Z)$, and $H(\psi)^{-1}([\zeta])=[\psi^\#(\zeta)]$ for $\zeta\in C(\X\times\N,\Z)$;
            \end{itemize}
        \item[(iv)] there are $f\in C(\X,\N_+)$ and $g\in C(\Y,\N_+)$ such that 
            there is a positive stabilizer-preserving continuous orbit equivalence between $\X_f$ and $\Y_g$.  
    \end{enumerate}
\end{theorem}

\begin{proof}
    (i) $\implies$ (ii):
    Suppose $\Lambda_\X$ and $\Lambda_\Y$ are flow equivalent.
    Then there are $f\in C(\X, \N_+)$ and $g\in C(\Y, \N_+)$ such that $\Lambda_{\X_f}$ and $\Lambda_{\Y_g}$ are conjugate. 
    It therefore follows from Lemmas~\ref{lem:conjugacy,SBC-homeo} and~\ref{lem:X_f} that (ii) holds.

    (ii)$\implies$(iv): 
    We shall identify ${\G_\X\times \R}^{(0)} = \tilde{\X}\times \N$.
    Since $\tilde{\X}$ is compact and $\Psi$ is continuous, there is an integer $n\in \N$ such that 
    $\Psi^{(0)}(\tilde{\X}\times \{0\}) \subset \tilde{\Y}\times \{0,\ldots, n -1\}$.

    Define $g\in C(\Y,\N_+)$ to be constantly equal to $n$. Then $\phi_{\Y_g}\colon\Y\times\{0,\dots,n-1\}\LRA \Y_g$ given by
    \begin{align}\label{eq:phi_(Y_g)}
        \phi_{\Y_g}(y,k) = (y,k)    
    \end{align}
    for $(y,k)\in \Y_g$, is a homeomorphism and $\Phi_{\Y_g}\colon \G_\Y\times \R|_{\tilde{\Y}\times \{0,\ldots,n - 1\}}\LRA \G_{\Y_g}$ defined by 
    \[
        \Phi_{\Y_g}\big( (\tilde{y},k), m, (\tilde{y}',l) \big) = \big( (\tilde{y},k), k + mn - l, (\tilde{y}',l)\big),
    \]
    for $\big( (\tilde{y},k), m, (\tilde{y}',l) \big)\in \G_\Y\times \R|_{\tilde{\Y}\times \{0,\ldots,n - 1\}}$,
    is an isomorphism such that $\phi_{\Y_g}\circ \pi_{\Y\times \N} = \pi_{\Y_g}\circ\Phi_{\Y_g}^{(0)}$.

    Define $\tilde{f}\colon \tilde{\X}\LRA \N_+$ by
    \[
        \tilde{f}(\tilde{x}) = | \{ k\in \N : \Psi^{(0)}(\tilde{x}, k)\in \tilde{\Y}\times \{0,\ldots, n - 1\} \}|
    \]
    for $\tilde{x}\in \tilde{\X}$.
    Then $\tilde{f}$ is continuous and $\tilde{f}\geq 1$.
    Note that if $\pi_\X(\tilde{x}) = \pi_\X(\tilde{x}')$ then the condition $\psi\circ \pi_{\X\times \N} = \pi_{\Y\times \N}\circ \Psi^{(0)}$ ensures that
    \[
        \{ k\in \N \mid \Psi^{(0)}(\tilde{x}, k)\in \tilde{\Y}\times \{0,\ldots, n - 1\}\} = 
        \{ k'\in \N \mid \Psi^{(0)}(\tilde{x}', k')\in \tilde{\Y}\times \{0,\ldots, n - 1\}\},
    \]
    so $\tilde{f}(\tilde{x}) = \tilde{f}(\tilde{x}')$.
    By Lemma~\ref{lem:continuity}, there is a continuous map $f\colon \X\LRA \N_+$ satisfying $\tilde{f} = f\circ \pi_\X$.

    For each $x\in \X$, there are exactly $f(x)$ integers $k(x,0),\ldots,k(x,f(x) - 1)\in \N$ such that $\psi(x, k(x,i))\in \Y\times \{0,\ldots,n - 1\}$.
    Arrange the integers in increasing order and define $\phi_{\X_f}\colon\X_f\LRA \psi^{-1}(\Y\times \{0,\ldots,n - 1\})$ by 
    \begin{align}\label{eq:phi_(X_f)}
        \phi_{\X_f}(x, i) = (x, k(x,i)),
    \end{align}
    for $(x,i)\in \X_f$.
    Define $\Phi_{\X_f}\colon\G_{\X_f}\LRA {\G_\X\times \R}|_{\pi_\X^{-1}(\psi^{-1}(\Y\times \{0,\ldots,n - 1\}))}$ by 
    \[
        \Phi_{\X_f}\big( (\tilde{x},i), m, (\tilde{x}',i') \big) 
        = \big( (\tilde{x}, k(\pi_\X(\tilde{x}),i)), k - k', (\tilde{x}, k(\pi_\X(\tilde{x}'),i')) \big)
    \] 
    where $k,k',\in\N$ are such that $\sigma_{\tilde{\X}}^{k}(\tilde{x}) = \sigma_{\tilde{\X}}^{k'}(\tilde{x}')$ and 
    \[
        m = i + \sum_{r = 1}^{k} f(\sigma_\X^r(\pi_X(\tilde{x}))) - i'- \sum_{r = 1}^{k'} f(\sigma_\X^r(\pi_X(\tilde{x}'))).
    \]
    Then $\Phi_{\X_f}$ is an isomorphism such that $\phi_{\X_f}\circ\pi_{\X_f} = \pi_{\X\times \N}\circ \Phi_{\X_f}^{(0)}$.

    We have that $\Phi:=\Phi_{\Y_g}\circ\Psi\circ\Phi_{\X_f}\colon\G_{\X_f}\LRA\G_{\Y_g}$ is an isomorphism and 
    $h := \phi_{\Y_g}\circ \psi\circ \phi_{\X_f} \colon \X_f\LRA\Y_g$ is a homeomorphism such that $h\circ\pi_{\X_f} = \pi_{\Y_g}\circ\Phi^{(0)}$.

    Let $\xi\in C(\Y\times\N,\Z)$ be defined by 
    \[
        \xi(y,i) =
        \begin{cases}
            1 & \text{if}~i > 0,\\
            n & \text{if}~i = 0,
        \end{cases}
    \]
    for $(y,i)\in \Y\times \N$.
    Then $\kappa_{\Y_g}(1)\circ \Phi_{\Y_g} = \kappa_{\Y\times\N}(\xi)$. 
    Set $\eta := \psi^*(\xi)\in C(\X\times\N,\Z)$ and define $d_{\X_f}\in C(\X_f,\Z)$ by 
    \[
        d_{\X_f}(x,i) = 
        \begin{cases}
            \sum_{j = k(x,i-1) + 1}^{k(x,i)} \eta(x,j) & \textrm{if}~i > 0, \\
            \eta(x,0) - \sum_{j = 1}^{k(\sigma_\X(x), f(\sigma_\X(x)) - 1)} \eta(\sigma_\X(x),j) & \textrm{if}~i = 0,
        \end{cases}       
    \]
    for $(x,i)\in \X_f$.    
    Then $\kappa_{\Y\times\N}(\xi)\circ\Psi = \kappa_{\X\times\N}(\eta)$ and $\kappa_{\X\times\N}(\eta)\circ \Phi_{\X_f} = \kappa_{\X_f}(d_{\X_f})$. 
    We thus have $\kappa_{\Y_g}(1)\circ\Phi = \kappa_{\X_f}(d_{\X_f})$.

    Similarly, $\kappa_{\X_f}(1)\circ\Phi_{\X_f}^{-1}=\kappa_{\X\times\N}(\rho)$ where $\rho\in C(\X\times\N,\Z)$ is defined by
    \[
        \rho(x,j) =
        \begin{cases}
            f(\sigma(x))    & \text{if}~j = 0, \\
            1               & \text{if}~j = k(x,i)~\text{for some}~i\in\{1,\dots,f(x)-1\}, \\
            0               & \text{otherwise}.
        \end{cases}
    \]
    Let $\chi = \psi^\#(\rho)\in C(\Y\times\N,\Z)$, and let $d_{\Y_g}\in C(\Y_g,\Z)$ be defined by 
    \[
        d_{\Y_g}(y,i) =
        \begin{cases}
            \chi(y,i)                                                           & \textrm{if}~i > 0, \\
            \chi(y,0) - \sum_{j = 1}^{n - 1} \chi(\sigma_\Y(y),j)   & \textrm{if}~i = 0,
        \end{cases}
    \]
    Then $\kappa_{\X\times\N}(\rho)\circ\Psi^{-1} = \kappa_{\Y\times\N}(\chi)$ and $\kappa_{\Y\times\N}(\chi)\circ \Phi_{\Y_g} = \kappa_{\Y_g}(d_{\Y_g})$. 
    Hence, $\kappa_{\X_f}(1)\circ \Phi^{-1} = \kappa_{\Y_g}(d_{\Y_g})$.

    It now follows from Theorem~\ref{thm:coe} that there are continuous maps $k_{\X_f},l_{\X_f}\colon \X_f\LRA \N$ and $k_{\Y_g}, l_{\Y_g}\colon \Y_g\LRA \N$ 
    such that $(h,k_{\X_f},l_{\X_f},k_{\Y_g},l_{\Y_g})$ is a stabilizer-preserving continuous orbit equivalence from $\X_f$ to $\Y_g$
    and $l_{\X_f} - k_{\X_f} = d_{\X_f}$ and $l_{\Y_g} - k_{\Y_g} = d_{\Y_g}$.

    Note that $[\xi]\in H_+^{\Y\times\N}$ and $[\eta] = [\psi^*(\xi)] = H(\psi)([\xi])\in H_+^{\X\times\N}$. 
    Since $H(\iota_0)\colon H^{\X\times\N}\LRA H^\X$ is a positive isomorphism, it follows that there are continuous maps $\alpha,\beta\colon\X\times\N\LRA\N$ 
    such that $\alpha$ is supported on $\X\times\{0\}$ and $\eta = \alpha + \beta - \beta\circ S_\X$. 
    Then $d_{\X_f}(x,i) = \beta(x,k(x,i))-\beta(x,k(x,i-1))$ for $i > 0$ and
    \begin{align*}
        d_{\X_f}(x,0) 
        &= \alpha(x,0) + \beta(x,0) - \beta(\sigma_\X(x),k(x,f(\sigma_\X(x))) - 1) \\
        &= \alpha(x,0) + \beta(x,0) - \beta\circ \sigma_f(x,0),
    \end{align*}
    for $x\in \X$.
    Thus, $[l_{\X_f} - k_{\X_f}] = [d_{\X_f}]\in H_+^{\X_f}$.

    Similarly, $[\rho]\in H_+^{\X\times\N}$ and $[\chi] = [\psi^\#(\rho)] = H(\psi)^{-1}([\rho])\in H_+^{\Y\times\N}$, 
    so there are continuous maps $\alpha', \beta'\colon \Y\times\N \LRA \N$ such that $\gamma$ is supported on $\Y\times\{0\}$ 
    and $\chi = \alpha' + \beta' - \beta'\circ S_\Y$, and then $d_{\Y_g}(y,i) = \theta(y,i) - \theta(y,i - 1)$ for $i > 0$, 
    and $\tau(y,0) = \alpha(y,0) + \beta'(y,0) - \beta'(\sigma_\Y(y), n - 1)$, for $y\in \Y$.
    This shows that $[l_{\Y_g} - k_{\Y_g}] = [\tau]\in H_+^{\Y_g}$. 

    We conclude that $(h,k_{\X_f},l_{\X_f},k_{\Y_g},l_{\Y_g})$ is a positive stabilizer-preserving continuous orbit equivalence.

    (iv) $\implies$ (i): 
    We have that $\Lambda_{\X_f}$ and $\Lambda_{\Y_g}$ are flow equivalent according to Corollary~\ref{cor:flow}. 
    Since $\Lambda_\X$ and $\Lambda_{\X_f}$ are flow equivalent, and $\Lambda_\Y$ and $\Lambda_{\Y_g}$ are flow equivalent, 
    it follows that $\Lambda_\X$ and $\Lambda_\Y$ are flow equivalent.

    (ii)$\implies$(iii): 
    The isomorphism $\Psi\colon \G_\X\times \R \LRA \G_\Y\times \R$ induces a $^*$-isomorphism 
    $\Phi\colon \OO_\X\otimes \K = \mathrm{C_r^*}(\G_\X\times \R)\LRA \mathrm{C_r^*}(\G_\Y\times \R) = \OO_\Y\otimes \K$ satisfying $\Phi(f) = f\circ \Psi^{-1}$,
    for $f\in C_c(\G_\X\times \R)$.
    In particular, $\Phi(\D_\X\otimes c_0) = \D_\Y\otimes c_0$.
    The hypothesis, $\psi\circ \pi_{\X\times \N} = \pi_{\Y\times \N}\circ \Psi^{(0)}$ ensures that $\Phi(f) = f\circ\psi^{-1}$, 
    for $f\in C(\X)\otimes c_0\subseteq C(\tilde{\X})\otimes c_0 = \D_\X\otimes c_0$, 
    and that $\Phi^{-1}(g) = g\circ\psi$ for $g\in C(\Y)\otimes c_0\subseteq C(\tilde{\Y})\otimes c_0 = \D_\Y\otimes c_0$.  
    Therefore, $\Phi(C(\X)\otimes c_0) = C(\Y)\otimes c_0$.

    Let $\eta\in C(\Y\times\N, \Z)$ and suppose $f\in C_c(\G_\X\times\R)$ has support in $\kappa_{\X\times\N}(\psi^*(\eta))^{-1}(\{1\})$. 
    By~\eqref{eq:flow-psistar}, $\Phi(f) = f\circ\Psi^{-1}$ has support in $\Psi(\kappa_{\X\times\N}(\psi^*(\eta))^{-1}(\{1\})) = \kappa_{\Y\times\N}(\eta)^{-1}(\{1\})$. 
    It follows that 
    \[
        \Phi\circ\beta_z^{\kappa_{\X\times\N}(\psi^*(\eta))} = \beta_z^{\kappa_{\Y\times\N}(\eta)}\circ\Phi,
    \]
    for $z\in\T$. 
    A similar argument using~\eqref{eq:flow-psihash} shows that $\Phi\circ\beta_z^{\kappa_{\X\times\N}(\zeta)} = \beta_z^{\kappa_{\Y\times\N}(\psi^\#(\zeta))}\circ\Phi$,
    for $\zeta\in C(\X\times\N,\Z)$ and $z\in\T$.

    (iii) $\implies$ (ii): 
    By Corollary~\ref{cor:diagonal-preserving-stable}, we have $\Phi(C(\X)\otimes c_0) = C(\Y)\otimes c_0$,
    so it follows from~\cite[Theorem 3.3]{CRST} that there is an isomorphism $\Psi\colon \G_\X\times \R \LRA \G_\Y\times \R$.

    Let $\eta\in C(\Y\times\N,\Z)$. 
    It then follows from~\cite[Theorem 6.2]{CRST} that there is an isomorphism $\Psi_\eta\colon \G_\X\times \R \LRA \G_\Y\times \R$ satisfying 
    $\kappa_{\X\times\N}(\psi^*(\eta)) = \kappa_{\Y\times\N}(\eta)\circ\Psi_\eta$,
    and according to Remark~\ref{remark:CRST}, we have $\Psi = \Psi_\eta$. 
    Therefore, $\kappa_{\X\times\N}(\psi^*(\eta)) = \kappa_{\Y\times\N}(\eta)\circ\Psi$, for every $\eta\in C(\Y\times\N,\Z)$. 
    A similar argument shows that $\kappa_{\Y\times\N}(\psi^\#(\zeta)) = \kappa_{\X\times\N}(\zeta)\circ\Psi^{-1}$, for every $\zeta\in C(\X\times\N,\Z)$.
    Finally, the restriction $\Phi|_{C(\X)\otimes c_0}\colon C(\X)\otimes c_0\LRA C(\Y)\otimes c_0$ induces a homeomorphism $\psi\colon \X\times \N \LRA \Y\times \N$ 
    such that $\psi\circ \pi_{\X\times \N} = \pi_{\Y\times \N}\circ \Psi^{(0)}$.

    The final remark follows from Corollary~\ref{cor:diagonal-preserving-stable}.
\end{proof}

If we restrict to the class of shift spaces which produce effective groupoids, we can relax some of the conditions of Theorem~\ref{thm:flow}.

\begin{theorem}\label{thm:flow-essentially-principal}
    Let $\Lambda_\X$ and $\Lambda_\Y$ be two-sided shift spaces such that $\X$ and $\Y$ contain no periodic points isolated in past equivalence.
    The following are equivalent:
    \begin{enumerate}
        \item[(i)] the systems $\Lambda_\X$ and $\Lambda_\Y$ are flow equivalent;
        \item[(ii)] there is an isomorphism of groupoids $\Psi\colon \G_\X\times \R \LRA \G_\Y\times \R$ and a homeomorphism 
            $\psi\colon \X\times \N \LRA \Y\times \N$ satisfying $\psi\circ \pi_{\X\times \N} = \pi_{\Y\times \N}\circ \Psi^{(0)}$
            and a positive isomorphism $\theta\colon H^{\X\times \N} \LRA H^{\Y\times \N}$ satisfying 
            $\theta\circ \kappa_{\X\times \N} = \kappa_{\Y\times \N}\circ H^1(\Psi)$.
    \end{enumerate}
\end{theorem}

\begin{proof}
    (i)$\implies$(ii):
    This follows from the proof of Theorem~\ref{thm:flow} (i) $\implies$ (ii).

    (ii) $\implies$ (i):
    Let $\Psi\colon \G_\X\times \R \LRA \G_\Y\times \R$ be a groupoid isomorphism and $\psi\colon \X\times \N \LRA \Y\times \N$ be a homeomorphism
    satisfying $h\circ \pi_{\X\times \N} = \pi_{\Y\times \N}\circ \Psi^{(0)}$.
    As in the proof of Theorem~\ref{thm:flow} (ii) $\implies$ (iv) we choose $n\in \N_+$ and $f\in C(\X, \N_+)$.
    Let $g\colon \Y\LRA \N$ be constantly equal to $n$.
    Then there is a groupoid isomorphism $\Psi'\colon \G_{\X_f} \LRA \G_{\Y_g}$ and a homeomorphism $h = \phi_{\Y_g}\circ \psi\circ \phi_{\X_f}$
    such that $h\circ \pi_{\X_f} = \pi_{\Y_g}\circ (\Psi')^{(0)}$.

    It is not hard to see that the maps $\phi_{\X_f}\colon \X_f\LRA \psi^{-1}(\Y\times \{0, \ldots, n - 1\})$ and $\phi_{\Y_g}\colon \Y_g\LRA \Y\times \{0, \ldots, n - 1\}$
    defined in~\eqref{eq:phi_(X_f)} and~\eqref{eq:phi_(Y_g)}, respectively, are positive continuous orbit equivalences.
    Since $\X$ and $\Y$ contain dense sets of aperiodic points,
    it follows from Theorem~\ref{thm:coe-essentially-principal} that 
    $\psi\colon \psi^{-1}(\Y\times \{0, \ldots, n - 1\}) \LRA \Y\times \{0,\ldots, n - 1\}$ is a continuous orbit equivalence.
    By the hypothesis in (ii), $\psi$ is also positive.
    Hence $h$ is a positive continuous orbit equivalence.
    It this follows from Corollary~\ref{cor:flow} that $\Lambda_{\X_f}$ and $\Lambda_{\Y_g}$ are flow equivalent.
    Since $\Lambda_\X$ and $\Lambda_{\X_f}$ are flow equivalent, and $\Lambda_\Y$ and $\Lambda_{\Y_g}$ are flow equivalent,
    we conclude that $\Lambda_\X$ and $\Lambda_\Y$ are flow equivalent.
\end{proof}

Finally, we restrict to the class of sofic shifts whose groupoids are effective.

\begin{theorem}\label{thm:flow-sofic}
    Let $\Lambda_\X$ and $\Lambda_\Y$ be two-sided sofic shift spaces such that $\X$ and $\Y$ contain no periodic points isolated in past equivalence.
    The following are equivalent:
    \begin{enumerate}
        \item[(i)] the two-sided subshifts $\Lambda_\X$ and $\Lambda_\Y$ are flow equivalent;
        \item[(ii)] there is an isomorphism $\Psi\colon \G_\X\times \R \LRA \G_\Y\times \R$ and a homeomorphism 
            $\psi\colon \X\times \N \LRA \Y\times \N$ satisfying $\psi\circ \pi_{\X\times \N} = \pi_{\Y\times \N}\circ \Psi^{(0)}$;
        \item[(iii)] there is a $^*$-isomorphism $\Phi\colon \OO_\X\otimes \K\LRA \OO_\Y\otimes \K$ satisfying $\Phi(C(\X)\otimes c_0) = C(\Y)\otimes c_0$.
    \end{enumerate}
\end{theorem}

\begin{proof}
(i) $\implies$ (ii):
This follows from Theorem~\ref{thm:flow}.

(ii)$\implies$(i): As in the proof of (ii)$\implies$(iv) in Theorem~\ref{thm:flow}, there are $f\in C(\X,\N_+)$, $g\in C(\Y,\N_+)$, 
a groupoid isomorphism $\Psi'\colon \G_{\X_f}\LRA\G_{\Y_g}$ and a homeomorphism $h\colon \X_f\LRA\Y_g$ such that $h\circ \pi_{\X_f} = \pi_{\Y_g}\circ(\Psi')^{(0)}$. 
It follows from Theorem~\ref{thm:flow-essentially-principal} and its proof that $h$ is a continuous orbit equivalence 
and that $(\Psi')^{(0)}\colon\tilde{\X}_f\LRA \tilde{\Y}_g$ is a continuous orbit equivalence. 
Since $\X_f$ and $\Y_g$ are sofic shift spaces, the covers $\tilde{\X}_f$ and $\tilde{\Y}_g$ are (conjugate to) shifts of finite type. 
By hypothesis, $\X$ and $\Y$ have no periodic points isolated in past equivalence, so $\tilde{\X}$ and $\tilde{\Y}$, and thus also $\tilde{\X}_f$ and $\tilde{\Y}_g$, 
have no isolated points. 
It therefore follows the proof of~\cite[Theorem 5.11]{MM-zeta} that the continuous orbit equivalence $(\Psi')^{(0)}$ is positive and least period preserving. 
It follows that $h$ is also positive and least period preserving. 
It therefore follows from Corollary~\ref{cor:flow} that $\X_f$ and $\Y_g$ are flow equivalent. 
Since $\X$ and $\X_f$ are flow equivalent, and $\Y$ and $\Y_g$ are flow equivalent, we conclude that $\X$ and $\Y$ are flow equivalent.

(ii) $\iff$ (iii):
This is~\cite[Corollary 11.4]{CRST}.
Note that if $\Phi\colon \OO_\X\otimes \K \LRA \OO_\Y\times \K$ is a $^*$-isomorphism as in (iii), 
then $\Phi(\D_\X\otimes c_0) = \D_\Y\otimes c_0$ by Corollary~\ref{cor:diagonal-preserving-stable}.
\end{proof}

\begin{corollary}\label{cor:coe-flow-sofic}
    Let $\X$ and $\Y$ be one-sided sofic shifts with no periodic points isolated in past equivalence.
    If $\X$ and $\Y$ are continuously orbit equivalent, then $\Lambda_\X$ and $\Lambda_\Y$ are flow equivalent.
\end{corollary}


\end{document}